\documentclass{article}
\usepackage{graphicx} %
\usepackage{amsmath}
\usepackage{amsthm}
\usepackage{amssymb}   
\usepackage{mathrsfs}  
\usepackage{enumitem}  
\usepackage{geometry} 
\usepackage{fullpage}
\usepackage[colorlinks,
            linkcolor=red,
            anchorcolor=blue,
            citecolor=green
            ]{hyperref}
\usepackage{chngcntr}  
\numberwithin{equation}{section}
            
\usepackage{mathtools}

\usepackage{algorithm}
\usepackage{algorithmic}
\usepackage{natbib}

\usepackage{microtype}
\usepackage{graphicx}
\usepackage{subfigure}
\usepackage{booktabs} %
\usepackage{makecell}

\usepackage{hyperref}

\newcommand{\bx}{\mathbf{x}}
\newcommand{\by}{\mathbf{y}}
\newcommand{\bz}{\mathbf{z}}

\newcommand{\bu}{\mathbf{u}}
\newcommand{\bv}{\mathbf{v}}

\newcommand{\bb}{\mathbf{b}}

\newcommand{\bt}{\mathbf{t}}

\hypersetup{
  linkcolor  = blue,
  citecolor  = teal,
  colorlinks = true,
}
\makeatletter
\newcommand\blfootnote[1]{%
  \begingroup
  \renewcommand{\@makefntext}[1]{\noindent\makebox[1.8em][r]#1}
  \renewcommand\thefootnote{}\footnote{#1}%
  \addtocounter{footnote}{-1}%
  \endgroup
}
\makeatother

\newcommand{\customfootnotetext}[2]{{%
  \renewcommand{\thefootnote}{#1}%
  \footnotetext[0]{#2}}}%

\DeclareMathOperator{\proj}{proj}
\DeclareMathOperator{\prox}{prox}
\DeclareMathOperator*{\argmin}{arg\,min}
\DeclareMathOperator{\Pen}{Pen}

\usepackage[capitalize,noabbrev]{cleveref}

\theoremstyle{plain}
\newtheorem{theorem}{Theorem}[section]
\newtheorem{proposition}[theorem]{Proposition}
\newtheorem{lemma}[theorem]{Lemma}
\newtheorem{corollary}[theorem]{Corollary}
\theoremstyle{definition}

\newtheorem{assumption}[theorem]{Assumption}
\theoremstyle{remark}
\newtheorem{remark}[theorem]{Remark}
\theoremstyle{fact}
\newtheorem{fact}[theorem]{Fact}

\usepackage[textsize=tiny]{todonotes}

\title{Stochastic Smoothed Primal-Dual Algorithms for \\
Nonconvex Optimization with Linear Inequality Constraints}
\author{Ruichuan Huang\footnote{Department of Mathematics, University of British Columbia. \url{hrc22@student.ubc.ca}} \and Jiawei Zhang\footnote{Laboratory of Information and Decision Systems, Massachusetts Institute of Technology. \url{jwzhang@mit.edu}} $^\P$\and Ahmet Alacaoglu\footnote{Department of Mathematics, University of British Columbia. \url{alacaoglu@math.ubc.ca}} $^\P$}
\date{}
\begin{document}
\maketitle

\begin{abstract}
We propose smoothed primal-dual algorithms for solving stochastic and smooth nonconvex optimization problems with linear inequality constraints. Our algorithms are single-loop and only require a single stochastic gradient based on one sample at each iteration. A distinguishing feature of our algorithm is that it is based on an inexact gradient descent framework for the Moreau envelope, where the gradient of the Moreau envelope is estimated using one step of a stochastic primal-dual augmented Lagrangian method. To handle inequality constraints and stochasticity, we combine the recently established global error bounds in constrained optimization with a Moreau envelope-based analysis of stochastic proximal algorithms. For obtaining $\varepsilon$-stationary points, we establish the optimal $O(\varepsilon^{-4})$ sample complexity guarantee for our algorithms and provide extensions to stochastic linear constraints. We also show how to improve this complexity to $O(\varepsilon^{-3})$ by using variance reduction and the expected smoothness assumption. Unlike existing methods, the iterations of our algorithms are free of subproblems, large batch sizes or increasing penalty parameters and use dual variable updates to ensure feasibility.
\end{abstract}
\section{Introduction}
\customfootnotetext{$\P$}{Co-last authors and corresponding authors.}
\label{intro}
We focus on the problem template
\begin{equation}\label{eq: prob1}
\min_{\bx\in X} f(\bx) \text{~subject to~} A\bx = \bb,
\end{equation}
where $f\colon\mathbb{R}^n \to \mathbb{R}$ is $L_f$-smooth, the set $X\subseteq \mathbb{R}^n$ is polyhedral, and easy to project. In particular, let $X$ be given as $X=\{ \bx\colon H\bx \leq h \}$ for some matrix $H$ and vector $h$. Taking $H=I$, for example, gives this template the ability to model \emph{linear inequality} constraints. 
In particular, when we have the problem
\begin{equation}\label{eq: prob2}
\min_{\bx} f(\bx) \text{~subject to~} A\bx \leq \bb,
\end{equation}
we introduce a slack variable $\bt = A\bx-\bb$ so that $A\bx-\bt=\bb$ and our optimization variable becomes $\binom{\bx}{\bt}$ and we can equivalently write the problem in the template \eqref{eq: prob1} by using the constraint $\bt \leq 0$. As such, we focus on \eqref{eq: prob1} and our results directly apply to solving \eqref{eq: prob2} by using this standard slack variable reformulation.

Throughout, we assume that we have access to an unbiased oracle $F(\bx)$ such that 
\begin{equation}\label{eq: stoc_oracle}
\mathbb{E}[F(\bx)] = \nabla f(\bx), \text{~and~} \mathbb{E}\|F(\bx) - \nabla f(\bx) \|^2 \leq \sigma^2.
\end{equation}
A common setting is when $f(\bx) = \mathbb{E}_{\xi \sim \Xi} [f(\bx, \xi)]$ where $\Xi$ is an unknown distribution that we can draw i.i.d. samples from. In this case, it is common to set $F(\bx) = \nabla f(\bx, \xi)$ and assume that $\mathbb{E}[\nabla f(\bx, \xi)] = \nabla f(\bx)$.

Inclusion of $X$ in the template \eqref{eq: prob1} increases the modeling power significantly, while causing difficulties in the analysis. Many problems fit this template, including constrained and distributed optimization, nonnegative matrix factorization, sparse subspace estimation and collaborative learning, see for example \cite{zhang2022decentralized, hong2016decomposing} and also \Cref{sec: apps}. Moreover, reformulations of nonconvex minimization problems are also common by using linear inequality constraints \citep{zhang2022decentralized}.

Algorithm development for \eqref{eq: prob1} and related templates have been active in the last couple of years \citep{alacaoglu2024complexity, zhang2020proximal, zhang2020single, lu2024variance, li2021rate, lin2022complexity, yan2022adaptive, li2024stochastic, boob2023stochastic, hong2016decomposing}, mainly due to the applications of functionally constrained nonconvex optimization problems in the context of neural network training \citep{katz2022training, dener2020training}. Stochastic augmented Lagrangian methods (ALM) have found widespread use in practice with problems involving nonconvex functional constraints \citep{katz2022training, dener2020training}, whereas their behavior for even linearly constrained nonconvex optimization of the form \eqref{eq: prob1} remain poorly understood. The focus of this work is to improve our understanding of stochastic ALM in the context of nonconvex optimization, by focusing on the fundamental template \eqref{eq: prob1}.

Compared to the setting of convex $f$, where the global complexity analysis is mostly settled for ALM and its stochastic version \citep{yan2022adaptive}, nonconvexity of $f$ poses significant difficulties in the analysis of ALM. Many works in the literature focus on penalty based algorithms (which will be formally introduced later in this section) that do not perform dual updates (or perform negligible dual updates that we clarify later) \citep{lu2024variance, li2021rate, lin2022complexity}, rather than primal-dual algorithms such as ALM. However, in practice, dual updates are known to be essential for accelerating convergence. Penalty methods are known to be unstable since increasing penalty parameter causes Lipschitz constant of the subproblems to increase and can lead to numerical issues. These differences in behavior between penalty and augmented Lagrangian methods are well-known, see for example classical books, such as \citep{bertsekas2014constrained, bertsekas2016nonlinear, nocedal1999numerical}.

For problem \eqref{eq: prob1} with access to full gradients of $f$ and the full matrix $A$, the optimal complexity with primal-dual methods are obtained in the work of \cite{zhang2022error}. When one has access to stochastic gradients of $f$ and the matrix $A$, a recent work by \cite{alacaoglu2024complexity} showed optimal complexity guarantees under expected smoothness (see Assumption \ref{assumption: variance_reduction}), for the special case of \eqref{eq: prob1} when $X=\mathbb{R}^n$, where this latter restriction significantly reduces the generality of the template. For example, modeling the standard quadratic programming problem requires $X$ to be a half-space, which was not supported in the analysis of \cite{alacaoglu2024complexity}. Our goal is to go beyond these results by handling both the case when $X\neq \mathbb{R}^n$ as well as the case when we do not have access to the matrix $A$ but only to an unbiased estimate of $A$, by keeping optimal complexity guarantees. A more detailed comparison of complexity guarantees will be made in Section \ref{sec: relwork} and a summary is provided in Table \ref{tab:1}.

\paragraph{Lagrangian, penalty and augmented Lagrangian functions}
The standard approach to tackle \eqref{eq: prob1} is to design algorithms operating on the Lagrangian, augmented Lagrangian or penalty functions. In particular, the Lagrangian function is given as
\begin{equation*}
L(\bx, \by) = f(\bx) + \langle A\bx-\bb, \by \rangle,
\end{equation*}
with the dual variable $\by$, whereas the penalty function has the form of
\begin{equation*}
\Pen_\rho(\bx) = f(\bx) + \frac{\rho}{2}\| A\bx-\bb\|^2.
\end{equation*}
It is common for algorithms based on the penalty function to require $\rho \to \infty$ for convergence \citep{bertsekas2014constrained}. One major disadvantage of this strategy is that $\rho$ getting larger makes the subproblem of minimizing the penalty function more and more ill-conditioned.

An influential idea was the introduction of the augmented Lagrangian (AL) function which combined the idea of the Lagrangian and penalty formulations \citep{hestenes1969multiplier}. In particular, the AL function is defined as
\begin{equation*}
L_\rho(\bx, \by) = f(\bx) + \langle A\bx-\bb, \by\rangle + \frac{\rho}{2} \| A\bx-\bb\|^2.
\end{equation*}
Augmented Lagrangian methods in the classical literature were favored because these methods worked without requiring $\rho$ to grow arbitrarily large. In fact, many instances of ALM converge to the optimal solution with fixed $\rho$ since the inclusion of the dual variable $y$ aids in satifying feasibility \citep{bertsekas2014constrained}.

\paragraph{Primal vs primal-dual algorithms. } The algorithms based on the penalty function are generally referred to as \emph{penalty algorithms} and are easier to analyze in different settings since they are primal-only algorithms, meaning that they only perform updates on primal variable $\bx$ where approximate feasibility is ensured by $\rho\to\infty$. In particular, a classical penalty method iterates for $k=1, 2, \dots$ as
\begin{align*}
&\bx_{k+1} \approx \arg\min_{\bx \in X} f(\bx) + \frac{\rho_k}{2} \| A\bx -\bb\|^2, \\
&\text{Select~} \rho_{k+1} > \rho_k.
\end{align*} 
The algorithms based on the augmented Lagrangian are generally more difficult to analyze due to the additional dynamics coming from the dual updates where the dual updates are critical to ensure that the approximate feasibility is attained with constant $\rho$. An ALM iteration proceeds for $k=1, 2, \dots$ by updating
\begin{align*}
\bx_{k+1} &\approx \arg\min_{\bx \in X} f(\bx) + \langle \by_k, A\bx - \bb \rangle + \frac{\rho}{2} \| A\bx -\bb\|^2, \\
\by_{k+1} &= \by_k + \sigma (A\bx_{k+1} - \bb).
\end{align*} 
For both penalty methods and ALM, different strategies exist to generate $\bx_{k+1}$ that approximately minimize the penalty or augmented Lagrangian functions by either iterating multiple steps of gradient descent (GD), known as \emph{inexact} algorithms, or applying one step of GD, known as \emph{linearized} algorithms\citep{ouyang2015accelerated}.

In view of the earlier discussion, when $f$ is nonconvex, most of the literature focuses on either analyzing penalty methods, or analyzing ALM with \emph{negligible} dual updates and increasing penalty parameters $\rho$, due to the inherent difficulty in analyzing the dual variable and its effect in convergence. In particular, as also highlighted in \cite{alacaoglu2024complexity}, many of the recent analysis of ALM is of the form of a \emph{perturbed penalty analysis}, meaning that the feasibility is driven by increasing penalty parameters, and the dual updates are designed so that they do not deteriorate the estimates too much. Because of this, the dual step sizes are selected to be small to ensure boundedness of the dual variable (or controlling the growth of the dual variable). We refer to such updates as \emph{negligible} dual updates since the analyses do not harness the benefit of such updates in ensuring feasibility. Feasibility is driven by large penalty parameters. Some representative examples are \cite{lu2024variance}, \cite{li2021rate}, \cite{lin2022complexity}, \cite{li2024stochastic}.

This is the case even in the deterministic setting and the only method that we are aware that can handle true ALM with fixed penalty parameters and non-negligible dual updates are due to \cite{zhang2022error} that uses a linearized \emph{proximal} AL function with a dynamic adjustment on the proximal center, which will be clarified in \Cref{sec: alg} since it will form the basis of our algorithmic development.

\subsection{Contributions}
In this paper, we propose a stochastic smoothed linearized augmented Lagrangian algorithm for solving \eqref{eq: prob1} that only uses a single sample of stochastic gradient at every iteration. This algorithm also works with a constant penalty parameter and incorporates non-negligible dual updates for feasibility where the dual step sizes have the same order as the primal step sizes. We show that this method has its iteration complexity and sample complexity guarantees in the order of ${O}(\varepsilon^{-4})$. Such a sample complexity result is optimal even in the unconstrained nonconvex case under our assumptions (see Assumption \ref{asmp: 1}) \citep{arjevani2023lower}. In contrast, the prior results with optimal complexity required large penalty parameters, no dual updates and further assumptions \citep{lu2024variance}. 
We then prove that this complexity can be improved to ${O}(\varepsilon^{-3})$ with variance reduction when an additional expected smoothness assumption is made (see Assumption \ref{assumption: variance_reduction}). Under this stronger assumption, this is the optimal complexity even without constraints \citep{arjevani2023lower}.

We consider extensions of this framework when we have linear constraints that hold in expectation, that is, when the constraints are given as $\mathbb{E}_\xi[A_\xi\bx - \bb_\xi] = 0$.
Our algorithm can also handle this stochastic constrained case with the same complexity guarantees.
To our knowledge, this is the first algorithm achieving the optimal ${O}(\varepsilon^{-4})$ benchmark sample complexity for nonconvex optimization with stochastic constraints using one sample per iteration, going beyond the best-known $O(\varepsilon^{-5})$ complexity that is achieved for a more general problem that does not capture the structure of linear constraints \citep{li2024stochastic,alacaoglu2024complexity}.

A more detailed comparison with the related works is given in \Cref{sec: relwork}. A summary is given in Table \ref{tab:1}.

\begin{table}[h]
    \centering
    \begin{tabular}{|l|l|c|c|c|l|}
    \hline
    \textbf{Reference} & \textbf{Constraint} & \textbf{Oracle} & \textbf{Complexity} & \textbf{Loops} & \textbf{Method} \\
    \hline
    {\cite{alacaoglu2024complexity}} & $A \bx = \bb$ & \makecell[l]{Eq. \eqref{eq: stoc_oracle} and \\ Asmp. \ref{assumption: variance_reduction}} & $\widetilde{O}(\varepsilon^{-3})$ & 1 & ALM \\
    \hline
    {\cite{alacaoglu2024complexity}} & \makecell[l]{$\mathbb{E}[c(\bx,\zeta)] = 0$,\\ and $\bx \in X$ where \\ $X$ is easy to project} & \makecell[l]{Eq. \eqref{eq: stoc_oracle} and \\ Asmp. \ref{assumption: variance_reduction}} & $\widetilde{O}(\varepsilon^{-5})$ & 1 & Penalty \\
    \hline
    {\cite{lu2024variance}} & \makecell[l]{$c(\bx) = 0$,\\ and $\bx \in X$ where \\ $X$ is easy to project} & \makecell[l]{Eq. \eqref{eq: stoc_oracle} and \\ Asmp. \ref{assumption: variance_reduction}} & $O(\varepsilon^{-3})$ & 1 & Penalty \\
    \hline
    {\cite{li2024stochastic}} & \makecell[l]{$\mathbb{E}[c(\bx,\zeta)] = 0$,\\ and $\bx \in X$ where \\ $X$ is easy to project} & \makecell[l]{Eq. \eqref{eq: stoc_oracle} and \\ Asmp. \ref{assumption: variance_reduction}} & $O(\varepsilon^{-5})$ & 2 & Penalty$^\ast$ \\
    \hline
    This work & \makecell[l]{$A\bx = \bb$,\\ and $\bx \in X$ is a polyhedral} & Eq. \eqref{eq: stoc_oracle}  & $O(\varepsilon^{-4})$ & 1 & ALM \\
    \hline
    This work & \makecell[l]{$\mathbb{E}_\zeta[A(\zeta)\bx - \bb(\zeta)] = 0$,\\ and $\bx \in X$ is a polyhedral} & Eq. \eqref{eq: stoc_oracle} & $O(\varepsilon^{-4})$ & 1 & ALM  \\  \hline
    This work & \makecell[l]{$A\bx=\bb$,\\ and $\bx \in X $ is a polyhedral} & \makecell[l]{Eq. \eqref{eq: stoc_oracle} and \\ Asmp. \ref{assumption: variance_reduction}} & $O(\varepsilon^{-3})$ & 1 & ALM \\
    \hline
    \end{tabular}
    \caption{Comparison of methods. $^\ast$This method is referred to as a penalty method because the penalty parameter is taken to infinity to ensure feasibility and dual updates do not contribute in achieving feasibility.}
    \label{tab:1}
\end{table}

\subsection{Preliminaries}\label{sec: prelim}
We denote the indicator function of a  set $X$ as
\begin{equation*}
I_X(\bx) = \begin{cases} 0 &\text{~if~} \bx \in X, \\
+\infty &\text{~if~} \bx \not\in X. \end{cases}
\end{equation*}
The notation $\partial f$ for a convex, closed function denotes the subdifferential set and $\partial I_X(\bx)$ is the normal cone of $X$ at $\bx$ by definition. For a matrix $A$, we use $\| A\|$ to denote its operator norm.

Given closed and convex $X$, we denote the projection onto this set as
\begin{equation*}
\proj_X(\bx) = \arg\min_{\bv\in X} \| \bx-\bv\|^2.
\end{equation*}
Similarly, we define the proximal operator of $f$ as
\begin{equation*}
\prox_{f}(\bx) = \argmin_{\bv} f(\bv) + \frac{1}{2}\|\bv-\bx\|^2.
\end{equation*}
We say that $f$ is $L$-smooth when the gradient of $f$ is $L$-Lipschitz:
\begin{equation*}
\|\nabla f(\bx) - \nabla f(\by) \| \leq L \|\bx-\by\|.
\end{equation*}
We say that $f$ is $\rho$-weakly convex when $f+\frac{\rho}{2} \| \cdot\|^2$ is convex. An $L$-smooth function is automatically $L$-weakly convex.
The Moreau envelope of a weakly convex function $f$ is defined as
\begin{equation*}
\varphi_{\lambda}(\bz) = \min_{\bv} f(\bv) + \frac{1}{2\lambda} \| \bv-\bz\|^2,
\end{equation*}
which can be interpreted as a notion of \emph{smoothing}. Moreau envelope has many useful properties such as being smooth when $f$ is nonsmooth and weakly convex, and when $\lambda$ is selected accordingly. 
Moreover, the stationary points of $f$ and the Moreau envelope coincide \citep[Lemma 4.3]{drusvyatskiy2019efficiency}.

The gradient of the Moreau envelope is computed as
\begin{equation*}
\lambda^{-1}(\bx - \prox_{\lambda \varphi}(\bx)).
\end{equation*}
\paragraph{Stationary points. }
A succinct and standard way of characterizing a stationary point of \eqref{eq: prob1} is the following: we call $\bx^\star$ to be a stationary point if there exists $\by^\star$ such that the following requirements hold:
\begin{align*}
0 &\in \nabla f(\bx^\star) + A^\top \by^\star + \partial I_X(\bx^\star), \\
0 &= A\bx^\star - \bb.
\end{align*}
One may, for example, refer to \cite{rockafellar2000extended}.

Accordingly, we say that $(\bx, \by)$ is an \emph{$\varepsilon$-stationary point} if we have 
\begin{align*}
\|A\bx-\bb\|&\leq\varepsilon \text{~and~} \\
\| \bv\|&\leq \varepsilon \text{~where~} \bv \in \nabla f(x) + A^\top \by + \partial I_X(\bx),
\end{align*}
which is a common notion used in related works, for example \cite{zhang2022error}.

We also use the following related and weaker notion of \emph{near-stationarity}, as used in~\cite{davis2019stochastic}. We say that $\bx$ is \emph{$\varepsilon$-near stationary}, if it satisfies
\begin{equation}\label{eq: soas4}
\|\nabla \Psi(\bx) \|\leq\varepsilon,
\end{equation}
where $\Psi(\bx)$ is the Moreau envelope of the objective function $f(\bx)+I_X(\bx) + I_{\{\bv: A\bv=\bb\}}(\bx)$ in \eqref{eq: prob1}. We refer to \cite{davis2019stochastic} for the precise notion of near stationarity.
\subsection{Assumptions}
We proceed to state the assumptions that will be used throughout. These assumptions are standard and to our knowledge, the weakest, in the literature for both deterministic and stochastic nonconvex problems with linear constraints \citep{zhang2022error, alacaoglu2024complexity}. A more detailed comparison of assumptions will be made in Section \ref{sec: relwork}.
\begin{assumption}\label{asmp: 1}
For the problem \eqref{eq: prob1}, the following holds:
\begin{enumerate}
\item The function $f$ is $L_f$-smooth and is lower bounded over the feasible set, that is,
\begin{equation*}
f(\bx) \geq \underline f > -\infty,
\end{equation*}
for any $\bx\in X$ and $A\bx=\bb$.
\item The set $X$ admits an efficient projection and is polyhedral. That is, it has the form $X=\{\bx\colon H\bx\leq h \}$ for some $H, h$.
\item We have access to a stochastic gradient $F$ of $f$ that satisfies \eqref{eq: stoc_oracle}.
\end{enumerate}
\end{assumption}
\section{Algorithm}\label{sec: alg}
We introduce Algorithm \ref{alg: sgd} in this section. To gain a deeper understanding of the algorithm, we will go over two different ways of interpreting it. 
\paragraph{Interpretation 1: Linearized proximal ALM. } Algorithm \ref{alg: sgd} incorporates a single-step stochastic gradient descent approximation of the proximal augmented Lagrangian function. This strategy is also known as the linearized proximal ALM. In particular, the first step of the algorithm approximates the proximal AL function, that is,
\begin{equation*}
\bx_{t+1} \approx \arg\min_{\bx \in X} L_{\rho}(\bx, \by_{t+1}) + \frac{\lambda}{2} \| \bx-\bz_t\|^2,
\end{equation*}
by a single step of projected SGD, followed by a dual variable update and updating the proximal center $\bz_t$, by a combination of $\bz_{t}$ and $\bx_{t}$, resulting in the terminology \emph{smoothed} that we use for the algorithm.

\paragraph{Interpretation 2: Inexact GD on the Moreau envelope. } Algorithm \ref{alg: sgd}  can also be interpreted as an inexact gradient descent step on the Moreau envelope of the function in \eqref{eq: prob1}. In particular, the Moreau envelope of \eqref{eq: prob1} is given as
\begin{equation}\label{eq: moreau_orig}
\Psi(\bz_t) = \min_{\bx\in X, A\bx=\bb} \left\{f(\bx) + \frac{\lambda}{2} \| \bx-\bz_t\|^2\right\}.
\end{equation}
By observing that minimizing the Moreau envelope helps in obtaining a near-stationary point in view of \eqref{eq: soas4} (cf. \cite{davis2019stochastic}), inexact gradient update on this function requires the computation of
\begin{equation*}
\argmin_{\bx\in X, A\bx=\bb} \left\{f(\bx) + \frac{\lambda}{2} \| \bx-\bz_t\|^2\right\},
\end{equation*}
which is a nontrivial optimization subproblem. However, it is easier than \eqref{eq: prob1} because the regularization (given that $\lambda$ is larger than $L_f$) provides us a \emph{strongly convex objective} in the subproblem. As a result, we can approximate the solution of this problem by applying one iteration of ALM since this problem is a strongly convex optimization problem over linear constraints. We show that just one step of stochastic ALM is sufficient at every iteration by using a stochastic gradient computed with a single sample and one dual update, followed by the update of the proximal center $\bz_t$.

On the surface, this algorithm strongly resembles the algorithm of \cite{zhang2022error} where we draw many ideas. However, in addition to using stochastic gradients, there is another subtle change, on the update of $\bz_{t+1}$. Unlike \cite{zhang2022error}, we update $\bz_{t+1}$ by using $\bx_t$ to be able to continue the analysis with the bounded variance assumption on $G$ instead of boundedness assumption on $G$, since the latter would require bounded domains. Thanks to this small change, in this section, we can handle the case where both primal and dual domains are unbounded.
\begin{algorithm*}[ht!]
   \caption{Stochastic smoothed and linearized ALM}
   \label{alg:example}
\begin{algorithmic}
   \STATE {\bfseries Initialize:} $\bx_0=\bz_0 \in X$, $\by_0\in\mathbb{R}^m$ and $\rho \geq 0$.
   \FOR{$t=0$ {\bfseries to} $T-1$}
   \STATE $\by_{t+1} = \by_t + \eta(A\bx_t - \bb)$
   \STATE Sample $\xi_t \in \Xi$ i.i.d. and generate $\mathbb{E}_{\xi_t}[G(\bx_t, \by_{t+1}, \bz_t, \xi_t)] = \nabla_\bx L_{\rho}(\bx_t, \by_{t+1}) + \mu(\bx_t-\bz_t)$
   \STATE $\bx_{t+1} = \proj_X(\bx_t - \tau G(\bx_t, \by_{t+1}, \bz_t, \xi_t))$   
   \STATE $\bz_{t+1} = \bz_t + \beta(\bx_t-\bz_t)$
   \ENDFOR
\end{algorithmic}
\label{alg: sgd}
\end{algorithm*}

\section{Convergence Analysis}\label{sec: main_conv}
In this section, we first provide the main complexity results, then introduce the main analysis tools and a proof sketch.
\subsection{Main Theorem}
In view of the two stationary notions given in Section \ref{sec: prelim}, we start with the result showing that Algorithm \ref{alg: sgd} outputs a point at which the norm of the gradient of Moreau envelope is small, in expectation.

For the result, we state the algorithmic parameters. To avoid clutter, we write the orders of the parameters by highlighting their dependencies on the problem parameters. The explicit forms of the parameters with all the constants are given in \eqref{eq: param_defs_appendix} in Appendix \ref{app: a}.
\begin{equation}\label{eq: alg_params}
\begin{aligned}
\tau &\asymp \frac{1}{\sqrt{T}},\quad \eta \asymp \frac{1}{\sqrt{T}},\quad \beta \asymp \frac{1}{\sqrt{T}},\\
 \mu &\asymp L_f,\quad \lambda \asymp L_f + \mu(\|A\|^2+1).
\end{aligned}
\end{equation}
We are now ready to state the first main result.
\begin{theorem}\label{th: moreau_env_det}
Let \Cref{asmp: 1} hold and run Algorithm \ref{alg: sgd} with the parameters given in \eqref{eq: alg_params} (see also \eqref{eq: param_defs_appendix}). We have that $\mathbb{E}\| \nabla \Psi(\bz_{t^*})\|\leq\varepsilon$ where $t^*$ is selected uniformly at random from $\{1, \dots, T\}$ with $T=\Theta(\varepsilon^{-4})$. The stochastic oracle complexity is $O(\varepsilon^{-4})$.
\end{theorem}
In particular, the above result gives us an $\varepsilon$-near stationary point in view of \cite{davis2019stochastic}.

Next, to get an $\varepsilon$-stationary point, we perform a post-processing procedure to obtain the following output from the result of Algorithm \ref{alg: sgd}:
\begin{equation}
\hat{\bx} = \proj_X(\bx_{t^*} - \tau \hat{G}(\bx_{t^*}, \by_{t^*+1}, \bz_{t^*})),
\end{equation}
with $\tau \leq \frac{1}{L_K}$ where $L_K$ is the Lipschitz constant of $L_{\rho}(\cdot, \by, \bz) + \frac{\lambda}{2} \| \cdot - \bx\|^2$ (cf. \eqref{eq: param_defs_appendix}) and
\begin{equation*}
\hat{G}(\bx_{t^*}, \by_{t^*+1}, \bz_{t^*}) = \frac{1}{B}\sum_{i=1}^B G(\bx_{t^*}, \by_{t^*+1}, \bz_{t^*}, \xi_i)
\end{equation*}
for $\xi_i$ i.i.d. and $B=\Omega(\varepsilon^{-2})$. We note that this is the only place where we use a large batch size and our algorithm only runs with a single sample at every iteration. This post processing step is only done once and does not affect the overall complexity. The details are given in \Cref{subsec: postprocess}.
\begin{corollary}\label{cor: postprocess}
Let Assumption \ref{alg: sgd} hold. From the output of Algorithm \ref{alg: sgd}, we can obtain an output $\hat \bx$ which is an $\varepsilon$-stationary point. The complexity of the whole procedure is $O(\varepsilon^{-4})$.
\end{corollary}

\subsection{Analysis Tools}
In our analysis, Moreau envelopes of two functions are critical. The first was the Moreau envelope of the composite objective in \eqref{eq: prob1}, defined in \eqref{eq: moreau_orig}. We next define the Moreau envelope on the proximal AL function which is the main function to analyze projected SGD update (cf. \cite{davis2019stochastic}):
\begin{align}
\varphi_{1/\lambda}(\bx, \by, \bz) = \min_{\bu\in X} \big\{L_{\rho}(\bu, \by) &+ \frac{\mu}{2} \| \bu - \bz\|^2 + \frac{\lambda}{2} \| \bu - \bx \|^2\big\}.\label{eq: varphi_def}
\end{align}
Another important quantity that has a significant role in the analysis is the proximal point with respect to the last definition.
\begin{equation}\label{eq: aeh4}
\begin{aligned}
\bu^*(\bx, \by, \bz) = \argmin_{\bu\in X} L_{\rho}(\bu, \by) &+ \frac{\mu}{2} \| \bu - \bz\|^2 + \frac{\lambda}{2} \| \bu - \bx \|^2.
\end{aligned}
\end{equation}
With this, we have
\begin{align*}
\varphi_{1/\lambda}&(\bx, \by, \bz)=L_\rho(\bu^*(\bx, \by, \bz), \by)+ \frac{\mu}{2} \|\bu^*(\bx, \by, \bz) - \bz\|^2 + \frac{\lambda}{2} \| \bu^*(\bx, \by, \bz) - \bx \|^2.
\end{align*}
Note that this is the main point of departure from \cite{zhang2022error} where the proximal AL function is used in the analysis and the potential function. This is because \cite{zhang2022error} used a projected \emph{full} GD step on the proximal AL function for which, a descent inequality follows directly. In our case, because we apply a projected SGD step, to be able to handle updates with single-sample stochastic gradients, we need to use the Moreau envelope of the proximal AL function in our potential. This analysis of projected SGD was pioneered in \cite{davis2019stochastic}.

The first result is a descent-type result on the Moreau envelope.
\begin{lemma}(cf. \Cref{lem: swo3})
Let Assumption \ref{asmp: 1} hold and set $\lambda = L_K, \tau \leq \frac{1}{6\lambda}$. Then for the $\bx_{t+1}$ update given in Algorithm \ref{alg: sgd}, we have
\begin{align*}
\mathbb{E} \left[ \varphi_{1/\lambda}(\bx_{t+1}, \by_{t+1}, \bz_{t+1}) \right] &\leq \mathbb{E} \left[ \varphi_{1/\lambda}(\bx_{t}, \by_{t+1}, \bz_{t+1}) \right] - \frac{\tau \lambda^2}{16} \mathbb{E}\| \bu^*(\bx_t, \by_{t+1}, \bz_t) - \bx_t\|^2 + \lambda \tau^2\sigma^2 \\
&\quad+\left( \lambda \tau \mu + 2\lambda \tau^2 \mu^2 + \tau \lambda^2 \mu^2/8\gamma_s^2 \right) \mathbb{E}\| \bz_t - \bz_{t+1}\|^2,
\end{align*}
where $\gamma_s=2\mu + \rho\|A\|, L_K = L_f + \rho \|A\| + \mu$.
\end{lemma}
This follows mostly from \cite{davis2019stochastic} and handles the transition from $\bx_t$ to $\bx_{t+1}$ in our analysis. One additional error term we have here is $\|\bz_{t+1} - \bz_t\|^2$, due to the change in the proximal center $\bz_t$, a term that was not involved in the analysis of \cite{davis2019stochastic}.

Next, we have to incorporate the dynamics of the updates on the dual variable $\by_t$ and the proximal center $\bz_t$ on top of the previous result. These results will use some ideas from \cite{zhang2022error} with some additional insights. The reason is that since \cite{zhang2022error}  uses the function $L_\rho(\bx, \by) + \frac{\lambda}{2} \| \bx-\bz\|^2$ in their potential function, their analysis only characterizes the change in $\by$ and $\bz$ in this function. Our analysis however, needs to characterize this change in the Moreau envelope of this function. This requires further estimations using the properties of the Moreau envelope, as well as the proximal point $\bu^*(\bx, \by, \bz)$ (see, for example, \Cref{lem: phi_misc} and Appendix \ref{app: a}).
\begin{lemma}(cf. \Cref{lem: phi_misc})
Let Assumption \ref{asmp: 1} hold, then for the iterates generated by Algorithm \ref{alg: sgd}, we have
\begin{align*}
\mathbb{E} \left[ \varphi_{1/\lambda}(\bx_{t}, \by_{t+1}, \bz_{t+1}) \right] &\leq \mathbb{E} \left[ \varphi_{1/\lambda}(\bx_{t}, \by_{t}, \bz_{t}) \right] -\mathbb{E}\langle \by_{t+1} -\by_t, A\bu^*(\bx_t, \by_t, \bz_t) - \bb \rangle \\
&\quad- \frac{\mu}{2} \mathbb{E}\langle \bz_{t}- \bz_{t+1}, 2\bu^*(\bx_t, \by_{t+1}, \bz_t) - \bz_{t+1} - \bz_t \rangle.
\end{align*}
\end{lemma}
It is easy to notice that combining the last two lemmas will give us a bound on the change of $\varphi_{1/\lambda}$ between timesteps $t$ and $t+1$. On the other hand, the inner products appearing on the right-hand side of the last bound will require an intricate analysis after combining with the terms coming from other components in the potential function, introduced next. One aim, is to make sure we get enough slack to be able to cancel error terms coming from $\|\bz_{t+1} - \bz_t\|^2$ in the previous lemma and further errors that will arise as we handle the inner products. 
 
\subsection{Proof Sketch}
\subsubsection{One iteration inequality on the potential function}
As alluded to earlier, we introduce the potential function we work with, which incorporates the Moreau envelopes defined earlier in \eqref{eq: moreau_orig} and \eqref{eq: varphi_def}:
\begin{equation*}
V_t = \varphi_{1/\lambda}(\bx_t, \by_t, \bz_t) - 2d(\by_t, \bz_t) + 2\Psi(\bz_t),
\end{equation*}
where we use
\begin{equation}\label{eq: def_d}
d(\by, \bz) = \min_{\bx\in X} L_\rho(\bx, \by) + \frac{\mu}{2} \| \bx-\bz\|^2.
\end{equation}
There are two main changes compared to the analysis of \cite{zhang2022error}. The first is that the \emph{primal descent} portion of our analysis investigates the behavior of the Moreau envelope of the proximal AL function (given in \eqref{eq: varphi_def}) whereas the analysis of \cite{zhang2022error} analyzes the proximal AL function (given in \eqref{eq: k_def_main}) directly. 

The reason for this departure is the well-known difficulty while analyzing SGD for constrained problems with single sample of stochastic gradients. Hence, it is not clear if it is possible to show descent for the proximal AL function in the constrained case without using large minibatch sizes. In particular, until the work of \cite{davis2019stochastic}, convergence analyses of projected SGD required large batches.

In addition to combining the bounds from the previous section on the change of $\varphi_{1/\lambda}$, we have to characterize the change in $d(\by, \bz)$ and $\Psi(\bz)$, for which we can use the following estimations, which only use the definition of $\by_{t+1}$ and hence have the same proof as the previous work.
\begin{lemma}{\citep[Lemma 3.2, Lemma 3.3]{zhang2020proximal}}\label{lem: swo4}
For the functions $d(\by, \bz)$ and $\Psi(\bz)$ defined in \eqref{eq: moreau_orig} and \eqref{eq: def_d}, we have
    \begin{align*} 
        d(\by_{t+1}, \bz_{t+1})-d(\by_t, \bz_t) \geq \eta \langle A\bx_{t}-\bb, A\bx^*(\by_{t+1}, \bz_t)-\bb \rangle+\frac{\mu}{2}\langle \bz_{t+1}-\bz_{t}, \bz_{t+1}+\bz_{t}-2\bx^*(\by_{t+1},\bz_{t+1}) \rangle,
    \end{align*}
    and
    \begin{align*}
    \Psi(\bz_{t+1})-\Psi(\bz_t)&\leq \mu\langle \bz_{t+1}-\bz_t, \bz_t-\bar{\bx}^*(\bz_t)\rangle+\frac{\mu}{2\sigma_4}\|\bz_t-\bz_{t+1}\|^2,
    \end{align*}
where $\sigma_4=\frac{\mu-L_f}{\mu}$ and
\begin{align}
\bx^*(\by, \bz) &=\argmin_{\bx\in X} L_\rho(\bx, \by) + \frac{\mu}{2} \| \bx-\bz\|^2,\label{eq: x_def}\\
\bar{\bx}^*(\bz) &=\argmin_{\bx\in X, A\bx=\bb} f(\bx) + \frac{\mu}{2} \| \bx-\bz\|^2.
\end{align}
\end{lemma}
We continue with the main \emph{descent}-type inequality on the potential function after one iteration of the algorithm. The proof of this lemma is rather intricate and requires a careful combination of the inner products coming from the previous lemmas, and using the particular update of the proximal center $\bz_{t+1}$ as well as parameter selections. Let us recall that $\bu^*(\bx, \by, \bz)$ and $\bx^*(\by, \bz)$ that appear in the lemma statement are defined in \eqref{eq: aeh4} and \eqref{eq: x_def}.
\begin{lemma} (cf.  \Cref{lem: appendix_descent}) \label{lem: vt_ineq}
Under Assumption \ref{asmp: 1}, with the parameters selected as \eqref{eq: alg_params} (see also \eqref{eq: param_defs_appendix}), the iterates of Algorithm \ref{alg: sgd} satisfy the inequality
    \begin{align} \label{desent lemma}
        \mathbb{E}V_t-\mathbb{E}V_{t+1}&\geq c_\beta\mathbb{E}\|\bz_{t+1}-\bz_t\|^2 -\lambda \tau^2\sigma^2 +c_\tau\mathbb{E}\|\bu^*(\bx_t,\by_{t+1},\bz_{t})-\bx_t\|^2 +c_\eta\mathbb{E}\|A\bx^*(\by_{t+1},\bz_t)-\bb\|^2,
    \end{align}
    where $c_\tau = \Theta(1/\sqrt{T})$, $c_\eta = \Theta(1/\sqrt{T})$, $c_\beta = \Theta(1/\sqrt{T})$ with their precise definitions given in \Cref{lem: appendix_descent}.
\end{lemma}
One novelty in our analysis is to show that this potential function is still lower bounded and decreases, in expectation, up to an error term depends on $\tau^2$ and the variance. To integrate this change into the framework of \cite{zhang2022error} under reasonable assumptions on the stochastic oracle as mentioned earlier in Section \ref{sec: alg}, we also slightly changed the definition of $\bz_{t+1}$ in the algorithm, due to technical reasons. In particular, in our case, we lose the control over $\|\bx_{t+1} - \bx_t\|^2$ (since we do not assume bounded domains in this section), whereas the deterministic analysis of \cite{zhang2022error} have a natural control over such terms.

The other change is the error coming from the variance of the stochastic gradients of $f$. This error causes the complexity to deteriorate compared to the deterministic case, but this is an effect that is common with algorithms based on SGD. In particular, with a correctly selected step size, we still obtain the same-order sample complexity as SGD, which is optimal even for unconstrained smooth nonconvex optimization \citep{arjevani2023lower}.

\subsubsection{Complexity analysis}
After \Cref{lem: vt_ineq}, it is straightforward to obtain
\begin{align*}
\mathbb{E}\|\bz_{t+1} - \bz_t\|^2 &\leq \varepsilon^2, \\
\mathbb{E}\|A\bx^*(\by_{t+1}, \bz_t) - \bb\|^2 &\leq \varepsilon^2, \\
\mathbb{E}\|\bu^*(\bx_t, \by_{t+1}, \bz_t) - \bx_t\|^2 &\leq \varepsilon^2,
\end{align*}
when $T=\Theta(\varepsilon^{-4})$.
Then, by tedious but straightforward calculations, we can directly get the bound on the norm of the gradient of the Moreau envelope, $\nabla \Psi(\bz_t)$, obtaining near-stationarity. The details for these estimations appear in \Cref{appendix: complexity_moreau}.

There are a couple more steps to go from this result to obtaining $\varepsilon$-stationary points, but the idea is simple. Since we know that small norm of $\nabla \Psi(\bz_t)$ means that we are near a stationary point, we can perform one more iteration of SGD with a batch size depending on $\varepsilon^{-2}$ to get an $\varepsilon$ stationary point, without changing the dominant term in the complexity. The details are given in \Cref{subsec: postprocess}.

\section{Extension to Random Linear Constraints}\label{sec: rand_sec}

We turn to the case when constraints are sampled, that is, we do not have access to the full matrix $A$, or vector $\bb$ but only to unbiased samples of them. This is a suitable setting, when, for example, we have a large matrix $A$. In particular, we have $A=\mathbb{E}_{\zeta \sim P}[A_\zeta], \bb=\mathbb{E}_{\zeta \sim P}[\bb_\zeta]$ and use
$
A_\zeta, \bb_\zeta
$ in the algorithm. We rewrite the template for convenience, as

\begin{equation}\label{eq: prob_stoc}
\min_{\bx\in X} f(\bx) \text{~subject to~} \mathbb{E}_{\zeta\sim P}[A_\zeta \bx - \bb_\zeta] = 0,
\end{equation}
where $f(\bx)=\mathbb{E}_{\xi\sim\Xi}[f(\bx, \xi)]$.
In this case, to get an unbiased stochastic gradient for the proximal augmented Lagrangian, we need to sample two i.i.d. samples of $\zeta$ and compute
\begin{equation}\label{eq: def_g_stoc}
\begin{aligned}
G(\bx, \by, \bz, \xi) &= f(\bx, \xi) + A_{\zeta^1}^\top \by + A_{\zeta^1}^\top(A_{\zeta^2}\bx - \bb_{\zeta^2}).
\end{aligned}
\end{equation}
An immediate issue that arises is that the variance of the stochastic gradients of the proximal AL function now scale as $\bx$ and $\by$. As such, assuming bounded variance would require assuming bounded dual variables, which is a strong assumption that is not satisfied in practice. To go around this difficulty, we have two adjustments, \emph{(i)} we need to assume a constraint qualification (CQ) condition and compactness of $X$ and \emph{(ii)} we include a safeguarding procedure in the algorithm to monitor when the dual variable gets too large. We will show that under these two modifications, we can obtain the same complexity guarantees as our previous setting with deterministic linear constraints.

\begin{algorithm*}[h]
    \caption{Stochastic smoothed and linearized ALM for stochastic constraints with dual safeguarding}
    \label{alg:example stochasitc}
 \begin{algorithmic}
    \STATE {\bfseries Input:} $M_y>\frac{M_V-M_\Psi+2M}{r}$ (check also \Cref{rem: my})
    \STATE {\bfseries Initialize:} $\bx_0=\bz_0 \in X$, $\by_0\in\mathbb{R}^m$, $\rho \geq 0$.
    \FOR{$t=0$ {\bfseries to} $T-1$}
    \STATE $\by_{t+1} = \by_t + \eta(A_{\zeta_t}\bx_t - \bb_{\zeta_t})$ where $\zeta_t\sim P$ is generated i.i.d. 
    \IF{$\|\by_{t+1}\| \geq M_y$}
    \STATE $\by_{t+1} =0$
    \ENDIF 
    \STATE Sample $\xi_t \sim \Xi$ i.i.d. and generate $\mathbb{E}_{\xi_t}[G(\bx_t, \by_{t+1}, \bz_t, \xi_t)] = \nabla_\bx L_{\rho}(\bx_t, \by_{t+1}) + \mu(\bx_t-\bz_t)$ as in \eqref{eq: def_g_stoc}
    \STATE $\bx_{t+1} = \proj_X(\bx_t - \tau G(\bx_t, \by_{t+1}, \bz_t, \xi_t))$   
    \STATE $\bz_{t+1} = \bz_t + \beta(\bx_t-\bz_t)$
    \ENDFOR
 \end{algorithmic}
 \label{alg: sgd2}
 \end{algorithm*}

\begin{remark}\label{rem: my}
 We give the choice of $M_y$ as follows. Let $M_V = \max_{\bx,\bz \in X} \{K(\bx,0,\bz) - 2d(0,\bz) + 2\Psi(\bz)\}$, $M = \max_{\bx,\bz \in X} \{|f(\bx)| + \frac{\mu}{2}\|\bx-\bz\|^2 + \frac{\rho}{2}\|A\bx-\bb\|^2\}$, where $K$ is defined in \eqref{eq: k_def_main} and $M_\Psi$ is a uniform lower bound of $\Psi(\bz_t)$, for example, $\underline{f}$.
 According to Assumption \ref{assumptions for stochasitc}, there exists a positive $r>0$ such that for any direction $\mathbf{d}\in \text{Range}(A)$, we can find a $\bx\in X$ satisfying $\|A\bx-\bb\|=r$ and $A\bx-\bb$ has the same direction as $\mathbf{d}$. Then, we choose $M_y$ as
 $$M_y>\frac{M_V-M_\Psi+2M}{r}.$$
 \end{remark}

\begin{assumption} \label{assumptions for stochasitc}
For the problem given in \eqref{eq: prob_stoc}, the following holds:
\begin{enumerate}
\item    The feasible set $\{\bx:\bx\in X, A\bx=\bb\}$ is bounded.
\item The origin is in the relative interior of the set $\{A\bx-\bb\colon \bx\in X\}$
\item $A$ has full row-rank.
    \end{enumerate}
\end{assumption}
Here, in addition to the assumptions in the earlier setting, we require a Slater's condition as well as compact domains to ensure boundedness of the dual variable. Slater's condition is a classical CQ, see for example \cite{bertsekas2003convex}.

In this setting, we only state our theorem for the near-stationarity. The $\varepsilon$-stationarity follows in the same way as the previous section by a post-processing step.
\begin{theorem} \label{thm:stochasitc linear constraints}
Let Assumptions \ref{asmp: 1} and \ref{assumptions for stochasitc} hold and run Algorithm \ref{alg: sgd2} with the parameters given in \eqref{eq: alg_params} (also \eqref{eq: param_defs_appendix}). We have that $\mathbb{E}\| \nabla \Psi(\bz_{t^*})\|\leq\varepsilon$ where $t^*$ is randomly selected from $\{1, \dots, T\}$ with $T=\Omega(\varepsilon^{-4})$. The stochastic oracle complexity is $O(\varepsilon^{-4})$.
\end{theorem}
For the proof of this theorem, we refer to \Cref{sec:appendix_stoc_const}.

As mentioned earlier, the optimal sample complexity for nonconvex optimization  with Lipschitz $\nabla f$ is $O(\varepsilon^{-4})$ \citep{arjevani2023lower}. Our result achieves this optimal complexity while handling linear constraints with random sampling.

\section{Extension with Variance Reduction}\label{sec: var_red}
\begin{algorithm}[ht!] 
    \caption{Stochastic smoothed and linearized ALM with STORM}
    \label{alg: alm-storm}
    \begin{algorithmic}
    \STATE Initialize: $x_0 = z_0 \in X, y_0 \in \mathbb{R}^m$, $\widehat{\nabla} f_0 = \frac{1}{N}\sum_{i=1}^N \nabla f(x_0, \zeta_i)$, $N=T^{1/6}$ and $\rho \geq 0$
    \FOR{$t = 0$ to $T-1$}
        \STATE $\by_{t+1} = \by_t + \eta (A\bx_t - \bb)$
        \STATE $G(\bx_t, \by_{t+1}, \bz_t) = \widehat{\nabla}f_t + A^\top\by+A^{\top} (A\bx_t - \bb) + \lambda (\bx_t - \bz_t)$
        \STATE $\bx_{t+1} = \text{proj}_X(\bx_t - \tau G(\bx_t, \by_{t+1}, \bz_t))$
        \STATE $\bz_{t+1} = \bz_t + \beta (\bx_t - \bz_t)$
        \STATE Sample $\xi_{t+1} \sim \Xi$ i.i.d. and set $\widehat{\nabla}f_{t+1} = \nabla f(\bx_{t+1}, \xi_{t+1}) + (1 - \alpha)(\widehat{\nabla}f_t  - \nabla f(\bx_t, \xi_{t+1}))$
    \ENDFOR
    \end{algorithmic}
    \end{algorithm}

We now introduce the STORM variance reduction technique from \cite{cutkosky2019momentum} into our algorithm, which improves the iteration and oracle complexity from $O(\varepsilon^{-4})$ to $O(\varepsilon^{-3})$ under a stronger assumption on the oracle, compared to our earlier sections.
With this variant, we reduce the variance of the stochastic gradients of the objective function, which leads to a faster convergence rate, and, also a simpler analysis that does not rely on the Moreau envelope $\varphi_{1/\lambda}$.

The ALM-STORM algorithm is given in Algorithm \ref{alg: alm-storm}. The main difference between ALM-STORM and the original algorithm is the update of the stochastic gradient estimate $\widehat{\nabla} f_t$. In STORM, the stochastic gradient estimate is updated using the previous stochastic gradient estimate, to reduce the variance of the stochastic gradients. The update of $\hat{\nabla} f_t$ is given by
\begin{equation}
    \widehat{\nabla} f_{t+1} = \nabla f(\bx_{t+1}, \xi_{t+1}) + (1 - \alpha)(\widehat{\nabla}f_t  -  \nabla f(\bx_t, \xi_t)),
\end{equation}
where $\alpha \in (0, 1)$ is a parameter to be determined. 

The update of $\widehat{\nabla}f_t $ is a linear combination of the current stochastic gradient estimate, the previous stochastic gradient and a correction term involving $\nabla f(\bx_{t+1}, \xi_t)$ and $\nabla f(\bx_t, \xi_t)$. It is easy to see that when $\alpha = 0$, Algorithm \ref{alg: alm-storm} reduces to Algorithm \ref{alg: sgd}, but we will see that a particular choice of $\alpha$ will help us obtain a better complexity under Assumption \ref{assumption: variance_reduction}, which is stronger than the oracle access and smoothness required in Assumption \ref{asmp: 1}.
\begin{remark}
We only use a minibatch in the initialization, which does not affect the overall complexity. The minibatch size is $N = T^{1/6}$, which is small compared to the total number of iterations $T$. The iterations of our algorithm only require $2$ stochastic gradients, $\nabla f(\bx_t, \xi_{t+1})$ and $\nabla f(\bx_{t+1}, \xi_{t+1})$.
\end{remark}

For the analysis of ALM-STORM, we introduce the new assumption mentioned above. This is used, for example, in \cite{arjevani2023lower}. In particular, \cite{arjevani2023lower} showed that the oracle complexity $O(\varepsilon^{-3})$ is tight under Assumption \ref{assumption: variance_reduction} even with no constraints.
\begin{assumption}\label{assumption: variance_reduction}
For a given $\xi\sim \Xi$, we can query $\nabla f(\bx, \xi)$ and $\nabla f(\by, \xi)$ for different points $\bx, \by$. There exists a constant $L_0 > 0$ such that for all $\bx, \by \in X$, we have
   \begin{equation*}
   \mathbb{E}_{\xi\sim \Xi}\|\nabla f(\bx,\xi)-\nabla f(\by,\xi)\|^2\leq L_0^2\|\bx-\by\|^2.
   \end{equation*} 
   We also have access to a stochastic gradient of $f$ that satisfies \eqref{eq: stoc_oracle}.
\end{assumption}

The proof of the following lemma, taken from \cite{cutkosky2019momentum}, is given in Appendix \ref{appendix: storm_proof} for completeness.
\begin{lemma}(from \cite{cutkosky2019momentum})\label{lemma: variance_reduction}
Let Assumption \ref{assumption: variance_reduction} hold. We have the estimation of the variance as:
$$\mathbb{E}\|\widehat{\nabla}f_{t+1}-\nabla f(\bx_{t+1})\|^2\leq (1-\alpha)^2\mathbb{E}\|\widehat{\nabla}f_t -\nabla f(\bx_t)\|^2+3(L_0^2+L_f^2)\mathbb{E}\|\bx_{t+1}-\bx_t\|^2+3\alpha^2\sigma^2.$$
\end{lemma}

We introduce a different potential function \( \bar{V}_t \) for the ALM-STORM algorithm compared to Sections  \ref{sec: main_conv} and \ref{sec: rand_sec}. The potential function we use in this section is similar to the one defined in \cite{zhang2022error}, with the exception of the last term that helps us control the error coming from the variance. In particular, we have
\begin{equation}\label{eq: barv_def}
    \bar{V}_t = K(\bx_t,\by_t,\bz_t) - 2d(\by_t, \bz_t) + 2\Psi(\bz_t)+\frac{1}{48(L_0^2+L_f^2)\tau}\mathbb{E}\|\widehat{\nabla}f_t -\nabla f(\bx_{t})\|^2,
\end{equation}
where 
\begin{equation}\label{eq: k_def_main}
K(\bx, \by, \bz) = L_\rho(\bx, \by) + \frac{\mu}{2}\|\bx-\bz\|^2
\end{equation} 
and $\bx\mapsto K(\bx, \by, \bz)$ is $L_K$-smooth with $L_K=L_f+\rho \|A\|+\mu$.

We first establish the descent-type lemma of this potential function, which is the key step in the analysis of the ALM-STORM algorithm. Compared to the deterministic settings as in \cite{zhang2022error}, we have the extra error due to using $\widehat{\nabla} f_t$ instead of the full gradient $\nabla f(\bx_t)$.
\begin{lemma}\label{lemma: descent_alm_storm}
Let \Cref{asmp: 1} hold.    For the iterates generated by Algorithm \ref{alg: alm-storm}, we have
    \begin{equation*}
        K(\bx_{t+1}, \by_{t+1}, \bz_t) - K(\bx_t, \by_{t+1}, \bz_t) \leq \frac{\tau}{2} \| \nabla f(\bx_t) - \widehat{\nabla}f_t  \|^2 - (\frac{1}{2\tau}-\frac{L_K}{2}) \| \bx_{t+1} - \bx_t \|^2.
    \end{equation*}
\end{lemma}
The proof of \Cref{lemma: descent_alm_storm} could be found in Appendix \ref{proof of lemma: descent_alm_storm}.
\begin{lemma} 
Let \Cref{asmp: 1} hold. For the iterates generated by Algorithm \ref{alg: alm-storm}, we have \label{lemma: descent_alm_storm2}
    \begin{equation} \label{eq: descent_K}
        \begin{aligned}
        K(\bx_t,\by_t,\bz_t)-K(\bx_{t+1},\by_{t+1},\bz_{t+1})&\geq -\eta \|A\bx_t-\bb\|^2+(\frac{\mu}{\beta}-\frac{3\mu}{4})\|\bz_{t+1}-\bz_t\|^2\\
        &-\frac{\tau}{2} \| \nabla f(\bx_t) - \widehat{\nabla}f_t  \|^2 +(\frac{1}{2\tau}-\frac{L_K}{2}-\mu) \| \bx_{t+1} - \bx_t \|^2
    \end{aligned}
    \end{equation}
\end{lemma}
The proof of Lemma \ref{lemma: descent_alm_storm2} also could be found in Appendix \ref{proof of lemma: descent_alm_storm}.

Then we can combine the above lemma analyzing one step change of $K(\bx_t,\by_t,\bz_t)$ with the lemmas analyzing one step changes of $d(\by_t,\bz_t), \Psi(\bz_t)$ (Lemma \ref{lem: swo4}) as well as the variance term (Lemma \ref{lemma: variance_reduction}), to obtain the final lemma for the change in the potential function $\bar V_t$ from $t$ to $t+1$. For the proof, we refer to Appendix \ref{appendix: storm_proof}.
\begin{theorem} \label{th: alm_storm_descent}
    Under \Cref{asmp: 1} and \Cref{assumption: variance_reduction}, with the parameters chosen as:
    \begin{equation}\label{eq: alg_params_storm}
        \begin{aligned}
            \mu &= \max\{2, 4L_f\}, \quad \tau \leq \min\left\{\frac{1}{4L_K+8\mu},\frac{1}{\sqrt{48(L_0^2+L_f^2)}}\right\}\\
            \eta &= \min\left\{ \frac{(\mu-L_f)^2\tau}{4\|A\|^2}, \frac{2\mu+\rho\|A\|}{4\|A\|^4}, \frac{\tau}{200\|A\|^2}, \frac{\tau(2\mu+\rho\|A\|^2)}{20\|A\|^2} \right\}, \\
            \beta &= \min\left\{ \frac{\tau}{100}, \frac{1}{50}, \frac{\eta}{36\mu\bar{\sigma}^2} \right\},\\
            \alpha&=48(L_0^2+L_f^2)\tau^2,
        \end{aligned}
    \end{equation}
    where $L_K=L_f+\rho \|A\|+\mu$, $\bar{\sigma}$ is defined in \Cref{global error}, 
    we have 
    \begin{equation} \label{eq: final descent_lemma_alm_storm}
    \begin{aligned}
        \mathbb{E}\bar{V}_t-\mathbb{E}\bar{V}_{t+1}&\geq \frac{\mu}{2\beta}\mathbb{E}\|\bz_t-\bz_{t+1}\|^2+\frac{1}{8\tau}\mathbb{E}\|\bx_t-\bx_{t+1}\|^2+\frac{\eta}{2}\mathbb{E}\|A\bx^*(\by_{t+1},\bz_t)-b\|^2+\frac{\tau}{4}\mathbb{E}\|\widehat{\nabla}f_t -\nabla f(\bx_{t})\|^2\\
        &\quad -144(L_0^2+L_f^2)\sigma^2\tau^3.
    \end{aligned}
\end{equation}
\end{theorem}
Note that, on a high level, the main difference between \Cref{th: alm_storm_descent} and \Cref{lem: vt_ineq} is that the order of $\tau$ in the error term is different. In \Cref{th: alm_storm_descent}, the order of $\tau$ is $O(\tau^3)$, while in \Cref{lem: vt_ineq}, the order of $\tau$ is $O(\tau^2)$, which contribute to  a faster convergence rate in the ALM-STORM algorithm.

\begin{theorem}\label{th: storm_complexity}
    Let Assumptions \ref{asmp: 1} and \ref{assumption: variance_reduction} hold. We have that $(x_{t^*},y_{t^*})$ is an $\varepsilon$-stationary point, where $t^*$ is selected uniformly at random from $\{1, \dots, T\}$ with $T=\Theta(\varepsilon^{-3})$. The complexity of the whole procedure is $O(\varepsilon^{-3})$.
\end{theorem}

For the proof of this theorem, we refer to Appendix \ref{appendix: storm_proof}.

\begin{remark}\label{rem: storm}
    Under Assumptions \ref{asmp: 1}, \ref{assumptions for stochasitc} and \ref{assumption: variance_reduction}, we can combine this variance reduction technique with our extension to stochastic constraints in \Cref{sec: rand_sec} to obtain the same $O(\varepsilon^{-3})$ complexity result for the stochastic linear constraints case. We provide a brief justification for this claim in Appendix \ref{appendix: storm_proof}.
\end{remark}
\section{Applications}\label{sec: apps}
\subsection{Distributed Optimization}
In this section, we consider the distributed optimization problem with the following form
\begin{equation}\label{eq: distributed}
    \min_{\bx\in X} \Big\{f(\bx)=\frac{1}{N}\sum_{i=1}^N f_i(\bx)\Big\},
\end{equation}
where $X\subset\mathbb{R}^n$ is a polyhedral set.

Typically, this problem is addressed using a network with \( N \) nodes, represented as an undirected graph \( G = (V, E) \), where \( V \) is the set of nodes and \( E \) is the set of edges. The number of nodes is \( |V| = N \) and the number of edges is \( |E| = M \). Each node \( i \) can only access its own local function \( f_i \) and communicate with with its neighboring node \( j \), meaning that an edge \( (i, j) \) exists in \( E \).
Classically, we now model this communication setting by introducing \( N \) local variables \( \bx_1, \bx_2, \dots, \bx_N \) for each node and define the concatenated vector as
\[
\bx = \begin{pmatrix} \bx_1  \\ \bx_2 \\ \vdots \\ \bx_N\end{pmatrix}.
\]
With this, the problem \eqref{eq: distributed} can be formulated as follows:
\begin{equation}\label{eq: prob_graph}
    \begin{aligned}
        &\min_{\bx_1, \dots, \bx_N} \frac{1}{N} \sum_{i=1}^{N} f_i(\bx_i)\\
        &\text{s.t. } \bx_i = \bx_j, \quad \forall (i,j) \in E \text{~and~} \bx_i \in X \quad \forall i=1,\dots, N.
    \end{aligned}
\end{equation}
Next, we work to reformulate this into a more concise representation. Specifically, we introduce the \emph{edge-agent incidence matrix} \( W \in \mathbb{R}^{\frac{N(N-1)}{2} \times N} \). Each row of \( W \) corresponds to a node in the graph \( G \). In particular, if we take the \( k \)th row and if this row represents the pair \( (i, j) \) in the graph, then if there is an edge between $(i,j)$, we define \( W_{k,i} = 1 \), \( W_{k,j} = -1 \). Then, we set all other entries in the \( k \)th row to zero.

Let us recall that $\bx_i\in\mathbb{R}^n$ and $\bx \in \mathbb{R}^{nN}$. To represent the relationships of nodes by using $W$, we define $A=W\otimes I_n$, where $\otimes$ denotes the Kronecker product. Then, we can rewrite the constraints \eqref{eq: prob_graph} as $A\bx=0$.

Here, we consider a more general case where the network structure graph is random, that is, the connections between the nodes may change from iteration to iteration.
We consider a discrete-time random graph model, which is discussed in \cite{Chaintreau2007networkmodel}. In this model, the network is represented as a time-varying graph \( G_t = (V, E_t) \), where \( V \) is the set of nodes and \( E_t \) is the set of edges at time \( t \). The edges \( E_t \) are determined by a probabilistic process, capturing the dynamic nature of the network.
$\mathbb{P}[(i,j)\in E_t]=p_{ij}$ for any pair of nodes $i,j$ and the events \{$(i,j)\in E_t$\}, for all pair of nodes $i,j$ are mutually independent.

This model is particularly useful for analyzing communication networks where connections between nodes are not static but change over time due to mobility, interference, or other dynamic factors. The random graph model allows us to study the behavior of algorithms and protocols under realistic, time-varying network conditions.

Hence, we model this situation with $W=W(\zeta)$, where $\zeta$ is a random variable. Then the constraints $A\bx=0$ changes to $\mathbb{E}_{\zeta \thicksim P}[A(\zeta)]\bx=A\bx=0$. In the discrete-time model, a row in $\mathbb{E}[W]$ represents the likelihood of a connection between $(i,j)$, that is, the i-th entry equals to $p_{ij}$, the j-th entry equals to $-p_{ij}$.

The problem \eqref{eq: prob_graph} comes to the following form:

\begin{equation}\label{eq:stochasitc linear constraints app}
    \begin{aligned}
        &\min_{\bx \in X^N} f(\bx)\\
        &\text{s.t. } \mathbb{E}_{\zeta \thicksim P}[A(\zeta)]\bx=0,
    \end{aligned}
\end{equation}
where \( f(\bx) = \frac{1}{N} \sum_{i=1}^{N} f_i(\bx_i) \). Then, we can use \Cref{alg:example stochasitc} to solve the problem \eqref{eq:stochasitc linear constraints app}.

In the discrete-time model, for example, the first row of $\mathbb{E}_{\zeta \thicksim P}[A(\zeta)]$ equals to $(p_{12},-p_{12},0,...,0)$, hence the first entry of $\mathbb{E}_{\zeta \thicksim P}[A(\zeta)\bx]=p_{12}x_1-p_{12}x_2$.  Then by the definition of $\varepsilon-$stationary point, particularly the feasibility bound,
we will get
\begin{align*}
    \sum_{i<j}\mathbb{E}\|p_{ij}x_i-p_{ij}x_j\|^2\leq \varepsilon^2.
\end{align*}
We assume $\forall i,j$, $p_{ij}\geq p>0$, then this condition assures that each $\bx_i$ converges to the same point, which is also the stationary solution of the original problem \eqref{eq: distributed}.

With our method, we do not need the assumption that the graph is connected at any iteration and our developments in \Cref{sec: rand_sec} apply for this problem. We can convert our stochastic primal-dual algorithm to distributed form, where we refer to \cite{chen21distributed}.

\subsection{Discrete Optimization with Smooth Nonconvex Regularizers}
In this section, we follow an idea from \cite{zhang17discrete_applications} to deal with discrete optimization problems by using continuous nonconvex regularizers to relax the discrete constraints. Then, this brings the need to handle objective functions with nonconvexity. 

We consider a communication network represented by a directed graph $G = (\mathcal{V}, \mathcal{L})$, where $\mathcal{V}$ denotes the set of nodes and $\mathcal{L}$ represents the set of directed links. 
We define $\mathcal{V}_f$ as the subset of function nodes capable of providing service function $f$, where each node has computational capacity $\mu_i$. The network serves $K$ data flows, each requiring a service function chain $\mathcal{F}(k)$ that must be executed in sequence. 
Then we denote $r_{ij}(k)$ as the rate of flow $k$ on link $(i,j)$, $r_{ij}(k,f)$ as the rate of virtual flow $(k,f)$ on link $(i,j)$ and $x_{i,j}(k)$ as the binary variable indicating whether or not  $f$ is used by flow $k$ in node $i$. 
The network slicing problem aims to determine optimal routes and flow rates that satisfy both service function chain requirements and capacity constraints of all links and function nodes.

We omit some details of the constraints about the network slicing problem for brevity, and just write down those as linear constraints which are discussed extensively in \cite{zhang17discrete_applications}. Then abstract form of this problem is
\begin{equation}\label{eq: prob_discrete}
    \begin{aligned}
    \min_{\mathbf{r}, \mathbf{x}} \quad & g(\mathbf{r}) = \sum_{k, (i,j)} r_{ij}(k) \\
    \text{s.t.} \quad & A\begin{bmatrix}
        \mathbf{r} \\
        \mathbf{x}
    \end{bmatrix}=0,  \\
    &\sum_{i\in V_f}x_{i,f}(k)=1, \forall f\in \mathcal{F}(k), \forall k\\
    & r_{ij}(k) \geq 0, \quad \forall k, \forall (i,j) \in \mathcal{L}, \\
    & r_{ij}(k,f) \geq 0, \quad \forall f \in \mathcal{F}(k), \forall k, \forall (i,j) \in \mathcal{L}, \\
    & x_{i,f}(k) \in \{0,1\}, \quad \forall i \in \mathcal{V}_f, \forall f \in \mathcal{F}(k), \forall k,
\end{aligned}
\end{equation}
where $\mathbf{r}=\{r_{ij}(k), r_{ij}(k,f)\}$ and $\mathbf{x}=\{x_{i,f}(k)\}$. This is a linear programming problem with binary constraints. 

Then \citet{zhang17discrete_applications} uses the continuous relaxation of the binary constraints and add a nonconvex regularizer to solve the problem. In particular, this work shows that the solution of the binary LP can be approximated by this continuous but \emph{nonconvex} problem
\begin{equation}\label{eq: prob_discrete_relax}
    \begin{aligned}
    \min_{\mathbf{r}, \mathbf{x}} \quad & g(\mathbf{r})+\sigma P_\epsilon(\bx) \\
    \text{s.t.} \quad & A\begin{bmatrix}
        \mathbf{r} \\
        \mathbf{x}
    \end{bmatrix}=0,  \\
    & r_{ij}(k) \geq 0, \quad \forall k, \forall (i,j) \in \mathcal{L}, \\
    & r_{ij}(k,f) \geq 0, \quad \forall f \in \mathcal{F}(k), \forall k, \forall (i,j) \in \mathcal{L}, \\
    & x_{i,f}(k) \in [0,1], \quad \forall i \in \mathcal{V}_f, \forall f \in \mathcal{F}(k), \forall k,
\end{aligned}
\end{equation}
where $\sigma>0$ is the penalty parameter and nonconvexity stems from $P_\epsilon$. In particular, we have $P_\epsilon(\bx)=\sum_k\sum_{f\in \mathcal{F}(k)}(\|x_f(k)+\epsilon\mathbf{1}\|_p^p-c_{\epsilon,f}$, where $x_f(k)=\{x_{i,f}(k)\}_{i\in V_f}$, $c_{\epsilon,f}=(1+\epsilon)^p+(|V_f|-1)\epsilon^p$ and $p\in (0,1)$, $\epsilon$ is any nonnegative constant.
We can then apply Algorithm \ref{alg:example} to solve the problem \eqref{eq: prob_discrete_relax}.

\subsection{Classification with Fairness}
We consider the setting of a \emph{binary classification} task, where the goal is to learn a decision rule
\[
f_\theta : \mathbb{R}^d \to \{-1, +1\},
\]
where $\theta$ is the parameter.

We note that we use a different notation in this section than the rest of our text to be compatible with the application we consider.

Given a training set of labeled examples \(\{(x_i, y_i)\}_{i=1}^N\), each \(x_i\) is a feature vector in \(\mathbb{R}^d\) and \(y_i \in \{-1, +1\}\). The task is to find parameters \(\theta\) that define a decision boundary and minimize a chosen loss function \(L(\theta)\) on the training data. Once trained, the classifier predicts \(+1\) if a test point's signed distance  to the boundary, denoted as $d_{\theta^*}(x)$, is non-negative, and \(-1\) otherwise, where $\theta^*=\argmin_\theta L(\theta)$.

In \cite{zafar17fairness}, the authors define the measure of (un)fairness of a decision boundary as the covariance between the set of sensitive attributes 
\(\{z_i\}_{i=1}^N\)
and the signed distance of each sample's feature vector to the decision boundary 
\(\{d_{\theta}(x_i)\}_{i=1}^N\). Formally,
\begin{equation}
\mathrm{Cov}\bigl(z, d_{\theta}(x)\bigr) 
= \mathbb{E}\bigl[(z - \bar{z})\,d_{\theta}(x)\bigr] 
\;-\; \mathbb{E}[z - \bar{z}]\,\bar{d}_{\theta}(x)
\;\;\approx\;\;
\frac{1}{N} \sum_{i=1}^N (z_i - \bar{z}) \, d_{\theta}\bigl(x_i\bigr),
\tag{2}
\end{equation}
where \(\mathbb{E}[z - \bar{z}]\,\bar{d}_{\theta}(x) = 0\) since \(\mathbb{E}[z - \bar{z}] = 0\).
Here, \(\bar{z}\) denotes the average of the sensitive attribute over the training set, and
\(\bar{d}_{\theta}(x)\) is the mean signed distance.

Considering the setting of linear classifier, that is, $f_\theta(x)=\langle \theta, x \rangle$,  one has the following problem:
\begin{equation}\label{eq: prob_fair_linear}
    \begin{aligned}
    &\min_{\theta} \quad  L(\theta) = \frac{1}{N} \sum_{i=1}^N V(f_\theta(x_i), y_i) \\
    &\text{s.t.}\quad \frac{1}{N}\sum_{i=1}^{N}(z_i-\bar{z})\theta^\top x_i\leq c\\
    &\quad \quad \; \frac{1}{N}\sum_{i=1}^{N}(z_i-\bar{z})\theta^\top x_i\geq -c,
\end{aligned}
\end{equation}
where $c$ is the covariance threshold. 

Although $L$ for logistic regression that is considered in \cite{zafar17fairness} is indeed convex, there are many nonconvex loss functions for this classification problem. For example
\cite{krause2004leveraging} proposes a smooth nonconvex loss function called Logistic difference loss function for classification problems, which is defined as follows:
\begin{equation}\label{eq: logistic_difference_loss}
    V(f(x),y)=\log(1+e^{-yf(x)})-\log(1+e^{-yf(x)-\mu}),
\end{equation}
where the $\mu$ is a parameter.

In \cite{zhao2010nonconvex_loss}, the authors propose smoothed 0-1 loss function as follow:
\begin{equation} \label{eq: smoothed_0-1_loss}
    V(f(x),y) =
\begin{cases} 
0, & yf(x) > 1 \\
\frac{1}{4} yf(x)^3 - \frac{3}{4} yf(x) + \frac{1}{2}, & -1 \leq yf(x) \leq 1 \\
1, & yf(x) < -1.
\end{cases}
\end{equation}
We refer to the review about the nonconvex loss functions used in classification problems in \cite{zhao2010nonconvex_loss}. This work showcases certain advantages of using nonconvex loss functions, such as robustness to outliers, better approximation to $0-1$ loss and improved generalization, which are supported by experiment results in \cite{zhao2010nonconvex_loss}.

When we use a nonconvex smooth loss function in the classification problem, we can apply our Algorithm \ref{alg:example} to solve the problem \eqref{eq: prob_fair_linear}.

\section{Related Works}\label{sec: relwork}
Since the literature of algorithms solving the problem \eqref{eq: prob1} is broad with different focuses, we will survey the related results in three sub-cases, covering different stochastic or deterministic access to objective and constraints. When we mention oracle or sample complexity results in the sequel, we always consider the complexity for obtaining an $\varepsilon$-stationary point, in view of the definition in \Cref{sec: prelim}.
\paragraph{Deterministic objective and deterministic constraints. } The setting when objective $f$ in \eqref{eq: prob1} is deterministic is the most well-studied with many results in the classical literature \citep{bertsekas2016nonlinear}. Recent work focused on characterizing the global oracle complexity of Lagrangian or augmented Lagrangian algorithms. With nonlinear and nonconvex constraints, many of the existing algorithms analyzing AL-based algorithms need to rely on strong constraint qualification and boundedness assumptions and use large penalty parameters to ensure feasibility \citep{li2021rate, lin2022complexity, kong2019complexity, kong2023accelerated, kong2023iteration}. The existing frameworks so far fail to capture the importance of dual variable updates, which are, in fact, the main reason behind the ability to use constant penalty parameters while ensuring convergence, see for example \cite{bertsekas2014constrained}. The recent works mentioned above obtained the complexity bounds $O(\varepsilon^{-3})$ for general nonlinear constraints with no specialization for linear constraints. When specialized to convex functional constraints, the best-known rate for these methods has been $O(\varepsilon^{-2.5})$ \citep{lin2022complexity}.

In the case when the constraints are linear, such as \eqref{eq: prob1} with $X=\mathbb{R}^n$, the work of \citet{hong2016decomposing} managed to analyze ALM with constant penalty parameters and non-negligible dual updates to get optimal complexity $O(\varepsilon^{-2})$. The case of $X\neq \mathbb{R}^n$ turned out to be significantly more challenging with many works focusing on variants of ALM with large penalty parameters (depending on the inverse of the final accuracy) to ensure near-feasibility and \emph{negligible} dual updates that do not help with feasibility and obtaining the suboptimal complexity $\widetilde{O}(\varepsilon^{-2.5})$ \citep{kong2023accelerated, kong2023iteration}. The exceptions are the works  \cite{zhang2020proximal, zhang2022error} that showed, for the case $X$ polyhedral,  near-optimal complexity $O(\varepsilon^{-2})$ with a constant penalty parameter and dual steps with constant step sizes, with no constraint qualification. The key step was the global error bound that our work also relied on.

\paragraph{Stochastic objective and deterministic constraints} One important step in generalizing the template to tasks arising in machine learning was to consider stochastic objectives where we have access to an unbiased gradient. With general nonlinear constraints and Lipschitzness of $\nabla f$,  the optimal sample complexity is $O(\varepsilon^{-4})$ which is obtained with double loop algorithms \citep{curtis2024worst, boob2023stochastic,ma2020quadratically}. These works also come with strong assumptions on the boundedness of the primal domain as well as constraint qualifications, which are often not necessary with linear constraints.

Another set of results concern stochastic optimization with deterministic nonlinear constraints with penalty-based algorithms and, requiring large penalty parameters to ensure near-feasibility rather than dual updates \citep{lu2024variance,alacaoglu2024complexity}. These works assume expected Lipschitzness of the stochastic gradients, stated in \Cref{assumption: variance_reduction}, which is stronger than Lipschitzness of $\nabla f$ (we will unpack this further in the sequel). Since these works focuses on nonlinear functional constraints, the analysis requires strong boundedness assumptions as well as constraint qualifications, unlike our results in \Cref{sec: main_conv} for deterministic linear constraints.

One of the most related to our setting is \cite{alacaoglu2024complexity} that considered an augmented Lagrangian algorithm with a constant penalty parameter and non-negligible dual updates and obtained the complexity $O(\varepsilon^{-3})$ for linear equality constraints and expected Lipschitzness. In particular, this work only covered the case $X=\mathbb{R}^n$ and left open the question of handling the case of more general $X$ (see \citep[Section 5]{alacaoglu2024complexity}). 

In this work, we address an important special case of this open question when $X$ is polyhedral, allowing our analysis to cover linear inequality constraints. The work of \cite{alacaoglu2024complexity} focused on applying variance reduction on estimation of the gradient of $f$, which means that the assumption on the stochastic gradients was Assumption \ref{assumption: variance_reduction}, stronger than Assumption \ref{asmp: 1}.
We show in \Cref{sec: var_red} how to obtain the same optimal complexity as this paper while handling the case when $X$ is polyhedral to cover problems with linear inequality constraints, which cannot be solved by \cite{alacaoglu2024complexity}.

Moreover, we also get the complexity $O(\varepsilon^{-4})$ under Assumption \ref{asmp: 1}.
This complexity is optimal under Assumption \ref{asmp: 1} and we refer to \cite{arjevani2023lower} for further details on the lower bounds.
In contrast, the work in \citep{alacaoglu2024complexity} does not have guarantees without \Cref{assumption: variance_reduction}.

\paragraph{Stochastic objective and stochastic constraints. } This is the most general class, where the existing results come with many assumptions that are not always easy to interpret, similar to the case of stochastic objective and deterministic nonconvex functional constraints described above \citep{li2024stochastic, alacaoglu2024complexity}. The state-of-the-art complexity result is $O(\varepsilon^{-5})$ and is obtained by using the expected Lipschitzness assumption above, by an inexact, double-loop, augmented Lagrangian algorithm in \cite{li2024stochastic} and by a single loop penalty algorithm in \cite{alacaoglu2024complexity}. These results concerning augmented Lagrangian methods all need to use large penalty parameters, which renders them as penalty methods since the dual updates do not contribute to the analysis for ensuring the feasibility. Other approaches for solving this sub-case also require double-loop algorithms and stronger assumptions since they focus on a generic nonconvex constraint \citep{boob2023stochastic, ma2020quadratically}. These works obtain the complexity $O(\varepsilon^{-6})$ since they do not assume expected Lipschitzness.

In conclusion, in this sub-case, none of the existing surveyed results used the structure of linear constraints, which we do in \Cref{sec: rand_sec} to achieve improved complexity guarantees.

\section*{Acknowledgements}
J. Zhang was supported by the MIT School of Engineering Postdoctoral Fellowship Program for Engineering Excellence. 

\bibliography{lit.bib}
\bibliographystyle{icml2025}

\newpage
\appendix
\onecolumn

\section{Proofs for \Cref{sec: main_conv}}\label{app: a}
Let us recall the following definition from \eqref{eq: k_def_main} which will be used extensively in the proofs
    \begin{equation*}
    K(\bx, \by, \bz) = L_\rho(\bx, \by) + \frac{\mu}{2}\|\bx-\bz\|^2.
    \end{equation*}
    With this notation, we have
    \begin{equation*}
    \bu^*(\bx, \by, \bz) = \min_{\bu \in X}K(\bu, \by, \bz) + \frac{\lambda}{2} \| \bu - \bx\|^2.
    \end{equation*}
We also introduce here some parameters that are used throughout, for convenience.
\begin{equation}\label{eq: param_defs_appendix}
\begin{aligned}
\mu &= \max\{2, 4L_f\}, \\
L_K &= L_f + \rho \|A\| + \mu, \\
\lambda &= L_K, \\
\sigma_4 &= \frac{\mu-L_f}{\mu}, \\
\tau &=  \frac{1}{6\lambda^2\sqrt{T}}, \\
\eta &= \min\left\{ \frac{2\mu+\rho\|A\|}{4\|A\|^4}, \frac{\tau}{200\|A\|^2}, \frac{\tau(2\mu+\rho\|A\|^2)}{20\|A\|^2} \right\}, \\
\beta &= \min\left\{ \frac{\tau}{100}, \frac{1}{50\lambda}, \frac{\eta}{36\mu\bar{\sigma}^2} \right\},\\
\gamma_s&=2\mu + \rho\|A\|, \gamma=\frac{(\mu-L_f)\lambda}{\mu-L_f+\lambda}, \gamma_K=\mu-L_f,
\end{aligned}
\end{equation}    
where $\bar{\sigma}$ is defined in \ref{global error}.

\subsection{Proofs for \Cref{lem: vt_ineq}}
    In the next lemma, the first part is using the idea of \cite{davis2019stochastic} to handle bounded variance assumption instead of the restricted bounded stochastic gradient assumption. The second part of the lemma also follows a similar idea as this work, with the exception of the dependence of the changing center point $\bz_t$. This introduces additional issues, since the stochastic gradient in the update of $\bx_{t+1}$ depends on $\bz_t$ whereas the proximal point $\bu^*(\bx_t, \by_{t+1}, \bz_{t+1})$ depends on $\bz_{t+1}$. Our analysis below estimates this additional error and shows it to be in the order of $\|\bz_{t+1}-\bz_t\|^2$, which will be handled later.

    \begin{lemma}Suppose that \Cref{asmp: 1} holds, for the proximal point $\bu^*(\bx_t, \by_{t+1}, \bz_t)$, defined in~\eqref{eq: aeh4} we have the characterization
\begin{equation}\label{eq: axg5}
\bu^*(\bx_t, \by_{t+1}, \bz_{t+1})=\proj_X(\tau\lambda \bx_t + (1-\tau \lambda)\bu^*(\bx_t, \by_{t+1}, \bz_{t+1}) - \tau \nabla_\bx K(\bu^*(\bx_t, \by_{t+1}, \bz_{t+1}), \by_{t+1}, \bz_{t+1})).
\end{equation} 
Moreover, for the sequence $\bx_{t+1}$ calculated as \Cref{alg: sgd}, if $\lambda = L_K=L_f+\rho \|A\|^2+\mu$ and $\tau \leq \frac{1}{6\lambda}$,  we have
      \begin{equation*}
      \mathbb{E}\|\bu^*(\bx_t, \by_{t+1}, \bz_{t+1})-\bx_{t+1}\|^2 \leq (1-\frac{\tau\lambda}{4})\mathbb{E}\|\bu^*(\bx_t, \by_{t+1}, \bz_{t+1})- \bx_t\|^2 +(\tau\mu+2\tau^2\mu^2)\mathbb{E}\|\bz_t-\bz_{t+1}\|^2+\tau^2\sigma^2
      \end{equation*}
    \begin{proof}
From the definition of $\bu^*(\bx_t , \by_{t+1}, \bz_{t+1})$ in \eqref{eq: aeh4}, we have
        \begin{equation*}
           \lambda(\bx_t-\bu^*(\bx_t, \by_{t+1}, \bz_{t+1})) \in  \nabla_\bx K(\bu^*(\bx_t, \by_{t+1}, \bz_{t+1}), \by_{t+1}, \bz_{t+1})+\partial I_X(\bu^*(\bx_t, \by_{t+1}, \bz_{t+1})).
           \end{equation*}
           Multiplying both sides by the step size $\tau$ and rearranging give
           \begin{align*}
&\tau\lambda \bx_t-\tau \nabla_\bx K(\bu^*(\bx_t, \by_{t+1}, \bz_{t+1}), \by_{t+1}, \bz_{t+1})+(1-\tau\lambda)\bu^*(\bx_t, \by_{t+1}, \bz_{t+1}) \\
& \in \bu^*(\bx_t, \by_{t+1}, \bz_{t+1})+ \tau\partial I_X(\bu^*(\bx_t, \by_{t+1}, \bz_{t+1})).
        \end{align*}
Since $(I+\tau\partial I_X)^{-1} = \prox_{I_X} = \proj_{X}$ due to $\partial I_X$ being the normal cone and proximal operator of a normal cone being the projection to the set, we have the first assertion.

We next establish the second assertion.
Using the just established identity \eqref{eq: axg5}, the update rule of $\bx_{t+1}$ and nonexpansiveness of the projection, we derive
\begin{align*}
&\| \bu^*(\bx_t, \by_{t+1}, \bz_{t+1}) - \bx_{t+1} \|^2 \\
&\leq \| \tau\lambda \bx_t + (1-\tau \lambda)\bu^*(\bx_t, \by_{t+1}, \bz_{t+1}) - \tau \nabla_\bx K(\bu^*(\bx_t, \by_{t+1}, \bz_{t+1}), \by_{t+1}, \bz_{t+1}) - [\bx_t -\tau G(\bx_t, \by_{t+1}, \bz_{t}, \xi_t)] \|^2.
\end{align*} 
We add and subtract $\nabla_{\bx} K(\bx_t, \by_{t+1}, \bz_{t})$ inside the squared norm, expand and take conditional expectation to obtain
\begin{align*}
&\mathbb{E}_t \| \bu^*(\bx_t, \by_{t+1}, \bz_{t+1}) - \bx_{t+1} \|^2 \\
&= \| (1-\tau \lambda)(\bu^*(\bx_t, \by_{t+1}, \bz_{t+1})-\bx_t) - \tau\nabla_\bx K(\bu^*(\bx_t, \by_{t+1}, \bz_{t+1}), \by_{t+1}, \bz_{t+1}) + \tau \nabla_\bx K(\bx_t, \by_{t+1}, \bz_{t}) \|^2 \\
&\quad + \tau^2\mathbb{E}_t \|G(\bx_t, \by_{t+1}, \bz_{t}, \xi_t) - \nabla_\bx K(\bx_t, \by_{t+1}, \bz_{t})\|^2.
\end{align*}
where the cross term disappeared because 
\begin{equation*}
\mathbb{E}_t [G(\bx_t, \by_{t+1}, \bz_{t}, \xi_t)] = \nabla_\bx K(\bx_t, \by_{t+1}, \bz_{t})
\end{equation*}
and $\bx_t, \by_{t+1}, \bz_{t+1}, \bu^*(\bx_t, \by_{t+1}, \bz_{t+1})$ are deterministic under the conditioning since $\bz_{t+1}$ only depends on $\bx_t$.

The second term on the right-hand side is trivially bounded by the oracle assumptions, that is,
\begin{equation*}
\mathbb{E}_t \| G(\bx_t, \by_{t+1}, \bz_{t+1}, \xi_t) - \nabla_\bx K(\bx_t, \by_{t+1}, \bz_{t+1}) \|^2 \leq \sigma^2.
\end{equation*}
For the first term, we further estimate as
\begin{align}
&\| (1-\tau \lambda)(\bu^*(\bx_t, \by_{t+1}, \bz_{t+1})-\bx_t) - \tau\nabla_\bx K(\bu^*(\bx_t, \by_{t+1}, \bz_{t+1}), \by_{t+1}, \bz_{t+1}) + \tau \nabla_\bx K(\bx_t, \by_{t+1}, \bz_{t}) \|^2 \notag \\
&\leq (1-\tau\lambda)^2 \| \bu^*(\bx_t, \by_{t+1}, \bz_{t+1})-\bx_t\|^2 \notag \\
&\quad+\tau(1-\tau\lambda)\langle \bu^*(\bx_t, \by_{t+1}, \bz_{t+1})-\bx_t,   \nabla_\bx K(\bx_t, \by_{t+1}, \bz_{t}) - \nabla_\bx K(\bu^*(\bx_t, \by_{t+1}, \bz_{t+1}), \by_{t+1}, \bz_{t+1}) \rangle \notag \\
&\quad+ \tau^2 \| \nabla_\bx K(\bx_t, \by_{t+1}, \bz_{t}) - \nabla_\bx K(\bu^*(\bx_t, \by_{t+1}, \bz_{t+1}), \by_{t+1}, \bz_{t+1})\|^2.\label{eq: sxp5}
\end{align}
Next, we turn to estimating 
\begin{align*}
&\| \nabla_\bx K(\bx_t, \by_{t+1}, \bz_{t}) - \nabla_\bx K(\bu^*(\bx_t, \by_{t+1}, \bz_{t+1}), \by_{t+1}, \bz_{t+1}) \| \\
&\leq \| \nabla_\bx K(\bx_t, \by_{t+1}, \bz_{t}) - \nabla_\bx K(\bx_t, \by_{t+1}, \bz_{t+1}) \| + \| \nabla_\bx K(\bx_t, \by_{t+1}, \bz_{t+1}) - \nabla_\bx K(\bu^*(\bx_t, \by_{t+1}, \bz_{t+1}), \by_{t+1}, \bz_{t+1}) \|.
\end{align*}
Note that, by definition, we have
\begin{equation*}
\nabla_\bx K(\bx_t, \by_{t+1}, \bz_{t}) - \nabla_\bx K(\bx_t, \by_{t+1}, \bz_{t+1}) = \mu (\bz_{t+1} - \bz_t).
\end{equation*}
Using this and the Lipschitzness of $\nabla_\bx K(\cdot, \by_{t+1}, \bz_{t+1})$, we then obtain
\begin{equation*}
\| \nabla_\bx K(\bx_t, \by_{t+1}, \bz_{t}) - \nabla_\bx K(\bu^*(\bx_t, \by_{t+1}, \bz_{t+1}), \by_{t+1}, \bz_{t+1}) \| \leq \mu\| \bz_{t+1} - \bz_t \| + L_K\| \bu^*(\bx_t, \by_{t+1}, \bz_{t+1})-\bx_t \|.
\end{equation*}
Plug this bound into the second term in the right-hand side of \eqref{eq: sxp5} after using Cauchy-Schwarz and Young's inequalities, we get
\begin{align*}
&\tau(1-\tau\lambda)\langle \bu^*(\bx_t, \by_{t+1}, \bz_{t+1})-\bx_t,   \nabla_\bx K(\bx_t, \by_{t+1}, \bz_{t}) - \nabla_\bx K(\bu^*(\bx_t, \by_{t+1}, \bz_{t+1}), \by_{t+1}, \bz_{t+1}) \rangle \\
&\leq \tau(1-\tau\lambda)\| \bu^*(\bx_t, \by_{t+1}, \bz_{t+1})-\bx_t \| (\mu\| \bz_{t+1} - \bz_t \| + L_K\| \bu^*(\bx_t, \by_{t+1}, \bz_{t+1})-\bx_t \|) \\
&\leq \tau(1-\tau\lambda)(L_K+\mu/2) \| \bu^*(\bx_t, \by_{t+1}, \bz_{t+1})-\bx_t \|^2 + \frac{\tau(1-\tau\lambda)\mu}{2}\| \bz_{t+1} - \bz_t\|^2.
\end{align*}
Using the last two inequalities in \eqref{eq: sxp5}, we obtain
\begin{align*}
&\| (1-\tau \lambda)(\bu^*(\bx_t, \by_{t+1}, \bz_{t+1})-\bx_t) - \tau\nabla_\bx K(\bu^*(\bx_t, \by_{t+1}, \bz_{t+1}), \by_{t+1}, \bz_{t+1}) + \tau \nabla_\bx K(\bx_t, \by_{t+1}, \bz_{t}) \|^2 \notag \\
&\leq [(1-\tau\lambda)^2 + \tau(1-\tau\lambda)(L_K+\mu) + 2\tau^2L_K^2]\| \bu^*(\bx_t, \by_{t+1}, \bz_{t+1})-\bx_t\|^2 + (\tau(1-\tau\lambda)\mu + 2\tau^2 \mu^2)\|\bz_{t+1} - \bz_t\|^2.
\end{align*}
We estimate the coefficient of the first term. First, note that $\tau \leq \frac{1}{\lambda}$ and $\mu \leq L_K = \lambda$. As a result, we have
\begin{align*}
(1-\tau\lambda)^2 + \tau(1-\tau\lambda)(L_K+\mu/2) + 2\tau^2L_K^2 &\leq 1-2\tau\lambda+\tau^2\lambda^2 + \frac{3\tau L_K}{2} - \frac{3\tau^2\lambda L_K}{2} + 2\tau^2L_K^2 \\
&\leq 1 - \frac{\tau\lambda}{2} + \tau^2\lambda^2 + \frac{\tau^2 L_K^2}{2}  \\
&\leq 1- \frac{\tau\lambda}{4},
\end{align*}
since $\tau \leq \frac{1}{6\lambda}$.

Finally, since $\tau(1-\tau\lambda)\mu + 2\tau^2\mu^2 \leq \tau\mu + 2\tau^2\mu^2$, the proof is completed after taking full expectation of the resulting equality.
    \end{proof}
\end{lemma}
\begin{lemma}\label{lem: swo3}
Let \Cref{asmp: 1} hold, then if $\lambda = L_K$ and $\tau \leq \frac{1}{6\lambda}$ we have
    \begin{align}
         \mathbb{E}\varphi_{1/\lambda}(\bx_{t+1}, \by_{t+1}, \bz_{t+1}) &\leq \mathbb{E}\varphi_{1/\lambda}(\bx_{t}, \by_{t+1}, \bz_{t+1})-\frac{\tau\lambda^2}{16}\mathbb{E}\|\bu^*(\bx_t, \by_{t+1}, \bz_{t})-\bx_t\|^2 \notag \\
         &\quad+(\lambda \tau \mu+2\lambda \tau^2\mu^2+\frac{\tau\lambda^2\mu^2}{8\gamma_s^2})\mathbb{E}\|\bz_t-\bz_{t+1}\|^2+\lambda \tau^2\sigma^2,
    \end{align}
    where $\gamma_s=2\mu+\rho \|A\|$. 
    \begin{proof}
    By the definition of $\varphi_{1/\lambda}$ and $\bu^*(\bx, \by_{t+1}, \bz_{t+1})$, we have
        \begin{align}
            \mathbb{E}\varphi_{1/\lambda}(\bx_{t+1}, \by_{t+1}, \bz_{t+1})&\leq \mathbb{E} K(\bu^*(\bx_t, \by_{t+1}, \bz_{t+1}), \by_{t+1}, \bz_{t+1})+\frac{\lambda}{2}\mathbb{E} \|\bu^*(\bx_{t}, \by_{t+1}, \bz_{t+1})-\bx_{t+1}\|^2\notag\\
                &\leq \mathbb{E} K(\bu^*(\bx_t, \by_{t+1}, \bz_{t+1}), \by_{t+1}, \bz_{t+1})+\left(\frac{\lambda}{2}-\frac{\tau\lambda^2}{8}\right)\mathbb{E}\|\bu^*(\bx_t, \by_{t+1}, \bz_{t+1})-\bx_t\|^2 \notag\\
                &\quad+(\lambda \tau \mu+2\lambda \tau^2\mu^2)\mathbb{E}\|\bz_t-\bz_{t+1}\|^2+\lambda \tau^2\sigma^2\notag\\
                &=\mathbb{E}\varphi_{1/\lambda}(\bx_{t}, \by_{t+1}, \bz_{t+1})-\frac{\tau\lambda^2}{8} \mathbb{E} \|\bu^*(\bx_t, \by_{t+1}, \bz_{t+1})-\bx_t\|^2\notag \\
                &\quad+(\lambda \tau \mu+2\lambda \tau^2\mu^2)\mathbb{E}\|\bz_t-\bz_{t+1}\|^2+\lambda \tau^2\sigma^2.\label{eq: aro3}
        \end{align}
We next bound the second term on the right-hand side by using
        \begin{align}
 \|\bu^*(\bx_t, \by_{t+1}, \bz_{t+1})-\bx_t\|^2&=\|\bu^*(\bx_t, \by_{t+1}, \bz_{t+1})-\bu^*(\bx_t, \by_{t+1}, \bz_{t})+\bu^*(\bx_t, \by_{t+1}, \bz_{t})-\bx_t\|^2\notag \\
            &\geq \frac{1}{2}\|\bu^*(\bx_t, \by_{t+1}, \bz_{t})-\bx_t\|^2-\|\bu^*(\bx_t, \by_{t+1}, \bz_{t+1})-\bu^*(\bx_t, \by_{t+1}, \bz_{t}\|^2\notag \\
            &\geq \frac{1}{2}\|\bu^*(\bx_t, \by_{t+1}, \bz_{t})-\bx_t\|^2-\frac{p^2}{\gamma_s^2}\|\bz_t-\bz_{t+1}\|^2,\\
        \end{align}
        where the last line used~\eqref{s^*(x,y,z')-s^*(x,y,z)}.
        
        We substitute the last inequality into  \eqref{eq: aro3} to conclude.
    \end{proof}
\end{lemma}
Since the previous result only allowed us to connect $\varphi_{1/\lambda}(\bx_{t+1}, \by_{t+1}, \bz_{t+1})$ to $\varphi_{1/\lambda}(\bx_{t}, \by_{t+1}, \bz_{t+1})$, we now need to analyze the effect of changing $\by_{t+1}$ and $\bz_{t+1}$ in $\varphi_{1/\lambda}$. The main idea of this lemma is similar to \cite{zhang2022error}, where the difference lies in the fact that our potential involves the Moreau envelope of $K(\bx, \by, \bz)$ whereas the potential of \cite{zhang2022error} involves $K(\bx, \by, \bz)$ hence this work considers the change of the arguments in the function $K$ instead of $\varphi_{1/\lambda}$. Therefore, our proof uses the properties of the Moreau envelope which was not needed in \cite{zhang2022error}.
\begin{lemma}\label{lem: phi_misc}
    Suppose that \Cref{asmp: 1} holds, for $\varphi_{1/\lambda}$ defined in \eqref{eq: varphi_def}, we have that
    \begin{equation*}
        \begin{aligned}
        \varphi_{1/\lambda}(\bx_{t}, \by_t, \bz_t)-\varphi_{1/\lambda}(\bx_{t}, \by_{t+1}, \bz_t)&\geq \langle \by_{t}- \by_{t+1},A\bu^{*}(\bx_{t}, \by_{t}, \bz_{t})-b\rangle +\frac{\gamma_s}{2}\|\bu^{*}(\bx_{t}, \by_{t}, \bz_{t})-\bu^{*}(\bx_{t}, \by_{t+1}, \bz_{t})\|^{2},\\
        \varphi_{1/\lambda}(\bx_{t}, \by_{t+1}, \bz_t)-\varphi_{1/\lambda}(\bx_{t}, \by_{t+1}, \bz_{t+1})&\geq \frac{\mu}{2} \langle \bz_{t+1} - \bz_t, 2\bu^*(\bx_t, \by_{t+1}, \bz_t) - \bz_{t+1} - \bz_t \rangle \\
        &\quad+\frac{\gamma_{s}}{2}\|\bu^{*}(\bx_{t}, \by_{t+1}, \bz_{t+1})-\bu^{*}(\bx_{t}, \by_{t+1}, \bz_{t})\|^{2},
    \end{aligned}
    \end{equation*}
    where $\gamma_s=2\mu+\rho \|A\|$. 

    \begin{proof}
    We first consider the change in $\by$ argument of $\varphi_{1/\lambda}$. By using the definition of $\varphi_{1/\lambda}$, we have
         \begin{align}
        \varphi_{1/\lambda}(\bx_{t}, \by_t, \bz_t)-\varphi_{1/\lambda}(\bx_{t}, \by_{t+1}, \bz_t)&=K(\bu^{*}(\bx_{t}, \by_{t}, \bz_{t}), \by_{t}, \bz_{t})+\frac{\lambda}{2}\|\bx_{t}-\bu^{*}(\bx_{t}, \by_{t}, \bz_{t})\|^{2}\notag \\
        &\quad-K(\bu^{*}(\bx_{t}, \by_{t+1}, \bz_{t}), \by_{t+1}, \bz_{t})-\frac{\lambda}{2}\|\bx_{t}-\bu^{*}(\bx_{t}, \by_{t+1}, \bz_{t})\|^{2}\notag \\
        &=K(\bu^{*}(\bx_{t}, \by_{t}, \bz_{t}), \by_{t}, \bz_{t})-K(\bu^{*}(\bx_{t}, \by_{t}, \bz_{t}), \by_{t+1}, \bz_{t})\notag \\
        &\quad+K(\bu^{*}(\bx_{t}, \by_{t}, \bz_{t}), \by_{t+1}, \bz_{t})+\frac{\lambda}{2}\|\bx_{t}-\bu^{*}(\bx_{t}, \by_{t}, \bz_{t})\|^{2}\notag \\
        &\quad-K(\bu^{*}(\bx_{t}, \by_{t+1}, \bz_{t}), \by_{t+1}, \bz_{t})-\frac{\lambda}{2}\|\bx_{t}-\bu^{*}(\bx_{t}, \by_{t+1}, \bz_{t})\|^{2},\label{eq: mpw4}
        \end{align}
        where the second equality adds and subtracts $K(\bu^{*}(\bx_{t}, \by_{t}, \bz_{t}), \by_{t+1}, \bz_{t})$. 
        
From the definition of $\by_{t+1}$, it trivially follows that
\begin{align*}
K(\bu^{*}(\bx_{t}, \by_{t}, \bz_{t}), \by_{t}, \bz_{t})-K(\bu^{*}(\bx_{t}, \by_{t}, \bz_{t}), \by_{t+1}, \bz_{t}) = \langle \by_{t}-\by_{t+1},A\bu^{*}(\bx_{t}, \by_{t}, \bz_{t})-\bb\rangle.
\end{align*}

Next, we use the property that $K(\cdot, \by_{t+1}, \bz_{t}) + \frac{\lambda}{2} \| \cdot - \bx_t\|^2$ is $\gamma_s$-strongly convex with minimizer $\bu^*(\bx_t, \by_{t+1}, \bz_t)$ to obtain
\begin{align*}
&K(\bu^{*}(\bx_{t}, \by_{t}, \bz_{t}), \by_{t+1}, \bz_{t})+\frac{\lambda}{2}\|\bx_{t}-\bu^{*}(\bx_{t}, \by_{t}, \bz_{t})\|^{2}-K(\bu^{*}(\bx_{t}, \by_{t+1}, \bz_{t}), \by_{t+1}, \bz_{t})-\frac{\lambda}{2}\|\bx_{t}-\bu^{*}(\bx_{t}, \by_{t+1}, \bz_{t})\|^{2} \\
        &\geq \frac{\gamma_s}{2}\|\bu^{*}(\bx_{t}, \by_{t}, \bz_{t})-\bu^{*}(\bx_{t}, \by_{t+1}, \bz_{t})\|^{2}.
\end{align*}
Combining the last two estimates in \eqref{eq: mpw4} gives the first assertion.
        
Next, we analyze the effect of changing the $\bz$ component in $\varphi_{1/\lambda}$. Similar to the proof of the first assertion, we start with the definition of $\varphi_{1/\lambda}$ and then add and subtract $K(\bu^*(\bx_t, \by_{t+1}, \bz_{t+1})$ to obtain
        \begin{align}
            &\varphi_{1/\lambda}(\bx_{t}, \by_{t+1}, \bz_t)-\varphi_{1/\lambda}(\bx_{t}, \by_{t+1}, \bz_{t+1})\notag \\
            &=K(\bu^{*}(\bx_{t}, \by_{t+1}, \bz_{t}), \by_{t+1}, \bz_{t})+\frac{\lambda}{2}\|\bx_{t}-\bu^{*}(\bx_{t}, \by_{t+1}, \bz_{t})\|^{2}\notag \\
            &\quad-K(\bu^{*}(\bx_{t}, \by_{t+1}, \bz_{t+1}), \by_{t+1},\bz_{t+1})-\frac{\lambda}{2}\|\bx_{t}-\bu^{*}(\bx_{t}, \by_{t+1}, \bz_{t+1})\|^{2}\notag \\
            &=K(\bu^{*}(\bx_{t}, \by_{t+1}, \bz_{t}), \by_{t+1}, \bz_{t})-K(\bu^{*}(\bx_{t}, \by_{t+1}, \bz_{t}), \by_{t+1}, \bz_{t+1})\notag \\
            &\quad +K(\bu^{*}(\bx_{t}, \by_{t+1}, \bz_{t}), \by_{t+1},\bz_{t+1})+\frac{\lambda}{2}\|\bx_{t}-\bu^{*}(\bx_{t}, \by_{t+1}, \bz_{t})\|^{2}\notag \\
            &\quad-K(\bu^{*}(\bx_{t}, \by_{t+1}, \bz_{t+1}), \by_{t+1}, \bz_{t+1})-\frac{\lambda}{2}\|\bx_{t}- \bu^{*}(\bx_{t}, \by_{t+1}, \bz_{t+1})\|^{2}.\label{eq: swp4}
        \end{align}
        First, by definition, of $K$, it trivially follows that
\begin{align*}
K(\bu^{*}(\bx_{t}, \by_{t+1}, \bz_{t}), \by_{t+1}, \bz_{t})-K(\bu^{*}(\bx_{t}, \by_{t+1}, \bz_{t}), \by_{t+1}, \bz_{t+1}) &= \frac{\mu}{2}\|\bu^{*}(\bx_{t}, \by_{t+1}, \bz_{t})-\bz_{t}\|^{2}\\
&\quad-\frac{\mu}{2}\|\bu^{*}(\bx_{t}, \by_{t+1}, \bz_{t})-\bz_{t+1}\|^{2}.
\end{align*}
For the remaining terms in the right-hand side, we again use that $K(\cdot, \by_{t+1}, \bz_{t+1}) + \frac{\lambda}{2} \| \cdot - \bx_t\|^2$ is $\gamma_s$-strongly convex with minimizer $\bu^*(\bx_t, \by_{t+1}, \bz_{t+1})$ to deduce
\begin{align*}
&K(\bu^{*}(\bx_{t}, \by_{t+1}, \bz_{t}), \by_{t+1},\bz_{t+1})+\frac{\lambda}{2}\|\bx_{t}-\bu^{*}(\bx_{t}, \by_{t+1}, \bz_{t})\|^{2}\\
&\quad-K(\bu^{*}(\bx_{t}, \by_{t+1}, \bz_{t+1}), \by_{t+1}, \bz_{t+1})-\frac{\lambda}{2}\|\bx_{t}- \bu^{*}(\bx_{t}, \by_{t+1}, \bz_{t+1})\|^{2} \\
            &\geq \frac{\gamma_{s}}{2}\|\bu^{*}(\bx_{t}, \by_{t+1}, \bz_{t+1})-\bu^{*}(\bx_{t}, \by_{t+1}, \bz_{t})\|^{2}.
\end{align*}
Plugging in the last two estimates in \eqref{eq: swp4} gives the second assertion.
    \end{proof}
\end{lemma}
\begin{corollary}\label{cor: asv4}
    Suppose that \Cref{asmp: 1} holds, for $\varphi_{1/\lambda}$ defined in \eqref{eq: varphi_def}, we have that
    \begin{equation*}
        \begin{aligned}
        \varphi_{1/\lambda}(\bx_{t}, \by_t, \bz_t) - \varphi_{1/\lambda}(\bx_{t+1}, \by_{t+1}, \bz_{t+1}) &\geq \frac{\tau\lambda^2}{16}\mathbb{E}\|\bu^*(\bx_t, \by_{t+1}, \bz_{t})-\bx_t\|^2 \notag \\
         &\quad-(\lambda \tau \mu+2\lambda \tau^2\mu^2+\frac{\tau\lambda^2\mu^2}{8\gamma_s^2})\mathbb{E}\|\bz_t-\bz_{t+1}\|^2-\lambda \tau^2\sigma^2\\
        &\quad-\eta \langle A\bx_t - \bb,A\bu^{*}(\bx_{t}, \by_{t}, \bz_{t})-b\rangle \\
        &\quad
 +\frac{\mu}{2} \langle \bz_{t+1} - \bz_t, 2\bu^*(\bx_t, \by_{t+1}, \bz_t) - \bz_{t+1} - \bz_t \rangle,
    \end{aligned}
    \end{equation*}
    where $\gamma_s=2\mu+\rho \|A\|$.
\end{corollary}
\begin{proof}
We sum up the results in \Cref{lem: swo3} and \Cref{lem: phi_misc}, plug in the definition of $\by_{t+1}$ and discard two nonnegative terms on the right-hand side to get the result.
\end{proof}
Next, we analyze the rest of the terms appearing in the potential function. This lemma is only using the definition of $d(\by, \bz)$ and $\Psi(\bz)$ and is equivalent to \cite{zhang2022error} and hence we omit its proof. Notably, these bounds are agnostic to the algorithm used to generate the sequences. Note that the only difference is that in the result below, we do not use the definition of $\by_{t+1}$ whereas the proof in \cite{zhang2022error} uses this definition. The rest of the estimations are precisely the same.
\begin{lemma}{\citep[Lemma 3.2, Lemma 3.3]{zhang2020proximal}}\label{lem: swo4}
For the functions $d(\by, \bz)$ and $\Psi(\bz)$ defined in \eqref{eq: def_d} and \eqref{eq: moreau_orig},we have
    \begin{align*} 
        d(\by_{t+1}, \bz_{t+1})-d(\by_t, \bz_t)&\geq \eta \langle A\bx_{t}-\bb, A\bx^*(\by_{t+1}, \bz_t)-\bb \rangle+\frac{\mu}{2}\langle \bz_{t+1}-\bz_{t}, \bz_{t+1}+\bz_{t}-2\bx^*(\by_{t+1},\bz_{t+1}) \rangle,\\
        \Psi(\bz_{t+1})-\Psi(\bz_t)&\leq \mu\langle \bz_{t+1}-\bz_t, \bz_t-\bar{\bx}^*(\bz_t)\rangle+\frac{\mu}{2\sigma_4}\|\bz_t-\bz_{t+1}\|^2,
    \end{align*}
where $\sigma_4$ is defined in \eqref{eq: param_defs_appendix}.
\end{lemma}
In the next lemma, we will join the previous lemmas and characterize the change in the potential function.
\begin{lemma}[cf. \Cref{lem: vt_ineq}]\label{lem: appendix_descent}
Let Assumption \ref{asmp: 1} hold. By using the parameters \eqref{eq: param_defs_appendix} in Algorithm \ref{alg: sgd}, we obtain
    \begin{equation} \label{desent lemma}
        \mathbb{E}V_t-\mathbb{E}V_{t+1}\geq c_\beta \mathbb{E}\|\bz_{t+1}-\bz_t\|^2+c_\tau\mathbb{E}\|\bu^*(\bx_t, \by_{t+1}, \bz_{t})-\bx_t\|^2+c_\eta \mathbb{E}\|A\bx^*(\by_{t+1}, \bz_t)-\bb\|^2-\lambda \tau^2\sigma^2,
    \end{equation}
    where $c_\beta = \frac{\mu}{50\beta}$, $c_\tau = \frac{7\tau\lambda^2}{400}$, $c_\eta = \frac{\eta}{4}$.
\end{lemma}
\begin{proof}

Combining \Cref{cor: asv4} and \Cref{lem: swo4}, we obtain
\begin{align*}
    \mathbb{E}[V_t]-\mathbb{E}[V_{t+1}]&=\mathbb{E}\varphi_{1/\lambda}(x_t,y_t,z_t)-\mathbb{E}\varphi_{1/\lambda}(x_{t+1},y_{t+1},z_{t+1})+2\mathbb{E}d(y_{t+1},z_{t+1})-2\mathbb{E}d(y_{t},z_{t})+2\Psi(z_t)-2\Psi(z_{t+1})\\
    &\geq \frac{\tau\lambda^2}{16}\mathbb{E}\|\bu^*(\bx_t, \by_{t+1}, \bz_{t})-\bx_t\|^2-(\lambda \tau \mu +2\lambda \tau^2\mu^2+\frac{\tau\lambda^2\mu^2}{8\gamma_s^2})\mathbb{E}\|\bz_t-\bz_{t+1}\|^2-\lambda \tau^2\sigma^2\\
    &\quad -{\mathbb{E}\eta \langle A\bx_t-\bb,A\bu^*(\bx_t, \by_t, \bz_t)-\bb\rangle }+{\frac{\mu}{2}\mathbb{E}\langle \bz_{t+1}-\bz_t, 2\bu^*(\bx_t, \by_{t+1}, \bz_t)- \bz_t- \bz_{t+1}\rangle }\\
    &\quad +2\eta \mathbb{E}\langle A\bx_t-\bb, A\bx^*(\by_{t+1}, \bz_t)-\bb \rangle+\mu\mathbb{E}\langle \bz_{t+1}-\bz_{t}, \bz_{t+1}+\bz_{t}-2\bx^*(\by_{t+1},\bz_{t+1})\rangle\\
    &\quad -2\mu\mathbb{E}\langle \bz_{t+1}-\bz_t, \bz_t-\bar{\bx}^*(\bz_t)\rangle-\frac{\mu}{\sigma_4}\mathbb{E}\|\bz_t-\bz_{t+1}\|^2.
\end{align*}
We next manipulate the terms on the right-hand side.

First, by adding and subtracting $A\bx_t$ on the second argument of the inner product, we get
\begin{equation*}
    -\eta \langle A\bx_t-\bb, A\bu^*(\bx_t, \by_t, \bz_t)-\bb\rangle=
-\eta\|A\bx_t-\bb\|^2-\eta\langle A\bx_t-\bb, A\bu^*(\bx_t, \by_t, \bz_t)-A\bx_t\rangle.
\end{equation*}
Consequently, we have
\begin{align*}
&-\eta \langle A\bx_t-\bb, A\bu^*(\bx_t, \by_t, \bz_t)-\bb\rangle+2\eta \langle A\bx_t-\bb, A\bx^*(\by_{t+1}, \bz_t)-\bb \rangle\\
&=-\eta\|A\bx_t-A\bx^*(\by_{t+1}, \bz_t)\|^2+\eta\|A\bx^*(\by_{t+1}, \bz_t)-\bb\|^2 -\eta\langle A\bx_t-\bb, A\bu^*(\bx_t, \by_t, \bz_t)-A\bx_t\rangle.
\end{align*}
Second, adding and subtracting $2\bx_t$ in the second argument of the inner product gives
\begin{equation*}
\frac{\mu}{2}\langle \bz_{t+1}-\bz_t, 2\bu^*(\bx_t, \by_{t+1}, \bz_t)-\bz_t-\bz_{t+1}\rangle=\frac{\mu}{2}\langle \bz_{t+1}-\bz_t, 2\bu^*(\bx_t, \by_{t+1}, \bz_t)-2\bx_{t}\rangle+\frac{\mu}{2}\langle \bz_{t+1}-\bz_t, 2\bx_{t}-\bz_t-\bz_{t+1}\rangle.
\end{equation*}
We continue estimating the inner products involving $\bz_{t+1} - \bz_t$. Note that $\bz_{t+1} = \bz_t + \beta(\bx_t - \bz_t) \iff 2\bx_t - 2\bz_t = \frac{2}{\beta}(\bz_{t+1} - \bz_t)$
\begin{align*}
\frac{\mu}{2} \langle \bz_{t+1} - \bz_t, 2\bx_t - \bz_t - \bz_{t+1} \rangle &= \frac{\mu}{2} \langle \bz_{t+1} - \bz_t, 2\bx_t - 2\bz_t \rangle +\frac{\mu}{2} \langle \bz_{t+1} - \bz_t, \bz_t - \bz_{t+1} \rangle \\
&= \frac{\mu}{2} \left( \frac{2}{\beta} - 1 \right) \| \bz_t-\bz_{t+1}\|^2 \geq \frac{\mu}{2\beta} \| \bz_t-\bz_{t+1}\|^2,
\end{align*}
where the last inequality is due to $\beta\leq 1$.
Next, we have
\begin{align*}
&\mu\langle \bz_{t+1}-\bz_{t}, \bz_{t+1}+\bz_{t}-2\bx^*(\by_{t+1},\bz_{t+1})\rangle -2\mu\langle \bz_{t+1}-\bz_t, \bz_t-\bar{\bx}^*(\bz_t)\rangle \\
&= \mu\| \bz_{t+1} - \bz_t\|^2 + 2\mu\langle \bz_{t+1} - \bz_t, \bar{\bx}^*(\bz_t) - \bx^*(\by_{t+1}, \bz_{t+1}) \rangle.
\end{align*}
We can use Cauchy-Schwarz, triangle and Young's inequalities on the second term to get
\begin{align*}
\langle \bz_{t+1} - \bz_t, \bar{\bx}^*(\bz_t) - \bx^*(\by_{t+1}, \bz_{t+1}) \rangle &\geq -\| \bz_{t+1} - \bz_t \| (\| \bar{\bx}^*(\bz_t) -\bx^*(\by_{t+1}, \bz_{t})\| +\| \bx^*(\by_{t+1}, \bz_{t}) - \bx^*(\by_{t+1}, \bz_{t+1}) \|) \\
&\geq -\left(\frac{1}{2\zeta} +\frac{1}{\sigma_4} \right) \| \bz_{t+1} - \bz_t\|^2 - \frac{\zeta}{2} \| \bar{\bx}^*(\bz_t) - \bx^*(y_{t+1}, \bz_t)\|^2,
\end{align*}
where the last step also used \eqref{x(y',z)-x(y,z)}.
Consequently, we obtain
\begin{align*}
&\mu\langle \bz_{t+1}-\bz_{t}, \bz_{t+1}+\bz_{t}-2\bx^*(\by_{t+1},\bz_{t+1})\rangle -2\mu\langle \bz_{t+1}-\bz_t, \bz_t-\bar{\bx}^*(\bz_t)\rangle \\
&\geq \left(\mu - \frac{\mu}{\zeta} - \frac{2\mu}{\sigma_4}\right)\| \bz_{t+1} - \bz_t\|^2 - \mu\zeta \| \bar{\bx}^*(\bz_t) - \bx^*(y_{t+1}, \bz_t)\|^2.
\end{align*}
As a result, we get
\begin{align}
    &\mathbb{E}[V_t]-\mathbb{E}[V_{t+1}]\notag \\
    &\geq \frac{\tau\lambda^2}{16}\mathbb{E}\|\bu^*(\bx_t, \by_{t+1}, \bz_{t})-\bx_t\|^2-(\lambda \tau \mu +2\lambda \tau^2\mu^2+\frac{\tau\lambda^2\mu^2}{8\gamma_s^2} + \frac{\mu}{\zeta} + \frac{3\mu}{\sigma_4} - \mu - \frac{\mu}{2\beta})\mathbb{E}\|\bz_t-\bz_{t+1}\|^2-\lambda \tau^2\sigma^2\notag\\
    &\quad -\eta\langle A\bx_t-\bb, A\bu^*(\bx_t, \by_t, \bz_t)-A\bx_t\rangle-\eta\|A\bx_t-A\bx^*(\by_{t+1}, \bz_t)\|^2+\eta\|A\bx^*(\by_{t+1}, \bz_t)-\bb\|^2\notag\\
    &\quad- \mu\zeta \| \bar{\bx}^*(\bz_t) - \bx^*(y_{t+1}, \bz_t)\|^2+\mu\langle \bz_{t+1}-\bz_t, \bu^*(\bx_t, \by_{t+1}, \bz_t)-\bx_{t}\rangle.\label{eq: smk4}
\end{align}
We will now operate on some of terms from the right-hand side of \eqref{eq: smk4}, by using Lemma \ref{lemma of error} and \ref{global error}.

First, we have by Cauchy-Schwarz and Young's inequalities that
\begin{align*}
    &-\eta\langle A\bx_t-\bb, A\bu^*(\bx_t, \by_t, \bz_t)-A\bx_t\rangle \\
    &\geq -\frac{\eta}{4}\|A\bx_t-\bb\|^2-\eta \|A\bu^*(\bx_t, \by_t, \bz_t)-A\bx_t\|^2\\
    &\geq -\frac{\eta}{4} \|A\bx_t-\bb\|^2 -2\eta \| A\bu^*(\bx_t, \by_t, \bz_t) -A\bu^*(\bx_t, \by_{t+1}, \bz_t)\|^2 - 2\eta \| A\bu^*(\bx_t, \by_{t+1}, \bz_t) - A\bx_t\|^2.
\end{align*}
Next, by using the Lipschitzness of $\bu^*(\bx_t, \cdot, \bz_t)$ from \eqref{s^*(x,y,z)-s^*(x,y',z)}, we have
\begin{align*}
\| A\bu^*(\bx_t, \by_t, \bz_t) -A\bu^*(\bx_t, \by_{t+1}, \bz_t)\|^2&\leq \|A\|^2 \| \bu^*(\bx_t, \by_t, \bz_t) -\bu^*(\bx_t, \by_{t+1}, \bz_t)\|^2 \\
&\leq \frac{\|A\|^4}{\gamma_s^2} \| \by_{t} - \by_{t+1}\|^2 \\
&= \frac{\|A\|^4\eta^2}{\gamma_s^2} \| A\bx_t - \bb\|^2,
\end{align*}
where the last step also used the definition of $\by_{t+1}$. Using this estimation along with \eqref{Ax_t-b} gives
\begin{align*}
    &-\eta\langle A\bx_t-\bb, A\bu^*(\bx_t, \by_t, \bz_t)-A\bx_t\rangle \\
    &\geq -(\frac{\eta}{4}+\frac{2\|A\|^4 \eta^3}{\gamma_s^2})\|A\bx_t-\bb\|^2-2\eta \|A\|^2\|\bu^*(\bx_t, \by_{t+1}, \bz_t)-\bx_t\|^2\\
    &\geq -(\frac{\eta \|A\|^2\lambda^2}{2\gamma^2}+\frac{4\|A\|^6 \eta^3 \lambda^2}{\gamma^2 \gamma_s^2}+2\eta \|A\|^2)\|\bu^*(\bx_t, \by_{t+1}, \bz_t)-\bx_t\|^2\\
    &\quad -(\frac{\eta}{2}+\frac{4\|A\|^4\eta^3}{\gamma_s^2})E\|A\bx^*(\by_{t+1}, \bz_t)-\bb\|^2.
\end{align*}
We next have by Young's inequality that for any $\theta>0$:
\begin{align*}
    \mu\langle \bz_{t+1}-\bz_t, \bu^*(\bx_t, \by_{t+1}, \bz_t)-\bx_{t}\rangle &\geq -\frac{\mu}{4\theta}\|\bz_{t+1}-\bz_t\|^2-\theta \mu\|\bu^*(\bx_t, \by_{t+1}, \bz_t)-\bx_{t}\|^2.
\end{align*}
The inequality derived in \eqref{Ax_t-Ax(y_t+1,z_t)} directly implies
\begin{align*}
    -\eta\|A\bx_t-A\bx^*(\by_{t+1}, \bz_t)\|^2\geq -\frac{\eta \|A\|^2\lambda^2}{\gamma^2}\|\bx_t-\bu^*(\bx_t, \by_{t+1}, \bz_{t})\|^2.
\end{align*}
The key global error bound given in \Cref{global error} originally proved in \cite{zhang2022error} results in
\begin{align*}
    -6\mu\beta \|\bx^*(\by_{t+1}, \bz_t)-\bar{\bx}^*(\bz_t)\|^2\geq -6\mu\beta \bar{\sigma}^2\|A\bx^*(\by_{t+1}, \bz_t)-\bb\|^2.
\end{align*}
Combining these estimates lead to
\begin{align}
    &\mathbb{E}[V_t]-\mathbb{E}[V_{t+1}]\geq -(\lambda \tau \mu +2\lambda \tau^2\mu^2+\frac{\tau\lambda^2\mu^2}{8\gamma_s^2} + \frac{\mu}{\zeta} + \frac{3\mu}{\sigma_4} - \mu - \frac{\mu}{2\beta} + \frac{\mu}{4\theta})\mathbb{E}\|\bz_t-\bz_{t+1}\|^2-\lambda \tau^2\sigma^2\notag\\
    &\quad+\left( \frac{\tau\lambda^2}{16} - \frac{3\|A\|^2\lambda^2\eta}{2\gamma^2} - \frac{4\|A\|^6\eta^3\lambda^2}{\gamma_s^2\gamma^2} - 2\eta\|A\|^2 - \mu\theta \right)\mathbb{E}\|\bu^*(\bx_t, \by_{t+1}, \bz_{t})-\bx_t\|^2\notag \\
    &\quad +\left(\frac{\eta}{2} - \frac{4\|A\|^4\eta^3}{\gamma_s^2} - 6\mu\beta\bar{\sigma}^2 \right)\|A\bx^*(\by_{t+1}, \bz_t)-\bb\|^2.\label{eq: smk8}
\end{align}
We now estimate the coefficients inside the parantheses, with straightforward but tedious calculations which follow from the parameter settings. 

First, we estimate the coefficient of $\mathbb{E}\|\bz_t-\bz_{t+1}\|^2$ in \eqref{eq: smk8}:
Let $\mu\geq 4L_f$, we have 
$\sigma_4 \geq \frac{1}{2}$
because $\sigma_4 = \frac{\mu-L_f}{\mu}$. Then letting $\zeta=6\beta, \beta < \frac{1}{30}$, we have
$$\mu-\frac{3\mu}{\sigma_4}\geq -5\mu\geq -\frac{\mu}{6\beta},\quad \frac{\mu}{\zeta}=\frac{\mu}{6\beta}.$$

Therefore,
\begin{equation}
    \frac{\mu}{2\beta}+\mu-\frac{3\mu}{\sigma_4}-\frac{\mu}{\zeta}\geq (\frac{1}{2}-\frac{1}{6}-\frac{1}{6})\frac{\mu}{\beta}\geq\frac{\mu}{6\beta}.
\end{equation}

Hence,
$$ \text{coefficient of\;}\mathbb{E}\|\bz_t-\bz_{t+1}\|\geq -\lambda \tau \mu-2\lambda \tau^2\mu^2-\frac{\tau\lambda^2\mu^2}{8\gamma_s^2}+\frac{\mu}{6\beta}-\frac{\mu}{8\beta}.$$

Let $\eta=\frac{\eta'}{2\|A\|^2}, \theta=2\beta,\eta'\leq\frac{1}{40}$, and $\mu=\max\{2,4L_f\},\lambda=L_K=L_f+\rho \|A\|+\mu, \tau\leq \frac{1}{10\lambda^2}$, and $\gamma_s = \mu-L_f+\gamma$ from \Cref{fact: params}.
We have $-\lambda \tau \mu \geq -\frac{\mu}{10}$ and $-2\lambda \tau^2\mu^2\geq -\frac{\mu}{50}$, then 

$$ \text{coefficient of\;}\mathbb{E}\|\bz_t-\bz_{t+1}\|\geq \frac{\mu}{24\beta}-\frac{\mu}{10}-\frac{\mu}{50}-\tau\lambda^2\frac{\mu^2}{(\mu-L_f+\lambda)^2}.$$

By $\beta \leq 1/30$, we have $\frac{1}{24\beta}-\frac{1}{10}-\frac{1}{50}\geq \frac{1}{30\beta}$. In addition, using $\tau \lambda^2\frac{\mu^2}{(\mu-L_f+\lambda)^2}\leq \tau \lambda^2\leq \frac{1}{10}$, we fanally obtain:
\begin{equation}\label{eq: est_coefficient_z}
\text{coefficient of\;}\mathbb{E}\|\bz_t-\bz_{t+1}\|\geq \frac{\mu}{30\beta}-\frac{1}{10}\overset{\mu\geq 2}{\geq } \frac{\mu}{50\beta}.
\end{equation}

Then we estimate the coefficient of $\mathbb{E}\|\bu^*(\bx_t, \by_{t+1}, \bz_{t})-\bx_t\|^2$ in \eqref{eq: smk8}.

From above assumptions, we can easily get $\gamma=\frac{(p-L_f)\lambda}{p-L_f+\lambda}\geq \frac{1}{2}$ because $\lambda \geq \mu \geq 2$. Moreover, we assume $\eta'\leq \frac{\tau}{40},\frac{\eta'}{\mu-L_f+\lambda}\leq\frac{\tau}{10}, \beta\leq \frac{\tau}{40}$
First, by our new notations, we have
$$\text{coefficient of\;} \mathbb{E}\|\bu^*(\bx_t, \by_{t+1}, \bz_{t+1})-\bx_t\|^2=\frac{\tau\lambda^2}{16}-\frac{3\eta' \lambda^2}{4\gamma^2}-\frac{\eta'^3 \lambda^2}{2\gamma^2\gamma_s^2}-\eta'-2\mu\beta$$

By $\gamma\geq \frac{1}{2}$ and the definition of $\gamma_s$, we have $-\frac{3\eta' \lambda^2}{4\gamma^2}\geq -3\eta'\lambda^2$, $-\frac{\eta'^3 \lambda^2}{2\gamma^2\gamma_s^2}\geq \frac{\eta'^2\lambda^2}{(\mu-L_f+\lambda)^2}$, Then
$$\text{coefficient of\;} \mathbb{E}\|\bu^*(\bx_t, \by_{t+1}, \bz_{t+1})-\bx_t\|^2\geq \frac{\tau\lambda^2}{16}-3\eta' \lambda^2-\frac{2\eta'^3 \lambda^2}{(\mu-L_f+\lambda)^2}-\eta'-2\mu\beta $$

With $2\leq \mu\leq \lambda$, $\eta'\leq \frac{\tau}{100},\frac{\eta'}{\mu-L_f+\lambda}\leq\frac{\tau}{10}, \beta\leq \frac{\tau}{200}$, we can obtain $-3\eta'\lambda^2 \geq -\frac{3\tau\lambda^2}{400}$, $-\frac{2\eta'^3 \lambda^2}{(\mu-L_f+\lambda)^2}\geq -\frac{\lambda^2\tau^2}{400}\geq -\frac{\lambda^2\tau}{400} $, $-\eta'\geq -\frac{\tau}{100}\geq -\frac{\tau \lambda^2}{100}$, $-2\mu \beta \geq \frac{\tau \mu}{50}\overset{\mu \leq \lambda}{\geq}-\frac{\tau \lambda}{50}\geq \frac{\tau \lambda^2}{100} $. Hence,
\begin{equation} \label{eq: est_coefficient_u}
    \text{coefficient of\;} \mathbb{E}\|\bu^*(\bx_t, \by_{t+1}, \bz_{t+1})-\bx_t\|^2\geq \frac{\tau\lambda^2}{16}-\frac{3\tau\lambda^2}{100}-\frac{\tau\lambda^2}{400}-\frac{\tau\lambda^2}{400}-\frac{\tau\lambda^2}{100}=\frac{7\tau \lambda^2}{400}
\end{equation}

Last, we estimate the coefficient of $\mathbb{E}\|A\bx^*(\by_{t+1}, \bz_t)-\bb\|^2$ in \eqref{eq: smk8}:

By $6\mu\beta \bar{\sigma}^2\leq \frac{\eta}{6}$ and the definition $\eta',\gamma_s$, we have $-\frac{4\|A\|^2\eta^3}{\gamma_s^2}=-\frac{\eta'^2\eta}{(\mu-L_f+\lambda)^2}\overset{\frac{\eta'}{\mu-L_f+\lambda}\leq\frac{\tau}{10}}{\geq}-\frac{\eta \tau^2}{100}\geq -\frac{\eta}{100}$ and $-6\mu\beta \bar{\sigma}^2\geq -\frac{\eta}{6}$. Hence, we have
\begin{equation}\label{eq: est_coefficient_Ax}
    \text{coefficient of\;}\mathbb{E}\|A\bx^*(\by_{t+1}, \bz_t)- \bb\|^2\geq \frac{\eta}{4}.
\end{equation}

Plug \ref{eq: est_coefficient_z}, \ref{eq: est_coefficient_u} and \ref{eq: est_coefficient_Ax} to \ref{eq: smk8}, we finishe the proof.
\end{proof}
\subsection{Proof of Theorem \ref{th: moreau_env_det}}\label{appendix: complexity_moreau}
\begin{proof}
We start from the result in \Cref{lem: appendix_descent}. First, it follows from the definition of $\bz_{t+1}$ that
\begin{equation*}
\|\bz_t-\bz_{t+1}\|=\beta\|\bx_{t}-\bz_t\|.
\end{equation*}
So, we rewrite (\ref{desent lemma}), as:
\begin{equation}\label{eq: siw3}
     \mathbb{E}V_t-\mathbb{E}V_{t+1}\geq \beta^2 c_\beta\mathbb{E}\|\bx_{t}-\bz_t\|^2+c_\tau\mathbb{E}\|\bu^*(\bx_t, \by_{t+1}, \bz_{t})-\bx_t)\|^2+c_\eta\mathbb{E}\|A\bx^*(\by_{t+1}, \bz_t)-\bb\|^2-\lambda \tau^2\sigma^2.
\end{equation}

For $t>0$, we have $V_t\geq\underline{f}$, which is proven in \Cref{lem: lb}. It then follows that
\begin{equation}\label{eq: hfr4}
\sum_{t=0}^{T-1}(\mathbb{E}V_t-\mathbb{E}V_{t+1})=\mathbb{E}V_0-\mathbb{E}V_{T}\leq\mathbb{E}V_0-\underline{f}.
\end{equation}

Then, summing up \eqref{eq: siw3}, using \eqref{eq: hfr4}, and the fact that $c_\tau = \Theta(\tau)$, $c_\eta = \Theta(\tau)$, $\beta^2 c_\beta = \Theta(\tau)$ from \eqref{eq: param_defs_appendix}, we have
\begin{equation*}
V_{0}-\underline f+T\lambda \tau^2\sigma^2 \geq \sum_{t=1}^TC_0\tau\left[\mathbb{E}\|\bx_{t}-\bz_t\|^2+\mathbb{E}\|\bu^*(\bx_t,\by_{t+1},\bz_{t})-\bx_t\|^2+\mathbb{E}\|A\bx^*(\by_{t+1},\bz_t)-\bb\|^2\right],
\end{equation*}
for some explicit constant $C_0$.

Dividing both sides by $T$, rearranging and using the definition $\tau =\frac{1}{6\lambda\sqrt{T}}$ gives
\begin{equation}\label{eq: asn4}
    \frac{1}{T}\sum_{t=0}^{T-1}\mathbb{E}\|\bx_{t}-\bz_t\|^2+\mathbb{E}\|\bu^*(\bx_t,\by_{t+1},\bz_{t})-\bx_t\|^2+\mathbb{E}\|A\bx^*(\by_{t+1},\bz_t)-\bb\|^2\leq \frac{1}{C_0\sqrt{T}}\left(6\lambda(V_0-\underline{f}) + \frac{\sigma^2}{6}\right).
\end{equation}

Since we have 
\begin{equation*}
\nabla \Psi(\bz_t) = \mu(\bz_t - \bar{\bx}^*(\bz_t)),
\end{equation*}
by Danskin's theorem, we deduce for any $t$
        \begin{align*}
            \frac{1}{\mu^2}\|\nabla\Psi(\bz_t)\|&=\|\bz_t-\bar{\bx}^*(\bz_t)\| \\
            &\leq \|\bz_t-\bx^*(\by_{t+1},z_t)\|+\|\bx^*(\by_{t+1}, \bz_t)-\bar{\bx}^*(\bz_t)\|\\
            &\leq \|\bz_t-\bx^*(\by_{t+1},\bz_t)\|+\bar{\sigma}\|A\bx^*(\by_{t+1},\bz_s)-\bb\|\\
            &\leq \|\bz_t-\bx_{t}\|+\|\bx_t-\bx^*(\by_{t+1},\bz_t)\|+\bar{\sigma}\|A\bx^*(\by_{t+1},\bz_t)-\bb\| \\
            &\leq \|\bz_t-\bx_{t}\|+\frac{\lambda}{\gamma}\|\bx_t-\bu^*(\bx_t, \by_{t+1},\bz_t)\|+\bar{\sigma}\|A\bx^*(\by_{t+1},\bz_t)-\bb\|.
        \end{align*}
        where the first inequality is by triangle inequality, the second by \eqref{global error}, the third by triangle inequality and the fourth by \eqref{x-s*}.
        
        Next, we take square of both sides, take expectation, use Young's inequality, sum for all $t=0, \dots, T-1$, divide by $T$ and use \eqref{eq: asn4} to derive
        \begin{align*}
            \frac{1}{\mu^2}\frac{1}{T}\sum_{t=1}^T \mathbb{E}\|\nabla\Psi(\bz_t)\|^2 &\leq \frac{1}{T}\sum_{t=1}^T \mathbb{E}\left[ 3\|\bz_t-\bx_{t}\|^2 + \frac{3\lambda^2}{\gamma^2}\|\bx_t-\bu^*(\bx_t, \by_{t+1},\bz_t)\|^2 + 3\bar{\sigma}^2 \|A\bx^*(\by_{t+1},\bz_t)-\bb\|^2\right] \\
            &=O\left(\frac{1}{\sqrt{T}}\right).
        \end{align*}
The result then follows since $t^*$ is selected uniformly at random from $\{1, 2, \dots, T\}$.
\end{proof}
\subsection{Proof of \Cref{cor: postprocess}}\label{subsec: postprocess}
\begin{proof}
    From the definition of $\hat{x}$, we have

    $$0\in \hat{G}(\bx_t,\by_{t+1},\bz_t)+\frac{2}{\tau}(\hat{\bx}-\bx_t)+\partial I_X(\hat{\bx}).$$
Let us set
    \begin{equation}\label{eq: ssa3}
    \bv=\nabla_\bx K(\hat{\bx},\by_{t+1},\bz_t)-\hat{G}(\bx_t,\by_{t+1},\bz_t)-\frac{2}{\tau}(\hat{\bx}-\bx_t)-\rho A^T(A\hat{\bx}-\bb)-\mu(\hat{\bx}-\bz_t).
    \end{equation}
    Combining with the optimality condition, we have
    \begin{align*}
    \bv&\in \nabla_\bx K(\hat{\bx},\by_{t+1},\bz_t)-\rho A^T(A\hat{\bx}-\bb)-\mu(\hat{\bx}-\bz_t)+\partial I_X(\hat{\bx})  \\
&=  \nabla f(\hat{\bx})+A^T\by_{t+1}+\partial I_X(\hat{\bx}).
    \end{align*}

Hence, we need to estimate $\mathbb{E}\|A\hat{\bx}-\bb\|$ and $\mathbb{E}\|\bv\|$.
    
For the mini-batch gradient in the post-processing step, we have
\begin{equation}
\mathbb{E}\| \hat{G}(\bx, \by, \bz) - \nabla K(\bx, \by, \bz) \|^2 \leq \frac{\sigma^2}{B}.
\end{equation}
which is a standard calculation \citep[Section 5.2.3]{lan2020first}. Since $B=\Omega(\varepsilon^{-2})$, this gives us
\begin{equation}\label{eq: asl4}
\mathbb{E}\| \hat{G}(\bx, \by, \bz) - \nabla K(\bx, \by, \bz)\|^2 \leq \varepsilon^2.
\end{equation}

First, let us note that the purpose of $\hat{\bx}$ is to estimate $\bu^*(\bx_t, \by_{t+1}, \bz_t)$, where
\begin{equation*}
\bu^*(\bx_t, \by_{t+1}, \bz_t) =\arg\min_{\bu\in X} \{l(\bu):= K(\bu, \by_{t+1}, \bz_t) + \frac{\lambda}{2} \| \bx_t - \bu\|^2\}.
\end{equation*}

Note that the gradient of this objective is
\begin{equation*}
\nabla l(\bu) = \nabla_\bx K(\bx, \by_{t+1}, \bz_t) + \lambda (\bx-\bx_t).
\end{equation*}
As a result, we have $\nabla l(\bx_t) = \nabla_\bx K (\bx_t, \by_{t+1}, \bz_t)$.
 Let us also denote
\begin{equation*}
\bx^*_t = \proj_X(\bx_t - \tau \nabla l(\bx_t)).
\end{equation*}
That is, $\bx^*_t$ is the output of doing a full-gradient step on $\bx_t$. Of course, this is not tractable in our setting, but we only use this as a theoretical tool.

Since this is a GD step on the objective $l$ which is $L_K$-smooth and convex with optimizer $\bu^*(\bx_t, \by_{t+1}, \bz_t)$, the standard analysis for GD gives 
\begin{align}\label{eq: asx3}
\|\bx^*_t - \bu^*(\bx_t, \by_{t+1}, \bz_t) \|^2 \leq \|\bx_t - \bu^*(\bx_t, \by_{t+1}, \bz_t) \|^2,
\end{align}
as long as $\tau \leq \frac{1}{L_K}$.

Next, by the definitions of $\bx_t^*$ and $\hat\bx$, along with nonexpansiveness of the projection, we have
\begin{align}
\mathbb{E}\| \bx^*_t - \hat\bx \|^2 &\leq \mathbb{E}\tau^2 \| \hat G(\bx_t, \by_{t+1}, \bz_{t+1}) - \nabla_\bx K(\bx_t, \by_{t+1}, \bz_t)\|^2 \notag \\
&\leq \tau^2\varepsilon^2,\label{eq: asa3}
\end{align}
where the second inequality used \eqref{eq: asl4}.

In view of \eqref{eq: ssa3}, we estimate $\|\bv\|$ as
\begin{align*}
\|\bv\| \leq \| \nabla_x K(\bx_t, \by_{t+1}, \bz_t) - \hat{G}(\bx_t, \by_{t+1}, \bz_t)\| + L_K\| \bx_t-\hat \bx\| + \frac{2}{\tau} \| \hat \bx- \bx_t\| + \rho \| A\| \|A\hat\bx-\bb\| + \mu \|\hat\bx - \bz_t\|.
\end{align*}
On this, multiple applications of triangle inequality gives
\begin{align}
\| \hat\bx - \bx_t \| &\leq \| \hat\bx -\bx^*_t \| + \| \bx_t^* - \bu^*(\bx_t, \by_{t+1}, \bz_t) \| + \| \bu^*(\bx_t, \by_{t+1}, \bz_t)-\bx_t \|\notag  \\
&\leq \| \hat\bx -\bx^*_t \| + 2 \| \bu^*(\bx_t, \by_{t+1}, \bx_t) - \bx_t \|,\label{eq: axze4}
\end{align}
where the second line is due to \eqref{eq: asx3}.

For the feasibility, we have by triangle inequality that
\begin{equation}
\|\hat\bx-\bz_t\| \leq \|\hat\bx-\bx_t \| + \| \bx_t - \bz_t\|.
\end{equation}

As a result, we have that
\begin{align}\label{eq: awq4}
\|\bv\| &= O\big( \|\hat\bx-\bx_t^*\| + \|\bx_t - \bu^*(\bx_t, \by_{t+1}, \bz_t) \| + \| A\hat\bx - \bb\| + \| \bx_t-\bz_t\| \notag \\
&\qquad+ \| \nabla_x K(\bx_t, \by_{t+1}, \bz_t) - \hat{G}(\bx_t, \by_{t+1}, \bz_t)\| \big).
\end{align}
For the feasibility, we have
\begin{align*}
\|A\hat\bx - \bb \| &\leq \|A\hat\bx - A\bx_t \| + \|A\bx_t - \bb\| \\
&\leq \|A\| \|\hat\bx - \bx_t\| + \|A\bx_t-\bb\|.
\end{align*}
Now, by invoking the above inequality for $t=t^*$, taking expectation, using Young's inequality, \eqref{eq: axze4}, \eqref{eq: asa3} and \eqref{eq: asn4} along with \eqref{Ax_t-b}, we get that
\begin{equation}\label{eq: asxf4}
\mathbb{E}\|A\hat\bx-\bb\|^2\leq\varepsilon^2,
\end{equation}
since $T=\Omega(\varepsilon^{-4})$.

Finally, using $t=t^*$, taking square and then expectation of \eqref{eq: awq4}, using Young's inequality and then combining \eqref{eq: asxf4}, \eqref{eq: asa3}, \eqref{eq: asl4} and \eqref{eq: asn4} gives the result since $T=\Omega(\varepsilon^{-4})$.
\end{proof}

\subsection{Auxiliary Results}\label{app: aux}

\begin{lemma} \label{error bound}
Under \Cref{asmp: 1},    for any $ \bx,\bz,\bz'\in X$, we have
    \begin{align}
        \frac{\lambda}{\gamma}\|\bx - \bu^*(\bx,\by,\bz)\| &\geq \|\bx - \bx^*(\by,\bz)\|, \label{x-s*}\\ 
        \|\bu^*(\bx,\by,\bz)-\bx\|&\leq \|\bx-\bx^*(\by,\bz)\|,\label{s^*-x(y,z)}\\ 
        \|\bu^*(\bx,\by,\bz)-\bu^*(\bx,\by',\bz)\|&\leq \frac{\|A\|}{\gamma_s}\|\by-\by'\|,\label{s^*(x,y,z)-s^*(x,y',z)}\\ 
        \|\bu^*(\bx,\by,\bz')-\bu^*(\bx,\by,\bz')\|&\leq \frac{\mu}{\gamma_s}\|\bz-\bz'\|,\label{s^*(x,y,z')-s^*(x,y,z)}\\
        \|\bz'-\bz\|&\geq \frac{\mu-L_f}{\mu}\|\bx^*(\by,\bz')-\bx^*(\by,\bz)\|, \label{x(y,z')-x(y,z)}\\
        \|\by'-\by\|&\geq \frac{\gamma_K}{\|A\|}\|\bx^*(\by',\bz)-\bx^*(\by,\bz)\|, \label{x(y',z)-x(y,z)}\\
        \|\bar{\bx}^*(\bz)-\bar{\bx}^*(\bz')\| &\leq \frac{\mu}{\mu-L_f}\|\bz-\bz'\|  \label{x(z)-x(z')}\\
    \end{align}
where $\gamma=\frac{(\mu-L_f)\lambda}{\mu-L_f+\lambda}, \gamma_s=\mu-L_f+\lambda, \gamma_K=\mu-L_f$.
\end{lemma}
    \begin{proof}
    The proofs for \eqref{x(y,z')-x(y,z)}, \eqref{x(y',z)-x(y,z)}, and \eqref{x(z)-x(z')} appear in \cite{zhang2022error}, so we omit these proofs.
    
    We first prove \eqref{x-s*}. Let us note that $\bx^*(\by, \bz)$ minimizes $\varphi_{1/\lambda}$, see for example \citep[Theorem XV4.1.7]{hiriartconvex}. As a result, we have $\nabla_\bx\varphi_{1/\lambda}(\bx^*(\by,\bz),\by,\bz)=0$. From Lemma \ref{lem: sc_moreau}, we have that $\varphi_{1/\lambda}(\cdot,y,z)$ is $\gamma=\frac{(\mu-L_f)\lambda}{\mu-L_f+\lambda}$-strongly convex.

        Then, by strong convexity, we have 
        \begin{align*}
        &\langle \nabla_\bx\varphi_{1/\lambda}(\bx^*(\by,\bz),\by,\bz)- \nabla_\bx\varphi_{1/\lambda}(\bx,\by,\bz),\bx^*(\by,\bz)-\bx \rangle\geq \gamma\|\bx-\bx^*(\by,\bz)\|^2 \\
        \iff & \| \nabla_\bx \varphi_{1/\lambda}(\bx, \by, \bz) \| \geq \gamma \|\bx-\bx^*(\by,\bz)\|,
        \end{align*}
        where the inclusion used $\nabla_\bx\varphi_{1/\lambda}(\bx^*(\by,\bz),\by,\bz)=0$ established in the previous paragraph as well as Cauchy-Schwarz inequality. Then, using $\nabla_\bx\varphi_{1/\lambda}(\bx, \by, \bz) = \lambda (\bx - \bu^*(\bx, \by, \bz))$, we obtain \eqref{x-s*}.

    From definition of $\bu^*(\bx, \by, \bz)$ in \eqref{eq: aeh4} , we have, 
    \begin{equation*}
    K(\bu^*(\bx,\by,\bz),\by,\bz)+\frac{\lambda}{2}\|\bx-\bu^*(\bx,\by,\bz)\|^2\leq K(\bx^*(\by,\bz),\by,\bz)+\frac{\lambda}{2}\|\bx-\bx^*(\by,\bz)\|^2,
    \end{equation*}
    where we also remark that $\bx^*(\by, \bz) \in X$. Combining with $K(\bx^*(\by,\bz),\by,\bz)\leq K(\bu^*(\bx,\by,\bz),\by,\bz)$, which follows from the definition of $\bx^*(\by, \bz)$ in \eqref{eq: x_def}, we have (\ref{s^*-x(y,z)}).

The proofs of the other two assertions will use a similar idea to \cite{zhang2022error}, but there will be differences in the estimations since this previous work did not use the function $\varphi_{1/\lambda}$.

    For \eqref{s^*(x,y,z)-s^*(x,y',z)}, we proceed by using the definition of $\varphi_{1/\lambda}$ and adding and subtracting $K(\bu^*(\bx, \by', \bz), \by, \bz)$ to get
    \begin{align*}
    & K(\bu^*(\bx,\by,\bz),\by,\bz) + \frac{\lambda}{2}\|\bu^*(\bx,\by,\bz) -\bx\|^2 - K(\bu^*(\bx,\by',\bz),\by',\bz) - \frac{\lambda}{2}\|\bu^*(\bx,\by', \bz) - \bx\|^2 \\
    = & \, K(\bu^*(\bx,\by,\bz),\by,\bz) + \frac{\lambda}{2}\|\bu^*(\bx,\by,\bz) - \bx\|^2 \\
    & \quad - K(\bu^*(\bx,\by',\bz),\by,\bz) - \frac{\lambda}{2}\|\bu^*(\bx,\by',\bz) - \bx\|^2 \\
    & \quad + K(\bu^*(\bx,\by',\bz),\by,\bz) - K(\bu^*(\bx,\by',\bz),\by',\bz) \\
    \leq & \, \frac{-\gamma_s}{2}\|\bu^*(\bx,\by,\bz) - \bu^*(\bx,\by',\bz)\|^2 
    + \langle \by - \by', A\bu^*(\bx,\by',\bz) - \bb \rangle.
    \end{align*}

The last step uses $\bu\mapsto K(\bu, \by, \bz) + \frac{\lambda}{2} \| \bu-\bx\|^2$ being $\gamma_s$-strongly convex (cf. \Cref{fact: params}) with minimizer $\bu^*(\bx, \by, \bz)$, as well as the definition of $K$. 

We then argue similarly, this time adding and subtracting $K(\bu^*(\bx, \by, \bz), \by', \bz)$:
    \begin{align*}
    & K(\bu^*(\bx,\by,\bz),\by,\bz) + \frac{\lambda}{2}\|\bu^*(\bx,\by,\bz) - \bx\|^2- K(\bu^*(\bx,\by',\bz),\by',\bz) - \frac{\lambda}{2}\|\bu^*(\bx,\by',\bz) - \bx\|^2 \\
    = & \, K(\bu^*(\bx,\by,\bz),\by',\bz) + \frac{\lambda}{2}\|\bu^*(\bx,\by,\bz) - \bx\|^2 \\
    & \quad - K(\bu^*(\bx,\by',\bz),\by',\bz) - \frac{\lambda}{2}\|\bu^*(\bx,\by',\bz) - \bx\|^2 \\
    & \quad - K(\bu^*(\bx,\by,\bz),\by',\bz) + K(\bu^*(\bx,\by,\bz),\by,\bz) \\
    \geq & \, \frac{\gamma_s}{2}\|\bu^*(\bx,\by,\bz) - \bu^*(\bx,\by',\bz)\|^2 
    + \langle \by - \by', A\bu^*(\bx,\by,\bz) - \bb \rangle.
\end{align*}
The last step uses that $\bu\mapsto K(\bu, \by', \bz) + \frac{\lambda}{2} \| \bu-\bx\|^2$ is $\gamma_s$-strongly convex (cf. \Cref{fact: params}) with minimizer $\bu^*(\bx, \by', \bz)$ and the definition of $K$.

Combining the last two estimates give
\begin{equation*}
\langle \by-\by',A\bu^*(\bx,\by',\bz)-A\bu^*(\bx,\by,\bz) \rangle \geq \gamma_s\|\bu^*(\bx,\by,\bz)-\bu^*(\bx,\by',\bz)\|^2.
\end{equation*} 
Using Cauchy-Schwarz inequality and the definition of operator norm gives \eqref{s^*(x,y,z)-s^*(x,y',z)}.

The proof of \eqref{s^*(x,y,z')-s^*(x,y,z)} is similar to the proof of \eqref{s^*(x,y,z)-s^*(x,y',z)}, just completed. In particular, by adding and subtracting $K(\bu^*(\bx, \by, \bz), \by, \bz')$, we have
\begin{align*}
    &K(\bu^*(\bx,\by,\bz),\by,\bz)+\frac{\lambda}{2}\|\bu^*(\bx,\by,\bz)-\bx\|^2
    - K(\bu^*(\bx,\by,\bz'),\by,\bz')+\frac{\lambda}{2}\|\bu^*(\bx,\by,\bz')-\bx\|^2\\
    &= K(\bu^*(\bx,\by,\bz),\by,\bz)+\frac{\lambda}{2}\|\bu^*(\bx,\by,\bz)-\bx\|^2- K(\bu^*(\bx,\by,\bz'),\by,\bz)-\frac{\lambda}{2}\|\bu^*(\bx,\by,\bz')-\bx\|^2\\
    &-K(\bu^*(\bx,\by,\bz'),\by,\bz')+K(\bu^*(\bx,\by,\bz'),\by,\bz)\\
    &\leq -\frac{\gamma_s}{2}\|\bu^*(\bx,\by,\bz)-\bu^*(\bx,\by,\bz')\|^2+\frac{\mu}{2}(\|\bu^*(\bx,\by,\bz')-\bz\|^2-\|\bu^*(\bx,\by,\bz')-\bz'\|^2),
\end{align*}
where we used that $\bu \mapsto K(\bu, \by, \bz) + \frac{\lambda}{2} \| \bu-\bx\|^2$ is $\gamma_s$-strongly convex with minimizer $\bu^*(\bx, \by, \bz)$ and the definition of $K$.

Finally, we add and subtract $K(\bu^*(\bx, \by, \bz'), \by, \bz)$ to get
\begin{align*}
    &K(\bu^*(\bx,\by,\bz),\by,\bz)+\frac{\lambda}{2}\|\bu^*(\bx,\by,\bz)-\bx\|^2
    - K(\bu^*(\bx,\by,\bz'),\by,\bz')-\frac{\lambda}{2}\|\bu^*(\bx,\by,\bz')-\bx\|^2\\
    &= K(\bu^*(\bx,\by,\bz),\by,\bz')+\frac{\lambda}{2}\|\bu^*(\bx,\by,\bz)-\bx\|^2- K(\bu^*(\bx,\by,\bz'),\by,\bz)-\frac{\lambda}{2}\|\bu^*(\bx,\by,\bz')-\bx\|^2\\
    &+K(\bu^*(\bx,\by,\bz),\by,\bz)-K(\bu^*(\bx,\by,\bz),\by,\bz')\\
    &\geq \frac{\gamma_s}{2}\|\bu^*(\bx,\by,\bz)-\bu^*(\bx,\by,\bz')\|^2+\frac{\mu}{2}(\|\bu^*(\bx,\by,\bz)-\bz\|^2-\|\bu^*(\bx,\by,\bz)-\bz'\|^2),
\end{align*}
where we used that $\bu \mapsto K(\bu, \by, \bz') + \frac{\lambda}{2} \| \bu-\bx\|^2$ is $\gamma_s$-strongly convex with minimizer $\bu^*(\bx, \by, \bz')$ and the definition of $K$. 

Combining the last two inequalities give
\begin{equation*}
\mu\langle \bu^*(\bx,\by,\bz')-\bu^*(\bx,\by,\bz),\bz'-\bz\rangle \geq \gamma_s\|\bu^*(\bx,\by,\bz)-\bu^*(\bx,\by,\bz')\|^2.
\end{equation*}
Using Cauchy-Schwarz inequality concludes the proof.
\end{proof}

\begin{lemma} \label{lemma of error}
Under \Cref{asmp: 1}, we have that
    \begin{align}
        \|A\bx_t-A\bx^*(\by_{t+1},\bz_t)\|^2 &\leq \frac{\|A\|^2\lambda^2}{\gamma^2}\|\bx_t-\bu^*(\bx_t,\by_{t+1},\bz_{t})\|^2, \label{Ax_t-Ax(y_t+1,z_t)} \\
    \|A\bx_t-\bb\|^2&\leq \frac{2\|A\|^2\lambda^2}{\gamma^2}\|\bx_t-\bu^*(\bx_t, \by_{t+1}, \bz_{t})\|^2+2\|A\bx^*(\by_{t+1},\bz_t)-\bb\|^2,\label{Ax_t-b}\\
        \|A\bu^*(\bx_t, \by_t, \bz_t)-A\bx_t\|^2 &\leq  \frac{2 \|A\|^4}{\gamma_s^2}\|\by_t-\by_{t+1}\|^2+2\|A\|^2\|\bu^*(\bx_t, \by_{t+1}, \bz_t)-\bx_t\|^2, \label{As^*(x_t,y_t,z_t)-Ax_t}
    \end{align}
    where $\gamma, \gamma_s$ are defined in \eqref{eq: param_defs_appendix}. 
\end{lemma}

\begin{proof}

    The assertion in (\ref{Ax_t-Ax(y_t+1,z_t)}) follows directly from \eqref{x-s*} since
    \begin{equation*}
    \|A\bx_t-A\bx^*(\by_{t+1},\bz_t)\|^2 \leq \|A\|^2\|\bx_t-\bx^*(\by_{t+1},\bz_t)\|^2\leq \frac{\|A\|^2\lambda^2}{\gamma^2}\|\bx_t-\bu^*(\bx_t,\by_{t+1},\bz_{t})\|^2.
    \end{equation*}
Combining the first assertion with Young's inequality gives the second assertion, since
    \begin{align*}
    \|A\bx_t-\bb\|^2&\leq 2\|A\bx_t-A\bx^*(\by_{t+1},\bz_t)\|^2+2\|A\bx^*(\by_{t+1},\bz_t)-\bb\|^2 \\
    &\leq \frac{2\|A\|^2\lambda^2}{\gamma^2}\|\bx_t-\bu^*(\bx_t, \by_{t+1}, \bz_{t})\|^2+2\|A\bx^*(\by_{t+1},\bz_t)-\bb\|^2.
    \end{align*}
Young's inequality and \eqref{s^*(x,y,z)-s^*(x,y',z)} gives the third assertion
    \begin{align*}
        \|A\bu^*(\bx_t, \by_t, \bz_t)-A\bx_t\|^2 &\leq 2\|A\bu^*(\bx_t,\by_t,\bz_t)-A\bu^*(\bx_t,\by_{t+1},\bz_t)\|^2+2\|A\bu^*(\bx_t,\by_{t+1},\bz_t)-A\bx_t\|^2\\
                                    &\leq \frac{2 \|A\|^4}{\gamma_s^2}\|\by_t-\by_{t+1}\|^2+2\|A\|^2\|\bu^*(\bx_t, \by_{t+1}, \bz_t)-\bx_t\|^2.
    \end{align*}
The completes the proof.
\end{proof}
The following important lemma is known as the global error bound in \cite{zhang2022error}. This global result holds in its entirety in our case, so we only state it here and refer to where it appeared originally for the precise definition of the constant $\bar{\sigma}$ which depends on Hoffman constant of certain linear systems.
\begin{lemma}\citep[Lemma 3.2]{zhang2022error} \label{global error}
    If $\mu>L_f$, then we have
    $$\|\bx^*(\by,\bz)-\bar{\bx}^*(\bz)\|\leq \bar{\sigma}\|A\bx^*(\by,\bz)-\bb\|  \text{~~~for any~~~} \by,\bz$$

    where $\bar{\sigma}>0$ depends only on the constants $C_1=(L_f+\rho\|A\|^2+\mu)$, $C_2=-L_f+\mu$, and the matrices $A, H$ and is always finite.
\end{lemma}

\begin{fact}\label{fact: params}
For $\bx \in X$, we have that  $\bx\mapsto K(\bx,\by,\bz)$ is strongly convex with modulus $\gamma_K=\mu-L_f$, and $\bx\mapsto \nabla_\bx K(\bx,\by,\bz)$ is $(L_f+\rho\|A\|^2+\mu)$-Lipschitz continuous.
    
For $\bu \in X$, $\bu\mapsto K(\bu,\by,\bz)+\frac{\lambda}{2}\|\bx-\bu\|^2$ is strongly convex with modulus $\gamma_s=\mu-L_f+\lambda$, and $\bu^*(\bx, \by, \bz)$ is the minimizer of $K(\cdot, \by, \bz)+I_X(\bu)+\frac{\lambda}{2} \|\bx-\bu\|^2$.
\end{fact}

\begin{lemma}
    \citep[Lemma 2.19]{planiden2016strongly}\label{lem: shawn_lemma} Let $r>0$. The function f is r-strongly convex if and only if $f_1(\bx)=\min_\bu f(\bu)+\frac{1}{2}\|\bx-\bu\|^2$ is $\frac{r}{r+1}$-strongly convex.
\end{lemma}

\begin{lemma}\label{lem: sc_moreau}
The function  $\bx\mapsto\varphi_{1/\lambda}(\bx,\by,\bz)$ is $\gamma=\frac{(\mu-L_f)\lambda}{\mu-L_f+\lambda}$-strongly convex.
\end{lemma}
\begin{proof}
By definition, we have
\begin{equation*}
\varphi_{1/\lambda}(\bx,\by,\bz)=\min_\bu \big\{K(\bu,\by,\bz)+I_X(\bu)+\frac{\lambda}{2}\|\bx-\bu\|^2\big\}=\lambda \min_\bu \Big\{\frac{K(\bu,\by,\bz)+I_X(\bu)}{\lambda}+\frac{1}{2}\|\bx-\bu\|^2\Big\}.
\end{equation*}
Recall that $\gamma_K = \mu-L_f$. Then,    since $K(\bx,\by,\bz)/\lambda$ is $\frac{\gamma_K}{\lambda}$-strongly convex, we have $\min_{\bu} \frac{K(\bu,\by,\bz)+I_X(\bu)}{\lambda}+\frac{1}{2}\|\bx-\bu\|^2$ is $\frac{\gamma_K/\lambda}{\gamma_K/\lambda+1}$-strongly convex, by \Cref{lem: shawn_lemma}.
    Hence, $\varphi_{1/\lambda}(\bx,\by,\bz)$ is strongly convex with modulus $\frac{ \gamma_K}{\gamma_K/\lambda+1}=\frac{\lambda \gamma_K }{\lambda+\gamma_K}=\frac{(\mu-L_f)\lambda}{\mu-L_f+\lambda}$.
\end{proof}

\begin{lemma}\label{lem: lb}
    If $\bx\in X$, we have $\varphi_{1/\lambda}(\bx,\by,\bz)-2d(\by,\bz)+2\Psi(\bz)\geq \underline{f}$.
    \end{lemma}
    \begin{proof}
Because $\bx^*(\by, \bz)$ minimizes $\varphi_{1/\lambda}(\cdot, \by, \bz)$ (see for example \citep[Theorem XV4.1.7]{hiriartconvex}), we have
\begin{equation*}
\varphi_{1/\lambda} (\bx, \by, \bz) \geq \varphi_{1/\lambda}(\bx^*(\by, \bz)) = K(\bx^*(\by, \bz)).
\end{equation*}
We can then deduce
        \begin{align*}
            \varphi_{1/\lambda}(x,y,z)-2d(y,z)+2\Psi(z)&\geq K(x(y,z),y,z)-2d(y,z)+2\Psi(z)\\
            &=d(y,z)-2d(y,z)+2\Psi(z)\\
            &=\Psi(z)+\Psi(z)-d(y,z)\\
            &\geq \Psi(z)\\
            &\geq \underline{f}
        \end{align*}

The second inequality in the above chain comes from definition, that is, denoting $\bx^*_{\mu} = \argmin_{x\in X, A\bx=\bb}\{f(\bx) + \frac{\mu}{2}\|\bx-\bz\|^2\}$ in view of \eqref{eq: moreau_orig}, we have
\begin{equation*}
d(\by, \bz) = \min_{\bx\in X} K(\bx, \by, \bz) \leq K(\bx^*_\mu, \by, \bz) = f(\bx^*_\mu) + \frac{\mu}{2} \| \bx^*_\mu-\bz\|^2 = \Psi(\bz),
\end{equation*}
where the first inequality also uses $\bx_\mu^*\in X$, which is by definition.
    \end{proof}

\section{Proofs for \Cref{sec: rand_sec}}\label{sec:appendix_stoc_const}
\subsection{Proof of Theorem \ref{thm:stochasitc linear constraints}}
\begin{lemma}\label{stochasitc d(y,z)}
Let \Cref{assumptions for stochasitc} hold. With the update rule of $\by_{t+1}=\by_t+\eta(A_\zeta \bx_t-\bb_\zeta)$, where $\mathbb{E}_\zeta[A_\zeta \bx_t-\bb_\zeta] = A\bx_t - \bb$, we have
        \begin{equation}
            \begin{aligned}
            \mathbb{E}d(\by_{t+1},\bz_{t+1})-\mathbb{E}d(\by_t.\bz_t) &\geq \eta \mathbb{E}\langle(A\bx_t-\bb), Ax(\by_{t+1},\bz_t)-b \rangle-\frac{\eta^2}{32}\mathbb{E}\|A\bx_t-\bb\|^2-(\frac{1}{2}+\frac{17\|A\|^2}{2\gamma_K^2})\eta^2L^2\\
            &\quad +\frac{\mu}{2}\mathbb{E}\langle \bz_{t+1}-\bz_{t}, \bz_{t+1}+\bz_{t}-2\bx(\by_{t+1},\bz_{t+1}) \rangle,\\
            \mathbb{E}\Psi(\bz_{t+1})-\mathbb{E}\Psi(\bz_t)&\leq \mu\mathbb{E}\langle \bz_{t+1}-\bz_t, \bz_t-\bar{\bx}^*(\bz_t)\rangle+\frac{\mu}{2\sigma_4}\mathbb{E}\|\bz_t-\bz_{t+1}\|^2
                    \end{aligned}
        \end{equation}
where $\gamma_K, \sigma_4$ are introduceed in \ref{error bound}, and we assume $\mathbb{E}\|A_{\zeta_t}\bx_t-\bb_{\zeta_t}\|^2\leq L$.
    \end{lemma}
    \begin{proof}

It is easy to derive, for example as \citep[Lemma 3.2]{zhang2020proximal}, that
\begin{equation*}
d(\by_{t+1}, \bz_{t+1}) - d(\by_t, \bz_t) \geq  \langle \by_{t+1} - \by_t, A \bx^*(\by_{t+1}, \bz_t) - \bb \rangle + \frac{\mu}{2} \langle \bz_{t+1} - \bz_t, \bz_{t+1} + \bz_t - 2 \bx^*(\by_{t+1}, \bz_{t+1}) \rangle.
\end{equation*}
Hence, by using the update rule of $\by_{t+1}$, we get
\begin{align*}
    d(\by_{t+1}, \bz_{t+1}) - d(\by_t, \bz_t) &\geq \langle \by_{t+1} - \by_t, A \bx^*(\by_{t}, \bz_t) - \bb \rangle + \langle \by_{t+1} - \by_t, A \bx^*(\by_{t+1}, \bz_t) - A \bx^*(\by_{t}, \bz_t) \rangle \\
    &\quad+ \frac{\mu}{2} \langle \bz_{t+1} - \bz_t, \bz_{t+1} + \bz_t - 2 \bx^*(\by_{t+1}, \bz_{t+1}) \rangle \\
    &\geq \langle \by_{t+1} - \by_t, A \bx^*(\by_{t}, \bz_t) - \bb \rangle - \frac{1}{2} \|\by_{t+1} - \by_t\|^2 - \frac{1}{2} \|A \bx^*(\by_{t+1}, \bz_t) - A \bx^*(\by_{t}, \bz_t)\|^2 \\
    &\quad + \frac{\mu}{2} \langle \bz_{t+1} - \bz_t, \bz_{t+1} + \bz_t - 2 \bx^*(\by_{t+1}, \bz_{t+1}) \rangle \\
    &\geq \langle \eta(A(\omega_t) \bx_t - \bb(\omega_t)), A \bx^*(\by_{t}, \bz_t) - \bb \rangle - (\frac{1}{2}+\frac{\|A\|^2}{2 \gamma_K^2}) \eta^2 L^2  \\
    &\quad + \frac{p}{2} \langle \bz_{t+1} - \bz_t, \bz_{t+1} + \bz_t - 2 \bx^*(\by_{t+1}, \bz_{t+1}) \rangle,
\end{align*}
where we introduce a term $A\bx^*(\by_t,\bz_t)$ in the first inequality. Then we use Cauchy-Schwarz inequality in the second step, and the last inequality comes from (\ref{x(y',z)-x(y,z)}), $ \|A \bx^*(\by_{t+1}, \bz_t) - A \bx^*(\by_{t}, \bz_t)\|^2\leq \frac{\|A\|^2}{ \gamma_K^2} \|\by_{t+1} - \by_t\|^2 $ and the bound of $\mathbb{E}\|A_{\zeta_t}\bx_t-\bb_{\zeta_t}\|^2$.

After taking expectation and using tower property along with $\by_t, \bz_t$ being deterministic under the conditioning, we have
\begin{equation}
    \begin{aligned} \label{eq: stochasitc d(y,z)}
    \mathbb{E} d(\by_{t+1}, \bz_{t+1}) - \mathbb{E} d(\by_t, \bz_t) 
    &\geq \eta \mathbb{E}\langle (A \bx_t - \bb), A \bx^*(\by_{t}, \bz_t) - \bb \rangle - \left(\frac{1}{2}+\frac{\|A\|^2}{2 \gamma_K^2}\right) \eta^2 L^2 \\
    &\quad + \frac{\mu}{2} \mathbb{E} \langle \bz_{t+1} - \bz_t, \bz_{t+1} + \bz_t - 2 \bx^*(\by_{t+1}, \bz_{t+1}) \rangle 
\end{aligned}
\end{equation}

Then we estimate the first term in the above inequality. We have

\begin{equation}\label{eq: stochasitc d(y,z) first term}
    \begin{aligned}
    &\eta \mathbb{E}\langle A \bx_t - \bb, A \bx^*(\by_{t}, \bz_t)-\bb\rangle\notag \\
    &= \eta \mathbb{E}[\langle (A \bx_t - \bb), A \bx^*(\by_{t+1}, \bz_t) - \bb \rangle + \eta \langle (A \bx_t - \bb), A \bx^*(\by_{t+1}, \bz_t) - A \bx^*(\by_{t}, \bz_t) \rangle]\\
    &\geq \eta \mathbb{E}[\langle (A \bx_t - \bb), A \bx^*(\by_{t+1}, \bz_t) - \bb \rangle  - 8 \| A \bx^*(\by_{t+1}, \bz_t) - A \bx^*(\by_{t}, \bz_t) \|^2]-\frac{\eta^2}{32}\|A\bx_t-\bb\|^2\\
    &\geq \eta \mathbb{E}[\langle (A \bx_t - \bb), A \bx^*(\by_{t+1}, \bz_t) - \bb \rangle-\frac{8\|A\|^2}{\gamma_k^2}\eta^2L^2]-\frac{\eta^2}{32}\|A\bx_t-\bb\|^2,
\end{aligned}
\end{equation}
where we introduce a term $A\bx^*(\by_{t+1},\bz_t)$ to get the first equality. The second inequality comes from Young inequality( \(\langle a,b \rangle \leq \frac{1}{32} \|a\|^2 + 8 \|b\|^2 \forall a,b\)). In last inequality, we use (\ref{x(y',z)-x(y,z)}) and \(\mathbb{E}[\|A_{\zeta_t}\bx_t - \bb_{\zeta_t}\|^2] \leq L\) again.

Finally, plug  \eqref{eq: stochasitc d(y,z) first term} to \eqref{eq: stochasitc d(y,z)}, we obtain the desired result.
    \end{proof}

        \begin{lemma}\label{lem: appendix_descent_stochasitc}
    Let Assumption \ref{asmp: 1} and \ref{assumptions for stochasitc} hold. By using the parameters \eqref{eq: param_defs_appendix} in Algorithm \ref{alg: sgd} with the dual update changed to $\by_{t+1}=\by_t+\eta(A_\zeta \bx_t-\bb_\zeta)$, we obtain
        \begin{equation} \label{desent lemma stochasitc}
            \begin{aligned}
                 \mathbb{E}V_t-\mathbb{E}V_{t+1}&\geq \widetilde{c}_\beta \mathbb{E}\|\bz_{t+1}-\bz_t\|^2+\widetilde{c}_\tau\mathbb{E}\|\bu^*(\bx_t, \by_{t+1}, \bz_{t})-\bx_t\|^2+\widetilde{c}_\eta \mathbb{E}\|A\bx^*(\by_{t+1}, \bz_t)-\bb\|^2\\
                 &-\lambda \tau^2\sigma_2^2-(1+\frac{17\|A\|^2}{\gamma_K^2})\eta^2L^2
            \end{aligned}
        \end{equation}
        where $\widetilde{c}_\beta = \frac{\mu}{50\beta}$, $\widetilde{c}_\tau = \frac{6\tau\lambda^2}{400}$, $\widetilde{c}_\eta = \frac{\eta}{8}$ and $\mathbb{E}\|\hat{G}(\bx_t,\by_t,\bz_t,\xi_t)-\nabla_x K(\bx_t,\by_t,\bz_t)\|^2\leq \sigma_2^2$

        \begin{proof}
            First, we show $\mathbb{E}\|\hat{G}(\bx_t,\by_t,\bz_t,\xi_t)-\nabla_x K(\bx_t,\by_t,\bz_t)\|^2$ is bounded. 
            
            \begin{align*}
                &\mathbb{E}\|\hat{G}(\bx_t,\by_t,\bz_t,\xi_t)-\nabla_x K(\bx_t,\by_t,\bz_t)\|^2\\
                &\leq \mathbb{E}2\|\hat{G}(\bx_t,\by_t,\bz_t,\xi_t)-\hat{G}(\bx_t,0,\bz_t,\xi_t)\|^2+\mathbb{E}2\|\hat{G}(\bx_t,0,\bz_t,\xi_t)-K(\bx_t,\by_t,\bz_t)\|^2\\
                &\leq 2\mathbb{E}L_G\|\by_t\|^2+2\mathbb{E}\|\hat{G}(\bx_t,0,\bz_t,\xi_t)-\nabla_xK(\bx_t,\by_t,\bz_t)\|^2\\
                &\leq 2\mathbb{E}L_G\|\by_t\|^2+4\mathbb{E}\|\hat{G}(\bx_t,0,\bz_t,\xi_t)-\nabla_xK(\bx_t,0,\bz_t)\|^2+4\mathbb{E}\|\nabla_xK(\bx_t,0,\bz_t)-\nabla_xK(\bx_t,\by_t,\bz_t)\|^2\\
                &\leq 2L_G M_y^2+4\|A\|^2\|\by_t\|^2+4\mathbb{E}\|\hat{G}(\bx_t,0,\bz_t,\xi_t)-\nabla_xK(\bx_t,0,\bz_t)\|^2
            \end{align*}

            Because $\bx,\by,\bz$ are all bounded, $\mathbb{E}\|\hat{G}(\bx_t,\by_t,\bz_t,\xi_t)-\nabla_x K(\bx_t,\by_t,\bz_t)\|^2$ is bounded, we denote the upper bound as $\sigma_2^2$.

            Combining with deterministic linear result, we have:
            \begin{align*}
                \mathbb{E}V_t-\mathbb{E}V_{t+1}&\geq c_\beta \mathbb{E}\|\bz_{t+1}-\bz_t\|^2+c_\tau\mathbb{E}\|\bu^*(\bx_t, \by_{t+1}, \bz_{t})-\bx_t\|^2+c_\eta \mathbb{E}\|A\bx^*(\by_{t+1}, \bz_t)-\bb\|^2-\lambda \tau^2\sigma_2^2\\
                &-\frac{\eta^2}{16}\mathbb{E}\|A\bx_t-\bb\|^2-(1+\frac{17\|A\|^2}{\gamma_K^2})\eta^2L^2
            \end{align*}

             where $c_\beta = \frac{\mu}{50\beta}$, $c_\tau = \frac{7\tau\lambda^2}{400}$, $c_\eta = \frac{\eta}{4}$.

        Because
        $$-\frac{\eta^2}{16}\mathbb{E}\|A\bx_t-\bb\|^2 \geq -\frac{\|A\|^2\lambda^2 \eta^2}{8\gamma^2}\|\bx_t-\bu^*(\bx_t,\by_{t+1},\bz_t)\|^2-\frac{\eta^2}{8}\|A\bx(\by_{t+1},\bz_t)-b\|^2 $$

        By the parameter choices, we have $\frac{7\tau \lambda^2}{400}-\frac{\|A\|^2\lambda^2 \eta^2}{8\gamma^2}\geq \frac{6\tau \lambda^2}{400}$ and $\frac{\eta}{4}-\frac{\eta^2}{8}\geq \frac{\eta}{8}$.
        \end{proof}
    \end{lemma}

        \begin{proposition}\label{prop: y_bdd}
Under \Cref{assumptions for stochasitc},            $\|\by_t\|\leq\frac{\Psi(\bz_t)-d(\by_t,\bz_t)+2M}{r}$, where $M=\max_{\bx,\bz\in X}\{|f(\bx)|+\frac{\mu}{2}\|\bx-\bz\|^2+\frac{\rho}{2}\|A\bx-\bb\|^2\}$ and $r>0$ is defined as $\|A\hat\bx-\bb\|=r$ where $\hat\bx$ is in the relative interior of the constraints. The existence of this is guaranteed by our assumption.
            \begin{proof}
Given $\widetilde{\bx}\in X$, we have
            \begin{align*}
                \Psi(\bz_t)-d(\by_t,\bz_t) &\geq f(\bar{\bx}^*(\bz_t))+\frac{\mu}{2}\|\bar{\bx}^*(\bz_t)-\bz_t\|^2-K(\widetilde{\bx},\by_t,\bz_t)\\
                &\geq f(\bar{\bx}^*(\bz_t))+\frac{\mu}{2}\|\bar{\bx}^*(\bz_t)-\bz_t\|^2-[f(\widetilde{\bx})+\langle \by_t,A\widetilde{\bx} \rangle+\frac{\rho}{2}\|A\widetilde{\bx}-\bb\|^2+\frac{\mu}{2}\|\widetilde{\bx}-\bz_t\|^2]\\
                &= [ f(\bar{\bx}^*(\bz_t))+\frac{\mu}{2}\|\bar{\bx}^*(\bz_t)-\bz_t\|^2-f(\widetilde{\bx})-\frac{\mu}{2}\|\widetilde{\bx}-\bz_t\|^2]-\langle \by_t,A\widetilde{\bx}-\bb\rangle-\frac{\rho}{2}\|A\widetilde{\bx}-\bb\|^2\\
                &= [ f(\bar{\bx}^*(\bz_t))+\frac{\mu}{2}\|\bar{\bx}^*(\bz_t)-\bz_t\|^2-f(\widetilde{\bx})-\frac{\mu}{2}\|\widetilde{\bx}-\bz_t\|^2-\frac{\rho}{2}\|A\widetilde{\bx}-\bb\|^2]-\langle \by_t,A\widetilde{\bx}-\bb\rangle\\
                &\geq -2M-\langle \by_t,A\widetilde{\bx}-\bb\rangle.
            \end{align*}
        
            Where the first inequality comes from the definition of $\Psi(\bz_t)$ and 
            $$d(\by_t,\bz_t)=\min_{\bx\in X}K(\bx,\by,\bz)$$
        
            And in the last inequality, we let
            $$M=\max_{\bx,\bz\in X}\{|f(\bx)|+\frac{\mu}{2}\|\bx-\bz\|^2+\frac{\rho}{2}\|A\bx-\bb\|^2\} $$
        
            So we have the last inequality.
        
            According to Assumption \ref{assumptions for stochasitc}(2), there exists a positive $r>0$ such that for any direction $\mathbf{d}\in \text{Range}(A)$, we can find a $\bx\in X$ satisfying $\|A\bx-\bb\|=r$ and $A\bx-\bb$ has the same direction as $\mathbf{d}$. Because $\by_t\in \text{Range}(A)$ (by assumption \ref{assumptions for stochasitc}(3), $\text{Range}(A)=\mathbb{R}^m$) we can choose $\widetilde{\bx}$ such that $A\widetilde{\bx}-\bb$ is of the same direction as $-\by_t$ and $\|A\widetilde{\bx}-\bb\|=r$. Then we obtain
        $$\Psi(\bz_t)-d(\by_t,\bz_t) \geq -2M+r\|\by_t\| \Longrightarrow \|\by_t\|\leq \frac{\Psi(\bz_t)-d(\by_t,\bz_t)+2M}{r},\forall t\in \{0,1,...,T\}.$$
This concludes the proof.
            \end{proof}
        
        \end{proposition}

       \textbf{Then we start the proof for Theorem \ref{thm:stochasitc linear constraints}}

            \begin{proof}
                First, let $M_V=\max_{\bx,\bz\in X}\{K(\bx,0,\bz)-2d(0,\bz)+2\Psi(\bz)\}$ and $M_y>\frac{M_V-M_\Psi+2M}{r}$ where $M_\Psi$ is a uniform lower bound of $\Psi(\bz_t)$, for example, $\underline f$.
            
                If $\|\by_{t+1}\|\leq M_y$, then 
                \begin{equation} \label{y<My}
                    \begin{aligned}
                         \mathbb{E}V_t-\mathbb{E}V_{t+1}&\geq \widetilde{c}_\beta \mathbb{E}\|\bz_{t+1}-\bz_t\|^2+\widetilde{c}_\tau\mathbb{E}\|\bu^*(\bx_t, \by_{t+1}, \bz_{t})-\bx_t\|^2+\widetilde{c}_\eta \mathbb{E}\|A\bx^*(\by_{t+1}, \bz_t)-\bb\|^2\\
                         &\quad -\lambda \tau^2\sigma_2^2-(1+\frac{17\|A\|^2}{\gamma_K^2})\eta^2L^2
                    \end{aligned}
                \end{equation}
                according to the analysis of \Cref{lem: appendix_descent_stochasitc}.
                
                If $\|\by_{t+1}\|> M_y$. For distinction, if we perform the procedure $\by_{t+1}=0$, let us denote the update as $\by_{t+1}, \bx_{t+1}, \bz_{t+1}$ as $\hat{\by}_{t+1}, \hat{\bx}_{t+1}, \hat{\bz}_{t+1}$ and $\by_{t+1}, \bx_{t+1}, \bz_{t+1}$ denote the iteration generated without taking $\by_{t+1}=0$. Then
                \begin{align*}
                   K(\bx_{t+1},\by_{t+1},\bz_{t+1})-2d(\by_{t+1},\bz_{t+1})+2\Psi(\bz_{t+1})&\geq \Psi(\bz_{t+1}) -d(\by_{t+1}, \bz_{t+1}) + \Psi(\bz_{t+1})\\
                   &\geq r\|\by_{t+1}\|-2M+M_\Psi\\
                    &\geq rM_y-2M+M_\Psi\\
                    &\geq M_V\\
                    &=\max_{\bx,\bz\in X}\{K(\bx,0,\bz)-2d(0,\bz)+2\Psi(\bz)\}\\
                    &\geq K(\hat{\bx}_{t+1},0,\hat{\bz}_{t+1})-2d(0,\hat{\bz}_{t+1})+2\Psi(\hat{\bz}_{t+1}),
                \end{align*}
where the first step used $d(\by_{t+1}, \bz_{t+1}) \leq K(\bx_{t+1}, \by_{t+1}, \bz_{t+1})$ and the second line used $\Psi(\bz_{t+1}) \geq M_\Psi$.    
        
                Hence, $\mathbb{E}V_t-\mathbb{E}V_{t+1}$ becomes larger if we run $\by_{t+1}=0$. So \eqref{y<My} still holds, then the convergence result follows. 
                
                The rest of the proof for the complexity result proceeds the same as \Cref{appendix: complexity_moreau} up to simple changes in the constants, and hence is omitted.                
                \end{proof}
                
\section{Proof for Section \ref{sec: var_red}}\label{appendix: storm_proof}
\begin{lemma}\cite[Lemma 3.10]{zhang2020proximal}\label{Lem of K}
    Under \Cref{asmp: 1}, we have
    \begin{align*}
        \left\| \bx - \proj_X(\bx - \tau \nabla K(\bx,\by,\bz)) \right\| \geq \tau(\mu-L_f)\left\| \bx - \bx^*(\by, \bz) \right\|,
    \end{align*}
where $ K(\bx, \by, \bz) = L_\rho(\bx, \by) + \frac{\mu}{2}\|\bx-\bz\|^2$, and $\bx^*(\by,\bz)=\min_{\bx \in X}K(\bx,\by,\bz)$. 
\end{lemma}

\begin{lemma} \label{lem: x-x(y,z)}
    Under \Cref{asmp: 1}, for the iterates generated by \Cref{alg: alm-storm}  we have 
$$ \|\bx_t-\bx^*(\by_{t+1},\bz_t)\|\leq  \frac{1}{\tau(\mu-L_f)}\|\bx_t-\bx_{t+1}\|+\frac{1}{(\mu-L_f)}\|\widehat{\nabla}f_t-\nabla f(\bx_t)\|$$
\end{lemma}
\begin{proof}
    Taking $\bx, \by, \bz$ as $\bx_t, \by_{t+1},\bz_t$ in Lemma \ref{Lem of K}, we have
    \begin{align*}
        \|\bx_t-\bx^*(\by_{t+1},\bz_t)\|&\leq \frac{1}{\tau(\mu-L_f)}\|\bx_t-\proj_X(\bx_t - \tau \nabla K(\bx, \by_{t+1},\bz_t) )\|\\
        & \leq \frac{1}{\tau(\mu-L_f)}\|\bx_t-\proj_X(\bx -\tau G(\bx_t,\by_{t+1},\bz_t) )\|\\
        &\quad +\frac{1}{\tau(\mu-L_f)}\|\proj_X(\bx_t - \tau \nabla K(\bx_t, \by_{t+1},\bz_t) )-\proj_X(\bx_t-\tau G(\bx_t,\by_{t+1},\bz_t))\|\\
        &\leq  \frac{1}{\tau(\mu-L_f)}\|\bx_t-\bx_{t+1}\|+\frac{1}{(\mu-L_f)}\|\widehat{\nabla}f_t-\nabla f(\bx_t)\|,
    \end{align*}
    where the second inequality comes form triangle inequality and the last inequality comes from the fact that $\proj_X$ is nonexpansive and $\nabla K(\bx_t,\by_{t+1},\bz_t)-G(\bx_t,\by_{t+1},\bz_t)=\widehat{\nabla}f_t-\nabla f(\bx_t)$
\end{proof}

\begin{proof}[Proof of \Cref{lemma: variance_reduction}]
    By the definition of $\widehat{\nabla} f_{t+1}$, we have
    \begin{align}
        &\widehat{\nabla}f_{t+1}-\nabla f(\bx_{t+1}) \notag\\
        &=\nabla f(\bx_{t+1},\xi_{t+1})+(1-\alpha)( \widehat{\nabla}f_{t}-\nabla f(\bx_t,\xi_{t+1}))-\nabla f(\bx_{t+1})\notag \\
                                        &=\nabla f(\bx_{t+1},\xi_{t+1})+(1-\alpha)(\widehat{\nabla}f_{t}-\nabla f(\bx_{t}))+(1-\alpha)(\nabla f(\bx_t)-\nabla f(\bx_t,\xi_{t+1}))-\nabla f(\bx_{t+1})\notag \\
                                        &=(1-\alpha)(\widehat{\nabla}f_{t}-\nabla f(\bx_t))+(1-\alpha)(\nabla f(\bx_t)-f(\bx_t,\xi_{t+1}))+f(\bx_{t+1},\xi_{t+1})-\nabla f(\bx_{t+1}),\label{eq: sax4}
    \end{align}
where in the second equality, we added and subtracted $(1-\alpha)\nabla f(\bx_t)$.

Then, we compute the squared norm of \eqref{eq: sax4} and expand to get
\begin{align*}
    &\|\widehat{\nabla}f_{t+1}-\nabla f(\bx_{t+1})\|^2\\
    &=(1-\alpha)^2\|\widehat{\nabla}f_{t}-\nabla f(\bx_t)\|^2+\|(1-\alpha)(\nabla f(\bx_t)-f(\bx_t,\xi_{t+1}))+f(\bx_{t+1},\xi_{t+1})-\nabla f(\bx_{t+1})\|^2\\
    &\quad+2(1-\alpha)\langle \widehat{\nabla}f_{t}-\nabla f(\bx_t),(1-\alpha)(\nabla f(\bx_t)-f(\bx_t,\xi_{t+1}))+f(\bx_{t+1},\xi_{t+1})-\nabla f(\bx_{t+1})\rangle.
\end{align*}
Next, we take expectation with respect to $\xi_{t+1}$ to obtain
\begin{align}
    &\mathbb{E}_{\xi_{t+1}}\|\widehat{\nabla}f_{t+1}-\nabla f(\bx_{t+1})\|^2\notag \\
    &=(1-\alpha)^2\mathbb{E}_{\xi_{t+1}}\|\widehat{\nabla}f_{t}-\nabla f(\bx_t)\|^2+\mathbb{E}_{\xi_{t+1}}\|(1-\alpha)(\nabla f(\bx_t)-f(\bx_t,\xi_{t+1}))+f(\bx_{t+1},\xi_{t+1})-\nabla f(\bx_{t+1})\|^2,\label{eq: sxz4}
\end{align}
which is due to $\widehat{\nabla}f_{t}-\nabla f(\bx_t)$ being independent of $\xi_{t+1}$, and 
$$\mathbb{E}_{\xi_{t+1}}[\nabla f(\bx_t)-f(\bx_t,\xi_{t+1})]=0, \mathbb{E}_{\xi_{t+1}}[f(\bx_{t+1},\xi_{t+1})-\nabla f(\bx_{t+1})]=0.$$
Finally, we estimate the last term in the right-hand side of \eqref{eq: sxz4}:
\begin{align*}
    &\mathbb{E}_{\xi_{t+1}}\|(1-\alpha)(\nabla f(\bx_t)-f(\bx_t,\xi_{t+1}))+f(\bx_{t+1},\xi_{t+1})-\nabla f(\bx_{t+1})\|^2\\
    &=\mathbb{E}_{\xi_{t+1}}\|f(\bx_{t+1},\xi_{t+1})-f(\bx_{t},\xi_{t+1})+\nabla f(\bx_t)-\nabla f(\bx_{t+1})+\alpha(f(\bx_t,\xi_{t+1})-\nabla f(\bx_t))\|^2\\
    &\leq 3\mathbb{E}_{\xi_{t+1}}\|f(\bx_{t+1},\xi_{t+1})-f(\bx_{t},\xi_{t+1})\|^2+3\mathbb{E}_{\xi_{t+1}}\|\nabla f(\bx_t)-\nabla f(\bx_{t+1})\|^2+3\mathbb{E}_{\xi_{t+1}}\|\alpha(f(\bx_t,\xi_{t+1})-\nabla f(\bx_t))\|^2\\
    &\leq 3L_0^2\|\bx_{t+1}-\bx_t\|^2+3L_f^2\|\bx_t-\bx_{t+1}\|^2+3\alpha^2\sigma^2,
\end{align*}
where in the first equality, we rearrange the terms, and in the first inequality, we use Young's inequality. Then in the second inequality, we use the Assumption \ref{assumption: variance_reduction}, $L_f$-smoothness of $f(\bx)$ and $\mathbb{E}_{\xi}\|\nabla f(\bx,\xi)-\nabla f(\bx)\|^2 \leq \sigma^2$. We use this estimation in \eqref{eq: sxz4} and take total expectation to get the result.
\end{proof}

\begin{proof}[Proof of \Cref{lemma: descent_alm_storm}] \label{proof of lemma: descent_alm_storm}
    We have, by smoothness of \( K \):
\[
K(\bx_{t+1}, \by_{t+1}, \bz_t) \leq K(\bx_t, \by_{t+1}, \bz_t) + \langle \nabla_{\bx} K(\bx_t, \by_{t+1}, \bz_t), \bx_{t+1} - \bx_t \rangle + \frac{L_K}{2} \| \bx_{t+1} - \bx_t \|^2.
\]
We estimate the inner product here as
\[
\langle \nabla_{\bx} K(\bx_t, \by_{t+1}, \bz_t), \bx_{t+1} - \bx_t \rangle = \langle G(\bx_t, \by_{t+1}, \bz_t), \bx_{t+1} - \bx_k \rangle + \langle \nabla_{\bx} K(\bx_t, \by_{t+1}, \bz_t) - G(\bx_t, \by_{t+1}, \bz_t), \bx_{t+1} - \bx_t \rangle.
\]
We first have
\[
\nabla_{\bx} K(\bx_t, \by_{t+1}, \bz_t) - G(\bx_t, \by_{t+1}, \bz_t) = \nabla f(\bx_t) - \widehat{\nabla}f_t ,
\]
The definition of \( \bx_{t+1} \) gives
\[
\langle \bx_{t+1} - \bx_t + \tau G(\bx_t, \by_{t+1}, \bz_t), \bx_t - \bx_{t+1} \rangle \geq 0 \quad \Longleftrightarrow \quad \langle G(\bx_t, \by_{t+1}, \bz_t), \bx_t - \bx_{t+1} \rangle \geq \frac{1}{\tau} \| \bx_{t+1} - \bx_t \|^2,
\]
Using $ \langle \nabla_{\bx} K(\bx_t, \by_{t+1}, \bz_t) - G(\bx_t, \by_{t+1}, \bz_t), \bx_{t+1} - \bx_t \rangle \leq \frac{\tau}{2}\| \nabla f(\bx_t) - \widehat{\nabla}f_t \|+\frac{1}{2\tau}\|\bx_{t+1} - \bx_t\|$, we have
\[
\langle \nabla_{\bx} K(\bx_t, \by_{t+1}, \bz_t), \bx_{t+1} - \bx_t \rangle \leq \frac{\tau}{2} \| \nabla f(\bx_t) - \widehat{\nabla}f_t  \|^2 - \frac{1}{2\tau} \| \bx_{t+1} - \bx_t \|^2.
\]
Then the result follows.
\end{proof}

\begin{proof}[Proof of \Cref{lemma: descent_alm_storm2}]
    First, from the definition we have
    $$K(\bx_t,\by_t,\bz_t)-K(\bx_t,\by_{t+1},\bz_t)=-\eta \|A\bx_t-\bb\|^2$$
    and also
    \begin{align*}
        &K(\bx_{t+1},\by_{t+1},\bz_t)-K(\bx_{t+1},\by_{t+1},\bz_{t+1})\notag \\
        &=\frac{\mu}{2}(\|\bx_{t+1}-\bz_t\|^2-\|\bx_{t+1}-\bz_{t+1}\|^2)\\
        &=\frac{\mu}{2}\langle \bz_{t+1}-\bz_t,2\bx_{t+1}-\bz_t-\bz_{t+1}\rangle\\
        &=\frac{\mu}{2}\langle \bz_{t+1}-\bz_t,2\bx_{t+1}-2\bx_t+2\bx_t-2\bz_t+\bz_t-\bz_{t+1}\rangle\\
        &=\frac{\mu}{2}\langle \bz_{t+1}-\bz_t,2\bx_{t+1}-2\bx_t\rangle+\frac{\mu}{2}\langle \bz_{t+1}-\bz_t,2\bx_{t}-2\bz_t\rangle-\frac{\mu}{2}\|\bz_{t+1}-\bz_t\|^2\\
        &\geq -\frac{\mu}{4}\|\bz_{t+1}-\bz_t\|^2-\mu\|\bx_{t+1}-\bx_t\|^2+\frac{\mu}{\beta}\|\bz_t-\bz_{t+1}\|^2-\frac{\mu}{2}\|\bz_{t+1}-\bz_t\|^2,
    \end{align*}
where the first equality comes from the definition of $K$. In the last inequality, we use $\langle a,b\rangle\geq -\frac{1}{4}\|a\|^2-\|b\|^2$ and $\bx_t-\bz_t=\frac{\bz_{t+1}-\bz_t}{\beta}$ by the definition of $\bz_{t+1}$ in Algorithm \ref{alg: alm-storm}.

Then combining the above two results with Lemma \ref{lemma: descent_alm_storm} yields the claim.
\end{proof}

\begin{proof}[Proof of \Cref{th: alm_storm_descent}]
    We denote that
    \begin{equation}
        V_t = K(\bx_t,\by_t,\bz_t) - 2d(\by_t, \bz_t) + 2\Psi(\bz_t).
    \end{equation}
    Joining \eqref{eq: descent_K} with Lemma \ref{lem: swo4}, we have
    \begin{align}
        \mathbb{E}V_t-\mathbb{E}V_{t+1}&\geq -\eta \mathbb{E}\|A\bx_t-\bb\|^2+(\frac{\mu}{\beta}-\frac{3\mu}{4})\mathbb{E}\|\bz_{t+1}-\bz_t\|^2 \notag \\
        &\quad-\frac{\tau}{2} \mathbb{E}\| \nabla f(\bx_t) - \widehat{\nabla}f_t  \|^2 +(\frac{1}{2\tau}-\frac{L_K}{2}-\mu) \mathbb{E}\| \bx_{t+1} - \bx_t \|^2 \notag \\
        &\quad +2\eta \mathbb{E}\langle A\bx_t-\bb, A\bx^*(\by_{t+1}, \bz_t)-\bb \rangle+\mu\mathbb{E}\langle \bz_{t+1}-\bz_{t}, \bz_{t+1}+\bz_{t}-2\bx^*(\by_{t+1},\bz_{t+1})\rangle \notag \\
        &\quad -2\mu\mathbb{E}\langle \bz_{t+1}-\bz_t, \bz_t-\bar{\bx}^*(\bz_t)\rangle-\frac{\mu}{\sigma_4}\mathbb{E}\|\bz_t-\bz_{t+1}\|^2.\label{eq: sow44}
    \end{align}
    First, let us combine the first and fifth terms on the right-hand side to obtain
    \begin{align}
        -\eta \mathbb{E}\|A\bx_t-\bb\|^2+2\eta\langle A\bx_t-\bb, A\bx^*(\by_{t+1}, \bz_t)-\bb \rangle&=-\eta\|A\bx_t-A\bx^*(\by_{t+1}, \bz_t)\|^2+\eta\|A\bx^*(\by_{t+1}, \bz_t)-\bb\|^2. \label{eq: axo2}
    \end{align}
    Next, we combine the sixth and seventh terms on the right-hand side of \eqref{eq: sow44}
    \begin{align}
        &\mu\langle \bz_{t+1}-\bz_{t}, \bz_{t+1}+\bz_{t}-2\bx^*(\by_{t+1},\bz_{t+1})\rangle-2\mu\langle \bz_{t+1}-\bz_t, \bz_t-\bar{\bx}^*(\bz_t)\rangle\notag \\
        &=\mu\langle \bz_{t+1}-\bz_{t},\bz_{t+1}-\bz_{t}-2\bx^*(\by_{t+1},\bz_{t+1})+2\bar{\bx}^*(\bz_t)\rangle \notag \\
        &=\mu \| \bz_{t+1} - \bz_t\|^2+ 2\langle \bz_{t+1}-\bz_{t},-\bx^*(\by_{t+1},\bz_{t+1})+\bar{\bx}^*(\bz_t)\rangle\label{eq: axq4}
    \end{align}
    We now single out the inner product in the last equality and estimate it by adding and subtracting $\bx^\ast(\by_{t+1}, \bz_t)$ in the second argument of the inner product:
    \begin{align}
        &2\langle \bz_{t+1}-\bz_{t},-\bx^*(\by_{t+1},\bz_{t+1})+\bar{\bx}^*(\bz_t)\rangle \notag \\
         &= 2\mu\langle \bz_{t+1}-\bz_{t},-\bx^*(\by_{t+1},\bz_{t+1})+\bx^*(\by_{t+1},\bz_{t})\rangle+2\mu\langle \bz_{t+1}-\bz_{t},-\bx^*(\by_{t+1},\bz_{t})+\bar{\bx}^*(\bz_t)\rangle\notag \\
        &\geq -\mu \|\bz_{t+1}-\bz_t\|^2-\mu\|\bx^*(\by_{t+1},\bz_{t})-\bx^*(\by_{t+1},\bz_{t+1})\|^2-\frac{\mu}{\zeta}\|\bz_{t+1}-\bz_t\|^2-\mu \zeta\|\bar{\bx}^*(\bz_t)-\bx^*(\by_{t+1},\bz_{t})\|, \label{eq: aft4}
    \end{align}
    for any $\zeta$, where we used Young's inequality twice. 
    Then, we plug this into \eqref{eq: axq4} to obtain
    \begin{align}
    &\mu\langle \bz_{t+1}-\bz_{t}, \bz_{t+1}+\bz_{t}-2\bx^*(\by_{t+1},\bz_{t+1})\rangle-2\mu\langle \bz_{t+1}-\bz_t, \bz_t-\bar{\bx}^*(\bz_t)\rangle\notag \\
        &\geq -\frac{\mu}{\sigma_4^2}\|\bz_{t+1}-\bz_t\|^2-\frac{\mu}{\zeta}\|\bz_{t+1}-\bz_t\|^2-\mu \zeta\|\bar{\bx}^*(\bz_t)-\bx^*(\by_{t+1},\bz_{t})\|^2,\label{eq: dty3}
    \end{align}
    where
    we use \eqref{x(y,z')-x(y,z)} to bound the second term on the right-hand side of \eqref{eq: aft4}, where $\sigma_4$ is as defined in \eqref{eq: param_defs_appendix}.

    Then we use \eqref{eq: axo2} and \eqref{eq: dty3} in \eqref{eq: sow44} to obtain
    \begin{align*}
        \mathbb{E}V_t-\mathbb{E}V_{t+1}&\geq (\frac{\mu}{\beta}-\frac{3\mu}{4})\mathbb{E}\|\bz_{t+1}-\bz_t\|^2-\frac{\tau}{2} \mathbb{E}\| \nabla f(\bx_t) - \widehat{\nabla}f_t  \|^2 +(\frac{1}{2\tau}-\frac{L_K}{2}-\mu) \mathbb{E}\| \bx_{t+1} - \bx_t \|^2\\
        &\quad-\eta\mathbb{E}\|A\bx_t-A\bx^*(\by_{t+1}, \bz_t)\|^2+\eta\mathbb{E}\|A\bx^*(\by_{t+1}, \bz_t)-\bb\|^2\\
        &\quad-\frac{\mu}{\sigma_4^2}\mathbb{E}\|\bz_{t+1}-\bz_t\|^2-\frac{\mu}{\zeta}\mathbb{E}\|\bz_{t+1}-\bz_t\|^2-\mu \zeta\|\bar{\bx}^*(\bz_t)-\bx^*(\by_{t+1},\bz_{t})\|^2-\frac{\mu}{\sigma_4}\mathbb{E}\|\bz_{t+1}-\bz_t\|^2\\
        &\geq (\frac{\mu}{\beta}-\frac{3\mu}{4}-\frac{\mu}{\sigma_4^2}-\frac{\mu}{\zeta}-\frac{\mu}{\sigma_4})\mathbb{E}\|\bz_{t+1}-\bz_t\|^2\\
        &\quad-\frac{\tau}{2} \mathbb{E}\| \nabla f(\bx_t) - \widehat{\nabla}f_t  \|^2-\frac{2\eta \|A\|^2}{(\mu-L_f)^2}\mathbb{E}\| \nabla f(\bx_t) - \widehat{\nabla}f_t  \|^2 \\
         &\quad+(\frac{1}{2\tau}-\frac{L_K}{2}-\mu-\eta\|A\|^2\frac{2}{\tau^2(\mu-L_f)^2}) \mathbb{E}\| \bx_{t+1} - \bx_t \|^2\\
        &\quad+\eta\mathbb{E}\|A\bx^*(\by_{t+1},\bz_t)-\bb\|^2-\mu \zeta \bar{\sigma}^2\mathbb{E}\|A\bx^*(\by_{t+1},\bz_t)-\bb\|
    \end{align*}
    
    Where in the last inequality, we use the \ref{lem: x-x(y,z)} and \ref{global error}, then combine the same terms together.
    
    Then we need to estimate the coefficients of each terms in the above inequality.
    From the choosen parameters, we easily have $\sigma_4=\frac{\mu-L_f}{\mu}>\frac{1}{2}$ and let $\zeta=6\beta$

We now estimate the coefficient of $\mathbb{E}\|\bz_t-\bz_{t+1}\|^2$ in the last inequality. First, by $\sigma_4>\frac{1}{2}$, we have $\frac{\mu}{\sigma_4^2}\leq 4\mu$ and $\frac{\mu}{\sigma_4}\leq 2\mu$. By also using $\zeta=6\beta$, we have
$$\text{The coefficient of } \mathbb{E}\|\bz_t-\bz_{t+1}\|^2\geq \frac{\mu}{\beta}-\frac{3\mu}{4}-4\mu-\frac{\mu}{6\beta}-2\mu$$

Second, using $\beta \leq 1/50$, we obtain $(\frac{3}{4}+4+2)\mu\leq \frac{\mu}{5\beta}$, then
$$\text{The coefficient of } \mathbb{E}\|\bz_t-\bz_{t+1}\|^2\geq \frac{\mu}{\beta}-\frac{\mu}{5\beta}-\frac{\mu}{6\beta}\geq \frac{\mu}{\beta}.$$
We move on to estimating the coefficient of $\mathbb{E}\|\bx_t-\bx_{t+1}\|^2$.
With $\eta \leq \frac{(\mu-L_f)^2\tau}{8\|A\|^2}$, we have $2\eta\|A\|^2\frac{1}{\tau^2(\mu-L_f)^2}\leq \frac{1}{4\tau}$, we have
$$\text{The coefficient of } \mathbb{E}\|\bx_t-\bx_{t+1}\|^2\geq \frac{1}{4\tau}-\frac{L_K}{2}-\mu$$
Last, we work on the coefficient of $\mathbb{E}\|A\bx^*(\by_{t+1},\bz_t)-\bb\|^2$. Because $\zeta= 6\beta$, it follows that $\eta-\mu \zeta \bar{\sigma}^2=\eta -6\mu\beta\bar{\sigma}^2$.

With $\beta\leq \frac{\eta}{36\mu\bar{\sigma}^2}$, we have $6\mu\beta\bar{\sigma}^2\leq \frac{\eta}{6}$, then
$$\text{The coefficient of } \mathbb{E}\|A\bx^*(\by_{t+1},\bz_t)-\bb\|^2\geq \eta-\frac{\eta}{6}\geq \frac{\eta}{2}.$$
Next, we estimate the coefficient of $\mathbb{E}\| \nabla f(\bx_t) - \widehat{\nabla}f_t  \|^2$. With $\eta \leq \frac{(\mu-L_f)^2\tau}{8\|A\|^2}$, we have $-\frac{\tau}{2}-\frac{2\eta \|A\|^2}{(\mu-L_f)^2}\geq -\frac{3}{4}\tau$. Finally, we have
\begin{equation}\label{eq: alm_storm descent primal}
    \begin{aligned}
        &\mathbb{E}V_t-\mathbb{E}V_{t+1}\notag \\
        &\geq \frac{\mu}{2\beta}\mathbb{E}\|\bz_t-\bz_{t+1}\|^2+(\frac{1}{4\tau}-\frac{L_K}{2}-\mu)\mathbb{E}\|\bx_t-\bx_{t+1}\|^2+\frac{\eta}{2}\mathbb{E}\|A\bx^*(\by_{t+1},\bz_t)-\bb\|^2-\frac{3\tau}{4}\mathbb{E}\|\nabla f(\bx_t)-\widehat{\nabla}f_t \|^2\\
            &=\frac{\mu}{2\beta}\mathbb{E}\|\bz_t-\bz_{t+1}\|^2+(\frac{1}{4\tau}-\frac{L_K}{2}-\mu)\mathbb{E}\|\bx_t-\bx_{t+1}\|^2+\frac{\eta}{2}\mathbb{E}\|A\bx^*(\by_{t+1},\bz_t)-\bb\|^2\\
            &\quad +\frac{\tau}{4}\mathbb{E}\|\nabla f(\bx_t)-\widehat{\nabla}f_t \|^2-\tau\mathbb{E}\|\nabla f(\bx_t)-\widehat{\nabla}f_t \|^2
    \end{aligned}
\end{equation}

Then recalling Lemma \ref{lemma: variance_reduction} and assuming $0<\alpha\leq1$, we have
\begin{equation}\label{eq: variance_reduction2}
    \mathbb{E}\|\widehat{\nabla}f_{t+1} -\nabla f(\bx_{t+1})\|^2\leq (1-\alpha)\mathbb{E}\|\widehat{\nabla}f_t -\nabla f(\bx_t)\|^2+3(L_0^2+L_f^2)\mathbb{E}\|\bx_{t+1}-\bx_t\|^2+3\alpha^2\sigma^2
\end{equation}

We multiply \eqref{eq: variance_reduction2} by $\frac{\tau}{\alpha}$ and plug into \eqref{eq: alm_storm descent primal}, to get
\begin{align}
    \mathbb{E}V_t-\mathbb{E}V_{t+1}&\geq \frac{\mu}{2\beta}\mathbb{E}\|\bz_t-\bz_{t+1}\|^2+(\frac{1}{4\tau}-\frac{L_K}{2}-\mu)\mathbb{E}\|\bx_t-\bx_{t+1}\|^2\notag\\
    &\quad+\frac{\eta}{2}\mathbb{E}\|A\bx^*(\by_{t+1},\bz_t)-\bb\|^2 +\frac{\tau}{4}\mathbb{E}\|\nabla f(\bx_t)-\widehat{\nabla}f_t \|^2\notag \\
    &\quad+\frac{\tau}{\alpha}\mathbb{E}\|\widehat{\nabla}f_{t+1}-\nabla f(\bx_{t+1})\|^2-\frac{\tau}{\alpha}\mathbb{E}\|\widehat{\nabla}f_t -\nabla f(\bx_{t})\|^2\notag \\
    &\quad-\frac{3(L_0^2+L_f^2)\tau}{\alpha}\mathbb{E}\|\bx_t-\bx_{t+1}\|^2-3\alpha \sigma^2\tau.\label{eq: snv4}
\end{align}
Because $\alpha=48(L_0^2+L_f^2)\tau^2$ and $\tau \leq \min\{\frac{1}{4L_K+8\mu},\frac{1}{\sqrt{48(L_0^2+L_f^2)}}\}$, we  obtain
\begin{align*}
    \mathbb{E}V_t-\mathbb{E}V_{t+1}&\geq \frac{\mu}{2\beta}\mathbb{E}\|\bz_t-\bz_{t+1}\|^2+\frac{1}{8\tau}\mathbb{E}\|\bx_t-\bx_{t+1}\|^2+\frac{\eta}{2}\mathbb{E}\|A\bx^*(\by_{t+1},\bz_t)-\bb\|^2+\frac{\tau}{4}\mathbb{E}\|\nabla f(\bx_t)-\widehat{\nabla}f_t \|^2\\
    &+\frac{1}{48(L_0^2+L_f^2)\tau}\mathbb{E}\|\widehat{\nabla}f_{t+1}-\nabla f(\bx_{t+1})\|^2-\frac{1}{48(L_0^2+L_f^2)\tau}\mathbb{E}\|\widehat{\nabla}f_t -\nabla f(\bx_{t})\|^2-144(L_0^2+L_f^2)\sigma^2\tau^3.
\end{align*}
Finally, we move $\frac{1}{48(L_0^2+L_f^2)\tau}\mathbb{E}\|\widehat{\nabla}f_{t+1} -\nabla f(\bx_{t+1})\|^2-\frac{1}{48(L_0^2+L_f^2)\tau}\mathbb{E}\|\widehat{\nabla}f_t -\nabla f(\bx_{t})\|^2$ to the left-hand side of the above inequality and use the definition of $\bar V_t$ in \eqref{eq: barv_def} to get the desired result.
\end{proof}

\begin{proof}[Proof of \Cref{th: storm_complexity}]
    Because $\bz_{t+1}-\bz_t=\beta(\bx_t-\bz_t)$, $\frac{\mu \beta}{2}=\Theta(\tau)$ and $\frac{\eta}{2}=\Theta(\tau)$, hence there exists a constant $C$ such that
    \begin{align}
    \mathbb{E}\bar{V}_t-\mathbb{E}\bar{V}_{t+1}&\geq C\tau\{\mathbb{E}\|\bx_t-\bz_{t}\|^2+\mathbb{E}\|\tau^{-1}(\bx_t-\bx_{t+1})\|^2+\mathbb{E}\|A\bx^*(\by_{t+1},\bz_t)-\bb\|^2+\mathbb{E}\|\nabla f(\bx_t)-\widehat{\nabla}f_t \|^2\}\notag\\
    &\quad -144(L_0^2+L_f^2)\sigma^2\tau^3.\label{eq: alm_storm_descent}
\end{align}

Then, summing up \eqref{eq: alm_storm_descent} over $t=0,1,\ldots,T-1$, we have
\begin{align}
    \mathbb{E}\bar{V}_0-\mathbb{E}\bar{V}_{T}&\geq \sum_{t=0}^{T-1}C\tau\{\mathbb{E}\|\bx_t-\bz_{t}\|^2+\mathbb{E}\|\tau^{-1}(\bx_t-\bx_{t+1})\|^2+\mathbb{E}\|A\bx^*(\by_{t+1},\bz_t)-\bb\|^2+\mathbb{E}\|\nabla f(\bx_t)-\widehat{\nabla}f_t \|^2\}\notag\\
    &\quad -144(L_0^2+L_f^2)\sigma^2\tau^3T.\label{V_0-V_T}
\end{align}
Form the definition, we have $K(\bx,\by,\bz)\geq d(\by,\bz)$ and $\Psi(\bz)\geq d(\by,\bz)$ (see also Lemma \ref{lem: lb}), then
\begin{equation*}V_t=K(\bx_t,\by_t,\bz_t)-2d(\by_t, \bz_t) + 2\Psi(\bz_t)\geq \Psi(\bz_t)\geq \underline{f}.\end{equation*}
Then, we have 
\begin{equation}\label{V_T}
    \bar{V}_t=K(\bx_t,\by_t,\bz_t)-2d(\by_t, \bz_t) + 2\Psi(\bz_t)+\frac{1}{48(L_0^2+L_f^2)\tau}\mathbb{E}\|\widehat{\nabla}f_t -\nabla f(\bx_{t})\|^2\geq \underline{f}.
\end{equation}
Let $\tau=T^{-1/3}$ and use mini-batch in the initial step where we will have $\mathbb{E}\|\widehat{\nabla}f_0 -\nabla f(\bx_{0})\|^2\leq T^{-1/3}\sigma^2$ , then 
\begin{equation}\label{V_0}
\begin{aligned}
    \bar{V}_0&=K(\bx_0,\by_0,\bz_0)-2d(\by_0, \bz_0) + 2\Psi(\bz_0)+\frac{1}{48(L_0^2+L_f^2)\tau}\mathbb{E}\|\widehat{\nabla}f_0 -\nabla f(\bx_{0})\|^2\\
&\leq K(\bx_0,\by_0,\bz_0)-2d(\by_0, \bz_0) + 2\Psi(\bz_0)+\frac{\sigma^2}{48(L_0^2+L_f^2)}
\end{aligned}
\end{equation}
where the right-hand is a constant independent of $T$, we denote it as $C_0$.

Combining \eqref{V_0-V_T} with \eqref{V_T} and \eqref{V_0}, we have
\begin{equation} 
    \begin{aligned}
    &\frac{1}{T}\sum_{t=0}^{T-1}C\{\mathbb{E}\|\bx_t-\bz_{t}\|^2+\mathbb{E}\|\tau^{-1}(\bx_t-\bx_{t+1})\|^2+\mathbb{E}\|A\bx^*(\by_{t+1},\bz_t)-\bb\|^2+\mathbb{E}\|\nabla f(\bx_t)-\widehat{\nabla}f_t \|^2\}\\
    & \leq T^{-2/3}(C_0-\underline{f}+144(L_0^2+L_f^2)\sigma^2).
\end{aligned}
\end{equation}
Then, for index $s$ selected uniformly at random from $\{0,1,...,T-1\}$, we have
\begin{equation}\label{eq: aoe3}
\begin{aligned}
    &\mathbb{E}\|\bx_s-\bz_{s}\|^2=O(T^{-2/3}),\quad  \mathbb{E}\|\tau^{-1}(\bx_s-\bx_{s+1})\|^2=O(T^{-2/3}),\\
    &\mathbb{E}\|A\bx^*(\by_{s+1},\bz_s)-\bb\|^2=O(T^{-2/3}), \quad\mathbb{E}\|\nabla f(\bx_t)-\widehat{\nabla}f_t \|^2=O(T^{-2/3}).
\end{aligned}
\end{equation}
According to Algorithm \ref{alg: alm-storm}, we have
\[
\bx_{s+1} = \arg\min_{\bx} \left\{ \langle  G(\bx_s,\by_{s+1},\bz_s), \bx - \bx^s \rangle + \frac{1}{\tau} \|\bx - \bx_s\|^2 + \partial I_X(\bx) \right\}.
\]
By the definition of $\bx_{s+1}$, we have
\begin{equation}\label{eq: kwe3}
0 \in G(\bx_s,\by_{s+1},\bz_s) + \frac{2}{\tau}(\bx_{s+1} - \bx_s) + \partial I_X(\bx_{s+1}).
\end{equation}
We now set
\[
\bv = \nabla_\bx K(\bx_{s+1}, \by_{s+1},\bz_s) - G(\bx_s,\by_{s+1},\bz_s) - \frac{2}{\tau}(\bx_{s+1} - \bx_s) - \rho A^\top(A\bx_{s+1} - \bb) - \mu(\bx_{s+1} - \bz_s).
\]
Now, by using the definition of $K(\bx, \by, \bz)$ from \eqref{eq: k_def_main} and \eqref{eq: kwe3}, we obtain
\[
\bv \in \nabla f(\bx_{s+1}) + A^\top \by_{s+1} + \partial I_X(\bx_{s+1})
\]
We now derive the guarantees on the feasibility and the norm of $\bv$.

First, by triangle inequality, we have
\begin{align}
    \|A\bx_{s+1} - \bb\| &\leq \|A\bx^*(\by_{s+1}, \bz_s) - \bb\| + \|A\bx_{s+1} - A\bx_s \|+ \|A(\bx_{s} - \bx^*(\by_{s+1}, \bz_s))\|  \notag \\
    &\leq \|A\bx^*(\by_{s+1}, \bz_s) - \bb\|+\|A\|\|\bx_{s+1}-\bx_s\|+\frac{\|A\|}{\tau (\mu-L_f)}\|\bx_s-\bx_{s+1}\|+\frac{\|A\|}{\mu-L_f}\|\widehat{\nabla}f_s-\nabla f(\bx_s)\|  \notag \\
    &=O(T^{-1/3}),\label{eq: aoe4}
\end{align}
where in the second inequality, we use Lemma \ref{lem: x-x(y,z)}.

Then we have
\begin{align*}
    \|\bv\|&\leq \| \nabla_\bx K(\bx_{s+1},\by_{s+1},\bz_s)-\nabla_\bx K(\bx_{s},\by_{s+1},\bz_s)\|
    +\|\nabla_\bx K(\bx_{s},\by_{s+1},\bz_s)-G(\bx_{s},\by_{s+1},\bz_s)\|\\
    &\quad +\frac{2}{\tau}\|\bx_{s+1}-\bx_s\|+\rho \|A\| \|A\bx_{s+1}-\bb\|+\mu\|\bx_{s+1}-\bz_s\|\\
    &\leq \Big(L_K+\frac{2}{\tau}\Big)\|\bx_{s+1}-\bx_s\|+\|\nabla f(\bx_s)-\widehat{\nabla}f_s\|+\rho \|A\| \|A\bx_{s+1}-\bb\|+\mu(\|\bx_{s}-\bz_s\|+\|\bx_{s+1}-\bx_s\|)\\
    &=O(T^{-1/3}),
\end{align*}
where in first inequality, we introduce a term $\nabla_{\bx}K(\bx_s,\by_{s+1},\bz_s)$ and  then use triangle inequality. The second inequality used Lipschitzness of $K$, the definition of $G$, and the triangle inequality. The last step uses \eqref{eq: aoe3} and \eqref{eq: aoe4}.
\end{proof}

\begin{proof}[Proof of \Cref{rem: storm}]
    This is because for the proof for Section \ref{sec: rand_sec}, we only need to obtain a result similar to Lemma \ref{th: alm_storm_descent} when $\|\by\|<M_y$ (where the latter result will be shown in the same way as Proposition \ref{prop: y_bdd}).

    First, there is a difference in Lemma \ref{lemma: descent_alm_storm} that analyzes the progress of $K(\bx_t,\by_t,\bz_t)$. In particular, we now have
    \begin{align*}
        &\langle \nabla_{\bx} K(\bx_t, \by_{t+1}, \bz_t) - G(\bx_t, \by_{t+1}, \bz_t), \bx_{t+1} - \bx_t \rangle\\
        &=\langle \nabla f(\bx_t)+A^\top \by_{t+1}+A^\top(A\bx_t-\bb)-\widehat{\nabla}f_t -A_{\zeta_t^1}^\top \by_{t+1}+A_{\zeta_t^1}^\top(A_{\zeta_t^2}\bx_t-\bb_{\zeta_t^2}) , \bx_{t+1} - \bx_t    \rangle\\
        &\leq \frac{\tau}{2}\|(\nabla f(\bx_t)-\widehat{\nabla}f_t +(A^\top \by_{t+1}-A_{\zeta_t^1}^\top \by_{t+1})+(A^\top(A\bx_t-\bb)-A_{\zeta_t^1}^\top(A_{\zeta_t^2}\bx_t-\bb_{\zeta_t^2}))\|^2+\frac{1}{2\tau}\|\bx_{t+1}-\bx_t\|^2\\
        &\leq \frac{3\tau}{2}(\|\nabla f(\bx_t)-\widehat{\nabla}f_t \|^2+\|A^\top \by_{t+1}-A_{\zeta_t^1}^\top \by_{t+1}\|^2+\|A^\top(A\bx_t-\bb)-A_{\zeta_t^1}^\top(A_{\zeta_t^2}\bx_t-\bb_{\zeta_t^2})\|^2)+\frac{1}{2\tau}\|\bx_{t+1}-\bx_t\|^2
    \end{align*} 
The only difference is that we have more variance-type error terms. We use the same STORM technique to update the stochastic gradient $A_{\zeta_t^1}^\top \by_{t+1}$ and $A_{\zeta_t^1}^\top(A_{\zeta_t^2}\bx_t-\bb_{\zeta_t^2})$. Then, under \Cref{assumptions for stochasitc}, we will have a similar result that the first three terms could be bounded by similar recurrence relations to \ref{eq: variance_reduction2} and, as in \eqref{eq: snv4}, the order of $\tau$ will be $3$ since $\alpha \approx \tau^2$. 

Next, there is the second difference in Lemma \ref{lem: swo4} for $d(\by_t,\bz_t)$. In particular, we have
    \begin{align*} 
                &d(\by_{t+1}, \bz_{t+1})-d(\by_t, \bz_t)\\
                &\geq \langle \by_{t+1}-\by_{t}, A\bx^*(\by_{t+1}, \bz_t)-\bb \rangle +\frac{\mu}{2}\langle \bz_{t+1}-\bz_{t}, \bz_{t+1}+\bz_{t}-2\bx^*(\by_{t+1},\bz_{t+1}) \rangle\\
                &=\langle \eta(A_{\zeta_t}\bx_t-b_{\zeta_t})-\eta(A\bx_t-b), A\bx^*(\by_{t+1}, \bz_t)-\bb \rangle+\langle \eta(A\bx_t-b), A\bx^*(\by_{t+1}, \bz_t)-\bb \rangle \\
                &+\frac{\mu}{2}\langle \bz_{t+1}-\bz_{t}, \bz_{t+1}+\bz_{t}-2\bx^*(\by_{t+1},\bz_{t+1}) \rangle\\
                &\geq -\eta \|(A_{\zeta_t}\bx_t-b_{\zeta_t})-(A\bx_t-b)\|^2- \frac{\eta}{4}\|A\bx^*(\by_{t+1}, \bz_t)-\bb\|^2 \\
                &+\langle \eta(A\bx_t-b), A\bx^*(\by_{t+1}, \bz_t)-\bb \rangle
                +\frac{\mu}{2}\langle \bz_{t+1}-\bz_{t}, \bz_{t+1}+\bz_{t}-2\bx^*(\by_{t+1},\bz_{t+1}) \rangle
            \end{align*}
Note that the only difference between the above estimate and Lemma \ref{lem: swo4} are the first two (nonpositive) error terms on the right-hand side. We also use the STORM technique to update the stochastic gradient $A_{\zeta_t}\bx_t-b_{\zeta_t}$, then the first error term will be bounded with a similar recurrence relation to \eqref{eq: variance_reduction2}, and the error, as before, will be a constant term in the order of $\tau^3$. In addition the second error term will be canceled by the third term on the right-hand side of \eqref{eq: final descent_lemma_alm_storm}.

            Hence we will have an inequality for the change of $\mathbb{E}\bar{V}_t$ to $\mathbb{E} \bar{V}_{t+1}$, similar to \eqref{eq: final descent_lemma_alm_storm}. In particular, the main error term is of the order $\tau^3$. Then we will obtain the same $O(\varepsilon^{-3})$ complexity result by arguing the same way as in \Cref{th: storm_complexity}.
\end{proof}

\end{document}